\documentclass[a4paper,11pt,english]{amsart}
\usepackage[left=25mm,top=25mm,bottom=25mm,right=25mm]{geometry}
\usepackage{amsmath,amsfonts,amssymb,amsthm,dsfont,amscd}
\usepackage[utf8]{inputenc} 
\usepackage{cmap}           
\usepackage[pdfdisplaydoctitle=true,colorlinks=true,urlcolor=blue,citecolor=blue,linkcolor=blue,pdfstartview=FitH,pdfpagemode=UseNone]{hyperref}
\usepackage{hyperxmp}
\usepackage{babel}
\usepackage{array}
\usepackage{float}
\usepackage[all]{xy}
\usepackage{booktabs}
\usepackage{subcaption}
\usepackage{graphicx}
\newtheorem{thm}{Theorem}[section]
\newtheorem{cor}[thm]{Corollary}
\newtheorem{lem}[thm]{Lemma}

\theoremstyle{definition}
\newtheorem{defn}[thm]{Definition}
\theoremstyle{remark}
\newtheorem{rem}[thm]{Remark}
\newtheorem{exa}[thm]{Example}
\numberwithin{equation}{section}
\newcommand{\CC}{\mathbb{C}}                
\newcommand{\NN}{\mathbb{N}}                
\newcommand{\RR}{\mathbb{R}}                
\newcommand{\ZZ}{\mathbb{Z}}                

\newcommand{\GL}{\mathrm{GL}}               
\newcommand{\SL}{\mathrm{SL}}               
\newcommand{\OO}{\mathrm{O}}                
\newcommand{\SO}{\mathrm{SO}}               

\newcommand{\VV}{\mathbb{V}}                
\newcommand{\WW}{\mathbb{W}}                
\newcommand{\Ela}{\mathbb{E}\mathrm{la}}    
\newcommand{\Gel}{\mathbb{G}\mathrm{el}}    
\newcommand{\Piez}{\mathbb{P}\mathrm{iez}}  

\newcommand{\TT}{\mathbb{T}}                
\newcommand{\Sym}{\mathbb{S}}               
\newcommand{\HT}[1]{\mathbb{H}^{#1}(\mathbb{R}^{2})}        
\newcommand{\HP}[1]{\mathcal{H}_{#1}(\mathbb{R}^{2})}       
\newcommand{\Pn}[1]{\mathcal{P}_{#1}(\mathbb{R}^{2})}       

\newcommand{\cov}{\mathbf{Cov}}
\newcommand{\inv}{\mathbf{Inv}}

\newcommand{\bm}{\mathbf{m}}

\newcommand{\ee}{\pmb{e}}                   
\newcommand{\vv}{\pmb{v}}                   
\newcommand{\ww}{\pmb{w}}                   
\newcommand{\xx}{\pmb{x}}                   
\newcommand{\EE}{\pmb{E}}                   
\newcommand{\zb}{\overline{z}}              

\newcommand{\qq}{\mathrm{q}}                
\newcommand{\rp}{\mathrm{p}}                
\newcommand{\rh}{\mathrm{h}}                

\newcommand{\id}{\mathbf{1}}
\newcommand{\ba}{\mathbf{a}}

\newcommand{\bc}{\mathbf{c}}
\newcommand{\bd}{\mathbf{d}}
\newcommand{\be}{\mathbf{e}}

\newcommand{\bh}{\mathbf{h}}

\newcommand{\bepsilon}{\pmb{\epsilon}}
\newcommand{\bsigma}{\pmb{\sigma}}

\newcommand{\lc}{\pmb{\varepsilon}}         
\newcommand{\bC}{\mathbf{C}}                
\newcommand{\bH}{\mathbf{H}}                
\newcommand{\bK}{\mathbf{K}}                
\newcommand{\bP}{\mathbf{P}}                
\newcommand{\bPi}{\pmb{\Pi}}                
\newcommand{\pieze}{\mathbf{e}}             
\newcommand{\bS}{\mathbf{S}}                
\newcommand{\bT}{\mathbf{T}}                
\newcommand{\bU}{\mathbf{U}}                
\newcommand{\bV}{\mathbf{V}}                
\newcommand{\bW}{\mathbf{W}}                

\DeclareMathOperator{\tr}{tr}
\DeclareMathOperator{\rg}{Rank}

\renewcommand\Re{\operatorname{\mathfrak{Re}}}
\renewcommand\Im{\operatorname{\mathfrak{Im}}}
\renewcommand\vec{\pmb}

\DeclareMathOperator{\1dot}{\cdot}

\DeclareMathOperator{\3dots}{\resizebox{0.5ex}{!}{\textbf{\vdots}}}
\newcommand{\rdots}[1]{\overset{(#1)}{\cdot}}

\newcommand{\norm}[1]{\left\Vert#1\right\Vert}
\newcommand{\abs}[1]{\left\vert#1\right\vert}
\newcommand{\set}[1]{\left\{#1\right\}}

\newcounter{mtligne}

\hypersetup{
  pdfauthor={B. Desmorat, M. Olive, N. Auffray, R. Desmorat and B. Kolev}
  pdftitle={Computation of minimal covariants bases for 2D coupled constitutive laws}, 
  pdfsubject={MSC2010: 74B05; 15A72; 74-04}, 
  pdfkeywords={Constitutive tensors; Tensor invariants; Integrity basis; Higher-order tensors; Polynomial covariants; Harmonic decomposition}, 
  pdflang=en, 
  }

\begin{document}

\title[Computation of minimal covariants bases]{Computation of minimal covariants bases \protect\\ for 2D coupled constitutive laws}%

\author{B. Desmorat}
\address[Boris Desmorat]{Sorbonne Université, CNRS, Institut Jean Le Rond d’Alembert, UMR 7190, 75005 Paris, France}
\email{boris.desmorat@sorbonne-universite.fr}
\author{M. Olive}
\address[Marc Olive]{Université Paris-Saclay, ENS Paris-Saclay, CentraleSupélec, CNRS, LMPS - Laboratoire de Mécanique Paris-Saclay, 91190, Gif-sur-Yvette, France}
\email{marc.olive@math.cnrs.fr}
\author{N. Auffray}
\address[Nicolas Auffray]{Sorbonne Université, CNRS, Institut Jean Le Rond d’Alembert, UMR 7190, 75005 Paris, France}
\email{nicolas.auffray@sorbonne-universite.fr}
%
\author{R. Desmorat}
\address[Rodrigue Desmorat]{Université Paris-Saclay, ENS Paris-Saclay, CentraleSupélec, CNRS, LMPS - Laboratoire de Mécanique Paris-Saclay, 91190, Gif-sur-Yvette, France}
\email{rodrigue.desmorat@ens-paris-saclay.fr}
\author{B. Kolev}
\address[Boris Kolev]{Université Paris-Saclay, ENS Paris-Saclay, CentraleSupélec, CNRS, LMPS - Laboratoire de Mécanique Paris-Saclay, 91190, Gif-sur-Yvette, France}
\email{boris.kolev@math.cnrs.fr}


\date{April 24, 2023}%

\thanks{Rodrigue Desmorat, Boris Kolev and Marc Olive were partially supported by CNRS Projet 80--Prime GAMM (Géométrie algébrique complexe/réelle et mécanique des matériaux).}


\begin{abstract}
  We produce minimal integrity bases for both isotropic and hemitropic invariant algebras (and more generally covariant algebras) of most common bidimensional constitutive tensors and -- possibly coupled -- laws, including piezoelectricity law, photoelasticity, Eshelby and elasticity tensors, complex viscoelasticity tensor, Hill elasto-plasticity, and (totally symmetric) fabric tensors up to twelfth-order. The concept of covariant, which extends that of invariant is explained and motivated. It appears to be much more useful for applications. All the tools required to obtain these results are explained in detail and a cleaning algorithm is formulated to achieve minimality in the isotropic case. The invariants and covariants are first expressed in complex forms and then in tensorial forms, thanks to explicit translation formulas which are provided. The proposed approach also applies to any $n$-uplet of bidimensional constitutive tensors.
\end{abstract}

\maketitle


\begin{scriptsize}
  \setcounter{tocdepth}{2}
  \tableofcontents
\end{scriptsize}

\section{Introduction}
\label{sec:intro}

The modern assertion that \emph{a physics is a group}~\cite{Sou1996} has important consequences regarding the \emph{invariance} of the physical quantities. Hence, these questions have to be formulated and studied within the Mathematical framework of \emph{groups} and \emph{Representation Theory}~\cite{Ste1994}. In $d$-dimensional solid mechanics, the invariance properties of constitutive laws are {formulated} with respect to the full orthogonal group $\OO(d)$. Quantities that are invariant with respect to the special orthogonal group $\SO(d)$ are called \emph{hemitropic invariants}, while those that are invariant with respect to the full group $\OO(d)$ are called \emph{isotropic invariants}.

In continuum mechanics, constitutive equations, linear or non-linear, are naturally described using tensors~\cite{Gur1973,TN2004}. Yet classical in linear theories, \emph{constitutive tensors} are also encountered in non-linear mechanics of materials such as, for instance, anisotropic elasto-plasticity (\textit{e.g.} \emph{Hill yield tensor}~\cite{Hil1948}), in continuum damage mechanics (for a description of damage anisotropy see~\cite{Cha1978,Cha1984,Ona1984,LC1985,Ju1989,LD2005}) and in nonlinear piezoelectricity/magnetism (\emph{e.g.} the magnetostriction morphic tensor~\cite{DTdLac1993,Hub2019}).

Orders of constitutive tensors, which are usually lower than four in classical linear elasticity and piezoelectricity, can however reach six in generalized continuum theories, such as strain gradient or micromorphic continua~\cite{AQH2013,AHQ2019,Pol2018,Ber2016}. They can be odd orders, forbidding then the definition of \emph{spectral invariants}
from Kelvin's matrix representation\footnote{Moreover, the spectral invariants of the $\RR^{d(d+1)/2}$--Kelvin matrix of an elasticity tensor do not characterize the geometry of its orbits in $\RR^{d}$ (with $d=2$ in 2D and $d=3$ in 3D). They are algebraic invariants of $\OO(d(d+1)/2)$. Any function invariant with respect to $\OO(d(d+1)/2)$ is also invariant with respect to $\OO(d)$, but the converse is false.}  \cite{Bet1987,CM1990, KO09,BBS2008}. They can even rise up beyond order six when fabric tensors are involved~\cite{Oda1982,Kan1984,JT1984,Kac1993,TNS2001,RKM2009,DK2016,CW2010,CDC2018}. A sound analysis of these tensors, of their invariants and their symmetry classes, gives precious information and modelling tools for the physics that can be described by them.

Tensors having symmetries can be described using their invariants and covariants \cite{OKDD2018,ODK2021,OKDD2022}.
Note however that the invariants must be chosen inside a given class: \emph{polynomial}, \emph{rational} or \emph{algebraic} (such as the eigenvalues of a matrix). Besides, it is important to describe the tensors properties by a finite number of such invariants. To do so, mathematical definitions are required. For instance, one could be interested to describe the algebra of polynomial invariants using a minimal set of generators of this algebra (usually called an \emph{integrity basis}), or to find a finite \emph{separating set} of invariants ---in a given class--- which can be used to factor any invariant function of these tensors (set usually called a \emph{functional basis}),
with the property that an \textit{integrity basis} is a \textit{functional basis} \cite{WP1964}.
If the later set is probably more pertinent in practical applications and of lower cardinality \cite{MZC2019,DAD2019,OD2021}, there is however no algorithm to obtain such a set. On the contrary, there exist general ---but complex--- algorithms to compute a minimal integrity basis \cite{Wey1997,DK2015,Oli2016}.

The complexity increases with the order of the tensor. In 3D, this complexity is already high, while it remains reasonable in 2D.
For example in 3D, a minimal integrity basis for the elasticity tensor is constituted of 294 invariants (for both $\SO(3)$ and $\OO(3)$~\cite{OKA2017,OKDD2022}), while in 2D it is constituted of 5 invariants for $\OO(2)$ and of 6 invariants for $\SO(2)$ \cite{Via1997}. This huge difference is roughly due to the following facts:
\begin{itemize}
  \item The harmonic decomposition of the space of elasticity tensors contains more harmonic components in 3D than in 2D~\cite{Bac1970,BOR1996}.
  \item The dimension of the space of harmonic tensors of order $n$ is equal to $2n+1$ in 3D  (like the space of spherical harmonics of degree $n$) and is equal to 2 in 2D  (for $n\geq 1$). The later is the dimension of the space generated by $\cos n\theta$ and $\sin n\theta$ in the Fourier decomposition.
\end{itemize}
In 3D, the first step consists in determining an integrity basis for harmonic tensors. Such results are available in the literature but only up to order 5 (\emph{i.e.}, up to degree 10 for binary forms \cite{LO2017}), due to the exponential growth of the computations with the order $n$ of the tensor. In the present contribution on the 2D case, we formulate a general method and provide an algorithm to compute a minimal integrity basis for isotropic/hemitropic invariants and covariants of a tensor or a family of tensors. In particular, we provide a minimal integrity bases for 2D totally symmetric tensors ---such as fabric tensors \cite{Kan1984}--- up to order 12, which include the cases of 2D harmonic tensors of orders 2, 4, 6, 8, 10 and 12.

The structure of 2D constitutive tensors spaces and the determination of their symmetry classes has already been considered in a previous contribution~\cite{AKO2016}. Among the remaining open problems is the formulation of a systematic procedure to determine a minimum integrity basis for either isotropic or hemitropic polynomial invariant functions on a given tensor space. Such a set is useful for two reasons; on the one hand, every polynomial invariant can be recast as a polynomial function of these generating invariants and, on the other hand, it \emph{separates the orbits} (\textit{i.e.} sets of tensors of the same kind), which means that at least one of the generating invariants takes different values, when evaluated on two tensors which are not of the same kind. This last property makes it possible to decide whether, or not, two tensors describe the same material up to an orthogonal transformation.

Let us first present a quick overview of this question in the mechanical community. The story started in 1946 with Weyl's pioneering book~\cite{Wey1997}, from which was extracted the methods and vocabulary still used in continuum mechanics nowadays. It took, however, a few decades before some basic results, such as the determination of a minimal \emph{integrity basis} for a $n$-uplet of three-dimensional second-order tensors and vectors \cite{SR1958/1959,SR1962,Smi1965,Smi1971,KS1974,Riv1997}, were published. In these approaches, a (non necessarily minimal) generating set is obtained first, using for instance \emph{Weyl's polarization theorem}. In a second step, a reduction procedure is achieved, using polynomial relationships (\emph{syzygies}) between the polynomial invariants~\cite{Riv1955,Smi1994}, to eventually obtain a minimal integrity basis. For second-order tensors, these syzygies are essentially derived from Cayley-Hamilton's theorem, and are thus useless for higher order tensors. In that case, only partial results have been obtained~\cite{SR1971,SB1997,Bet1987,Zhe1994,BH1995,ZB1995b}, until recently.

In this approach~\cite{Oli2014a}, the problem of higher order tensors in 3D is recast in the realm of \emph{binary forms}, which are complex homogeneous polynomials in two variables. Using a powerful tool called the \emph{Cartan map}~\cite{Car1981,Bac1970,DAD2019}, an integrity basis for the binary form of degree $2n$ can then be translated into an integrity basis for the harmonic tensor of degree $n$ \cite{DAD2019,OKDD2022}. The gain is that \emph{invariant theory of binary forms} (also know as \emph{Classical Invariant Theory}) is an area of mathematics which has been extensively studied by a wide number of prestigious mathematicians such as Gordan or Hilbert and in which an impressive number of results has already been produced. Combining these results with the use of the \emph{harmonic decomposition}~\cite{Bac1970,Spe1970}, integrity bases for the third order totally symmetric tensor and for the fourth-order elasticity tensor have been obtained recently~\cite{OA2014b,Oli2014a,OKA2017}. In this approach, \emph{Gordan's algorithm} for binary forms~\cite{Gor1868,Gor1873,Gor1875,Gor1987} is used first to generate a (non necessarily minimal) integrity basis and then, a reduction process using modern computational means is achieved to obtain minimality~\cite{Oli2017a}.

In 2D, the situation is much simpler and integrity bases are known in specific situations. For instance, regarding its practical importance for plate theory and laminated structures, the bidimensional (plane) elasticity tensor has been widely studied \cite{Ver1982,BOR1996,HZ1996,Via1997,VV2001,Van2005,dSV2013,FV2014,DD2015}. In~\cite{FV2014}, a comparative review of the literature on $\OO(2)$ and $\SO(2)$ invariants of the elasticity tensor is provided. Recently, there has been an attempt (unfortunately with mistakes) to determine an integrity basis for fourth-order tensors of Eshelby type -- \emph{i.e.} photoelasticity type~\cite{MZC2019} -- and partial results for the piezoelectricity tensor are already known~\cite{Van2007}. Nevertheless, analysing the literature, it appears that a general, effective and systematic method to compute a minimal integrity basis for coupled constitutive laws and more generally for tensors of any order is still lacking, and that almost no results are known for the covariant integrity bases \cite{OKDD2022}. We moreover point out that the literature results on 2D invariants are rarely expressed using tensorial expressions.

Concerning practical applications, an integrity basis is required to formulate invariant relations that characterize intrinsic properties of a constitutive law, such as the belonging to a symmetry class \cite{Ver1982,Via1997,AKP2013,AR2016,OKDD2022}, the special ($r_0$) orthotropy \cite{VV2001,Van2005} or the existence of a pentamode~\cite{MC1995,KBS2012,DA2019}.
Such kind of relations are interesting for optimal design algorithms since they allow to formulate frame-independent constraints on the sought material \cite{VVA2013,RIBD2018}. Invariants for higher order tensors will naturally find applications to extend this approach to the generalized (Mindlin) elasticity models used to describe the effective behaviour of architectured materials~\cite{AS2018,PSM2018,RA2019}.

The goal of the present contribution is to propose a general and effective method to compute a minimal integrity basis for any $\OO(2)$ or $\SO(2)$ representation and to apply it to continuum mechanics constitutive laws. To this end, we follow the path traced by Vianello~\cite{Via1997} for the bidimensional elasticity tensor and formulate the problem within the framework of \emph{Invariant Theory}~\cite{Wey1997,Spr1980,Olv1999,KP2000,DK2015}. This allows us to produce a minimal integrity basis for 2D higher order tensors with or without any particular index symmetry, under both groups $\OO(2)$ and $\SO(2)$. Minimality of the integrity bases is obtained using a Computer Algebra System and an algorithm which is explained in details.
Integrity bases are first formulated using complex variables (as in \cite{Via1997}). In a more mechanistic way, we provide an original process allowing to  translate all these expressions into tensorial ones.

This paper intents to be as self-contained as possible, and many illustrating examples are provided along the lines.

The outline of the paper is the following. In~\autoref{sec:tensors}, we recall basic operations on tensors, some of which are well-known and others are new. The link between totaly symmetric tensors and homogeneous polynomials is explained. Main concepts from the theory of linear representations of the orthogonal groups $\OO(2)$ and $\SO(2)$ are presented in~\autoref{sec:representation-theory}. In~\autoref{sec:invariant-theory}, we introduce basic notions of \emph{Invariant Theory}, such as polynomial invariants and covariants. The main results of the paper are given in~\autoref{sec:integrity-bases}, where integrity bases for bidimensional tensors of any order (and more generally any linear representation of $\OO(2)$ and $\SO(2)$) are derived. Hemitropic (\textit{i.e.} for $\SO(2)$) integrity bases produced are already minimal but isotropic ones (\textit{i.e.} for $\OO(2)$) are not. A \emph{cleaning algorithm} to achieve this task is formulated in~\autoref{sec:algorithm}. The computed integrity bases are written in terms of complex monomials. It is more useful, in mechanics, to express them using tensorial operations. These translation rules are formulated in~\autoref{sec:monomials-versus-tensors-invariants}. In~\autoref{sec:applications}, we illustrate the power of our methods by providing minimal integrity bases for an extensive list of constitutive tensors (up to twelfth-order) and coupled laws in mechanics of materials,
including Eshelby/photoelasticity tensors, linear viscoelasticity, Hill elasto-plasticity, linear piezoelectricity and  fabric tensors. Besides, four appendices are provided to detail and deepen some technical points.

\section{Tensorial operations}
\label{sec:tensors}

This paper is about tensors polynomial invariants. In this section we recall basic operations on tensors, some of them are well-known, others are less. We shall denote by $\TT^{n}(\RR^{2})=\otimes^{n} \RR^2$, the vector space of 2D tensors of order $n$. Using the Euclidean structure of $\RR^{2}$, we will not make any difference between covariant, contravariant or mixed tensors. We will encounter tensors with various index symmetries, among them tensors which are totally symmetric. The subspace of $\TT^{n}(\RR^{2})$ of totally symmetric tensors will be denoted by $\Sym^{n}(\RR^{2})$. Given $\bT \in \TT^{n}(\RR^{2})$, the total symmetrization (over all subscripts) of $\bT$, denoted by $\bT^{s}$ is a projector from $\TT^{n}(\RR^{2})$ onto $\Sym^{n}(\RR^{2})$. The following tensorial operations will be used (see also \cite{DAD2019,OKDD2022}).
\begin{enumerate}
  \item The \emph{symmetric tensor product} between two tensors $\bS_{1}\in\Sym^{n_{1}}(\RR^{2})$ and $\bS_{2}\in\Sym^{n_{2}}(\RR^{2})$, defined as
        \begin{equation}\label{eq:odot}
          \bS_{1}\odot\bS_{2} : =  (\bS_{1}\otimes\bS_{2})^{s}\in\Sym^{n_{1} + n_{2}},
        \end{equation}

  \item The \emph{$r$-contraction} of two tensors $\bT_{1}\in\TT^{n_{1}}(\RR^{2})$ and $\bT_{2}\in\TT^{n_{2}}(\RR^{2})$, defined in any orthonormal basis as
        \begin{equation}\label{eq:rcontract}
          (\bT_{1}\rdots{r}\bT_{2})_{i_{1}\cdots i_{n_{1}-r}j_{r + 1}\cdots j_{n_{2}}} := T^{1}_{i_{1}\cdots i_{n_{1}-r}k_{1} \cdots k_{r}}T^{2}_{k_{1} \cdots k_{r}j_{r + 1}\cdots j_{n_{2}}},
        \end{equation}
        which is a tensor of order $n_{1} + n_{2} - 2r$.

  \item The \emph{skew-symmetric contraction} between two totally symmetric tensors $\bS_{1}\in\Sym^{n_{1}}(\RR^{2})$ and $\bS_{2}\in\Sym^{n_{2}}(\RR^{2})$ is defined as
        \begin{equation}\label{eq:cross}
          (\bS_{1}\times \bS_{2}) : = -(\bS_{1} \cdot \lc \cdot \bS_{2})^{s}\in \Sym^{n_{1} + n_{2} - 2}(\RR^{2}),
        \end{equation}
        where $\lc$ is the 2D Levi–Civita tensor. In any orthonormal basis  $(\ee_{1}, \ee_{2})$, we get
        $\varepsilon_{ij}=\det(\ee_{i}, \ee_{j})$ and
        \begin{equation*}
          (\bS_{1} \times \bS_{2})_{i_{1}\ldots i_{n_{1} + n_{2}-2}} = - \left(\varepsilon_{jk} S^{1}_{ji_{1}\ldots i_{n_{1}}}S^{2}_{ki_{n_{1} + 1} \ldots i_{n_{1} + n_{2}-2}} \right)^{s}.
        \end{equation*}
\end{enumerate}

There is a well-known correspondence $\phi: \Sym^{n}(\RR^{2})\to \Pn{n}$ between totally symmetric tensors of order $n$ on $\RR^{2}$ and homogeneous polynomials of degree $n$ in two variables. Given $\bT \in \TT^{n}(\RR^{2})$ we associate to it the polynomial $\rp = \phi(\bT)$ in $\Pn{n}$, where
\begin{equation}\label{eq:symmetric-tensor-polynomial}
  \rp(\xx) = \bT(\xx, \dotsc, \xx).
\end{equation}
In components, this writes
\begin{equation*}
  \rp(\xx) = T_{i_{1} i_{2} \dotsc i_{n}} x_{i_{1}}x_{i_{2}} \dotsc x_{i_{n}}, \quad \text{where} \quad \xx=(x_{1}, x_{2})=(x,y).
\end{equation*}
Note that $\phi(\bT) = \phi(\bT^{s})$ and that when restricted to $\Sym^{n}(\RR^{2})$, this correspondence $\bS \mapsto \rp$ is a bijection, the inverse operation being given by the \emph{polarization} of $\rp$ (see~\cite{Wey1997,Bac1970,Bae1993,OKDD2022}). Making use of this correspondence, the three tensorial operations defined above are recast into polynomial operations as follows. Let $\rp_{i} := \phi(\bS_{i})$ for $i = 1,2$, be the homogeneous polynomials associated with $\bS_{i}\in \Sym^{n_{i}}(\RR^2)$. Then,
\begin{enumerate}
  \item the symmetric tensor product $\bS_{1}\odot\bS_{2}$ translates as $\rp_{1} \rp_{2}$;
  \item the symmetrized $r$-contraction $(\bS_{1}\rdots{r}\bS_{2})^{s}$ translates as
        \begin{equation}\label{eq:rcontractp}
          \left\{\rp_{1},\rp_{2}\right\}_{r} := \frac{(n_{1}-r)! (n_{2}-r)!}{n_{1}! n_{2}!} \sum_{k=0}^{r} \binom{r}{k} \frac{\partial^{r} \rp_{1}}{\partial x^{k} \partial y^{r-k}} \frac{\partial^{r} \rp_{2}}{\partial x^{k} \partial y^{r-k}};
        \end{equation}
  \item The skew-symmetric contraction $\bS_{1} \times \bS_{2}$ translates as
        \begin{equation}\label{eq:crossp}
          \left[\rp_{1},\rp_{2}\right] :=- \frac{1}{n_{1}n_{2}} \det (\nabla\rp_{1},\nabla\rp_{2}),
        \end{equation}
        where $\nabla$ denotes the gradient.
\end{enumerate}

\section{Real linear representations of 2D orthogonal groups}
\label{sec:representation-theory}

The full \emph{orthogonal group} in dimension $2$, denoted by $\OO(2)$, is defined as the set of linear isometries of the canonical scalar product on $\RR^{2}$. In the canonical basis $(\ee_{1},\ee_{2})$, this group is represented by the two-by-two matrices $g$ which satisfies $g^{t}g = I$. In particular, we have $\det g = \pm 1$. The subset of matrices $g$ such that $\det g = 1$ is a subgroup of $\OO(2)$, denoted by $\SO(2)$, which is the rotation group of the Euclidean space $\RR^{2}$, each rotation being represented by the matrix
\begin{equation}\label{eq:rtheta}
  r_{\theta} =  \begin{pmatrix}
    \cos \theta & -\sin \theta \\
    \sin \theta & \cos \theta  \\
  \end{pmatrix}.
\end{equation}
The full orthogonal group is obtained from $\SO(2)$ by adding the reflection with respect to the horizontal axis
\begin{equation}\label{eq:sigma}
  \sigma: =
  \begin{pmatrix}
    1 & 0  \\
    0 & -1 \\
  \end{pmatrix}.
\end{equation}
Besides, each element of $\OO(2)$ can be written, either as $r_{\theta}$ (if $\det g = 1$) or $\sigma r_{\theta}$ (if $\det g = -1$), and we have the relations
\begin{equation*}
  \sigma^{2} = id, \qquad \sigma r_{\theta}  = r_{-\theta}\sigma ,
\end{equation*}
where $\sigma r_{\theta} = r_{-\theta}\sigma$ is the reflection with respect to the axis
\begin{equation*}
  r_{-\theta/2}(\ee_{1})=\cos\left(\frac{\theta}{2}\right)\ee_{1}-\sin\left(\frac{\theta}{2}\right)\ee_{2}.
\end{equation*}

Next, we recall a few basic concepts in \emph{representation theory of groups}. More details can be found, for instance, in \cite{Ste1994}.

\begin{defn}\label{def:linear-representation}
  A \emph{linear representation} $(\VV,\rho)$ of a group $G$ on a vector space $\VV$ is a linear action of $G$ on $\VV$. More precisely, it is given by a mapping
  \begin{equation*}
    \rho: G \to \GL(\VV),
  \end{equation*}
  where $\GL(\VV)$ is the group of invertible linear mappings on $\VV$ and such that $\rho(g_{1}g_{2}) = \rho(g_{1})\rho(g_{2})$, for all $g_{1}, g_{2} \in G$.
\end{defn}

Linear representations of $G=\SO(2)$ and $G=\OO(2)$ play a fundamental role in 2D solid mechanics. They arise, for instance, when $\VV = \Ela$ is the space of fourth-order plane elasticity tensors, or when $\VV=\Piez$ is the space of third-order bidimensional piezoelectric tensors.

A linear representation is by definition linear in $\vv \in \VV$ and thus $\rho(g)$ is represented by a matrix $[\rho(g)]$ once a basis of $\VV$ is fixed. A basic example is provided by the standard representation of $\OO(2)$ on $\VV = \TT^{n}(\RR^{2})$, the vector space of $n$-th order tensors of dimension $2^{n}$. In the canonical basis of $\RR^{2}$, $(\ee_{1},\ee_{2})$, the tensor $\rho(g)\bT$ has for components
\begin{equation}\label{eq:n-order-tensorial-representation}
  (\rho(g)\bT)_{i_{1} \dots i_{n}} = \sum_{j_{1}, \dotsc , j_{n}} g_{i_{1} j_{1}} \dotsm g_{i_{n} j_{n}} T_{j_{1} \dots j_{n}}.
\end{equation}
By the way, a natural basis for $\TT^{n}(\RR^{2})$ is provided by $(\be_{i_{1} \dotsb i_{n}})$, where
\begin{equation*}
  \be_{i_{1} \dotsb i_{n}} := \ee_{i_{1}} \otimes \dotsb \otimes \ee_{i_{n}},
\end{equation*}
and the \emph{lexicographic order} has been adopted on multi-index $(i_{1}, \dotsc , i_{n})$. Thus, the corresponding matrix representation in this basis writes
\begin{equation*}
  [\rho(g)\bT] = [\rho(g)][\bT],
\end{equation*}
where, introducing the multi-index $I=(i_{1}, \dotsc , i_{n})$, $J=(j_{1}, \dotsc , j_{n})$,
\begin{equation*}
  [\rho(g)]_{IJ} = g_{i_{1} j_{1}} \dotsm g_{i_{n} j_{n}}.
\end{equation*}
Three other examples are provided in~\autoref{sec:explicit-harmonic-decomposition}.
\begin{defn}
  A representation is \emph{irreducible} if there is no stable subspace under $G$ other than $\set{0}$ and $\VV$.
\end{defn}

\begin{defn}
  Two linear representations $(\VV_{1},\rho_{1})$ and $(\VV_{2},\rho_{2})$ of the same group $G$ are \emph{equivalent} if there exists a linear isomorphism $\varphi$ from $\VV_{1}$ to $\VV_{2}$ such that
  \begin{equation*}
    \varphi(\rho_{1}(g)\vv_{1}) = \rho_{2}(g)\varphi(\vv_{1}),
  \end{equation*}
  for all $\vv_{1}\in  \VV_{1}$ and $g\in G$.
\end{defn}

Real irreducible representations of 2D orthogonal groups are well-known (see for instance~\cite{GSS1988}). Each \emph{real irreducible} representation of $\SO(2)$ is either equivalent to the trivial representation on $\RR$, denoted by $\rho_{0}$ and defined by
\begin{equation*}
  \rho_{0}(g)\lambda = \lambda,
\end{equation*}
for all $g \in \SO(2)$ and all $\lambda \in \RR$, or to the two-dimensional representation on $\RR^{2}$ given by
\begin{equation*}
  r_{\theta}\in \SO(2)\mapsto
  \rho_{n}(r_{\theta}) =
  \begin{pmatrix}
    \cos n\theta & -\sin n\theta \\
    \sin n\theta & \cos n\theta
  \end{pmatrix}
  \in \GL(\RR^{2}),
\end{equation*}
and indexed by the integer $n\geq 1$.

Each \emph{real irreducible} representations of $\OO(2)$ is either equivalent to the trivial representation on $\RR$, denoted by $\rho_{0}$, the \emph{sign} representation on $\RR$, denoted by $\rho_{-1}$ and defined by
\begin{equation*}
  \rho_{-1}(g)\xi = (\det g)\xi,
\end{equation*}
for all $g \in \OO(2)$ and all $\xi \in \RR$ ($\xi$ is sometimes called a pseudo-scalar), or to one of the following representations on $\RR^{2}$ given by
\begin{equation*}
  \rho_{n}(r_{\theta}) =
  \begin{pmatrix}
    \cos n\theta & -\sin n\theta \\
    \sin n\theta & \cos n\theta
  \end{pmatrix},
  \quad
  \rho_{n}(\sigma r_{\theta}) =
  \begin{pmatrix}
    \cos n\theta & \sin n\theta   \\
    \sin n\theta & - \cos n\theta
  \end{pmatrix},
\end{equation*}
and indexed by the integer $n\geq 1$.

\begin{rem}
  Besides $\rho_{n}$, one can build a new irreducible representation, called the \emph{twisted} representation (associated in 3D with pseudo-tensors). It is defined by
  \begin{equation*}
    \hat{\rho}_{n}(g) = \det(g)\rho_{n}(g), \qquad g \in \OO(2).
  \end{equation*}
  However, and contrary to what happens in 3D, this new representation $\hat{\rho}_{n}$ is equivalent to $\rho_{n}$ for every $n \ge 1$. Indeed, if we set
  \begin{equation*}
    \varphi : =  r_{\pi/2} : \RR^{2} \to \RR^{2},
  \end{equation*}
  one can check that $\varphi$ is an equivariant isomorphism:
  \begin{equation*}
    \varphi \circ \rho_{n}(r_{\theta}) = \hat{\rho}_{n}(r_{\theta}) \circ \varphi, \quad \text{and} \quad \varphi \circ \rho_{n}(\sigma r_{\theta}) = \hat{\rho}_{n}(\sigma r_{\theta}) \circ \varphi .
  \end{equation*}
  For $n = 0$, we have $\hat{\rho}_{0} = \rho_{-1}$, which is not equivalent to $\rho_{0}$. In other words, there are neither pseudo-vectors and nor pseudo-tensors in 2D but there exists pseudo-scalars.
\end{rem}

There are two models, useful in practice, for $2$-dimensional irreducible representations of the orthogonal groups:
\begin{enumerate}
  \item The spaces $\HP{n}$ of \emph{homogeneous harmonic polynomials} (polynomials with vanishing Laplacian) in two variables $x$, $y$ of degree $n \ge 1$,
  \item The spaces $\HT{n}$ of $n$th-order \emph{harmonic tensors} (totally symmetric tensors with vanishing traces).
\end{enumerate}
And, to complete these alternative models for $n = 0$ and $n = -1$, we set
\begin{equation*}
  \HT{0} = \HP{0} = \RR, \quad \text{with the trivial representation},
\end{equation*}
and
\begin{equation*}
  \HT{-1} = \HP{-1} = \RR, \quad \text{with the sign representation}.
\end{equation*}

Any linear representation $(\VV,\rho)$ of $\SO(2)$ or $\OO(2)$ can be decomposed into a direct sum of irreducible representations. This is known as the \emph{harmonic decomposition} of $\VV$ and means that
\begin{equation}\label{eq:harmonic-decomposition}
  \VV \simeq \HT{n_{1}}\oplus \dotsb \oplus \HT{n_{p}},
\end{equation}
where $n_{i} \in \set{-1,0,1,2, \dotsc}$ and where multiplicities are allowed. An explicit method to achieve such a decomposition for any representation $\VV$ of $\SO(2)$ or $\OO(2)$, based on the infinitesimal action of $\SO(2)$, is described in~\autoref{sec:explicit-harmonic-decomposition}. It extends, somehow, a method used in~\cite{Ver1982,Via1997,Van2005,Van2007} for bidimensional elasticity (see also \cite{Spe1970,Bac1970} and~\cite{Cow1989,Bae1993} for 3D elasticity using different approaches). The harmonic decomposition of the space of totally symmetric tensors is handled in~\autoref{sec:Sym-harmonic-decomposition}.

\begin{exa}\label{ex:dec-harm-4-order}
  The harmonic decomposition of fourth-order tensors under $\OO(2)$ are given below, depending on their index symmetries, as described in~\cite{AKO2016}, and where we have set $\bT=(T_{ijkl})$.
  \begin{align*}
     & T_{ijkl}                     &                                                        & \text{No index symmetry}      &  & 3\HT{-1} \oplus 3\HT{0}\oplus 4\HT{2}\oplus  \HT{4} \\
     & T_{(ij)kl}                   &                                                        & \text{One minor symmetry}
     &                              & 2\HT{-1}  \oplus 2\HT{0} \oplus 3\HT{2}\oplus \HT{4}                                                                                            \\
     & T_{ij | kl}                  &                                                        & \text{Major symmetry}
     &                              & \HT{-1} \oplus 3\HT{0} \oplus 2\HT{2} \oplus \HT{4}                                                                                             \\
     & T_{(ij)(kl)}                 &                                                        & \text{Minor symmetries}
     &                              & \HT{-1}  \oplus 2\HT{0} \oplus 2\HT{2}  \oplus  \HT{4}                                                                                          \\
     & T_{ijkl}= T_{jilk}= T_{klij} &                                                        & \text{Normal Klein sym.}
     &                              & 3\HT{0} \oplus \HT{2}\oplus \HT{4}                                                                                                              \\
     & T_{(ijk)l}                   &                                                        & \text{Tot. sym. over 3 index}
     &                              & \HT{-1}  \oplus \HT{0} \oplus 2\HT{2} \oplus \HT{4}                                                                                             \\
     & T_{(ij)|(kl)}                &                                                        & \text{Elasticity}
     &                              & 2\HT{0}  \oplus \HT{2} \oplus  \HT{4}                                                                                                           \\
     & T_{(ijkl)}                   &                                                        & \text{Totally symmetric}
     &                              & \HT{0}\oplus \HT{2}\oplus \HT{4}
  \end{align*}
\end{exa}

\begin{rem}\label{rem:dec-harm-Ela}
  Once an explicit harmonic decomposition has been fixed, a given tensor is parameterized by scalars, pseudo-scalars and harmonic tensors of order $n \ge 1$ (which depend on two parameters). For instance, in the case of the elasticity tensor, we get
  \begin{equation*}
    \bC = (\lambda, \mu, \bh, \bH)
  \end{equation*}
  where $\lambda, \mu \in \HT{0}$, $\bh\in \HT{2}$, $\bH\in \HT{4}$. The remarkable fact is that all the results we present in \autoref{sec:applications} are independent of this choice. All formulas are valid, independently of the particular choice of an explicit harmonic decomposition.
\end{rem}

\section{Invariant theory in 2D}
\label{sec:invariant-theory}

Let $(\VV,\rho)$ be a linear representation of $G = \OO(2)$ or $G = \SO(2)$. The action of $G$ on $\VV$ induces a linear representation of $G$ on the algebra $\RR[\VV]$ of polynomials functions on $\VV$, which will be denoted by $\star$, and which is given by
\begin{equation}\label{eq:gstarpdev}
  (g\star p)(\vv) :=  p(\rho(g)^{-1}\vv).
\end{equation}

\subsection{Invariant algebra}

The invariant algebra of $\VV$ under the group $G$, denoted by $\inv(\VV, G)$ (and more usually by $\RR[\VV]^G$ in the Mathematical community), is defined as
\begin{equation*}
  \inv(\VV, G) := \set{p\in \RR[\VV],\quad g\star p = p,\, \forall g\in G}.
\end{equation*}
It is a subalgebra of $\RR[\VV]$, which is furthermore \emph{finitely generated}, thanks to Hilbert's theorem~\cite{Hil1993,Stu2008}. Moreover, since the group action on polynomials preserves vector spaces of \emph{homogeneous polynomials of given degrees}, it can always be generated by homogeneous polynomial invariants.

\begin{defn}[Integrity basis]
  A finite set of $G$-invariant homogeneous polynomials $\set{J_{1}, \dotsc , J_{N}}$ over $\VV$ is a \emph{generating set} (also called an \emph{integrity basis}) of the invariant algebra $\inv(\VV, G)$ if any $G$-invariant polynomial $J$ over $\VV$ is a polynomial function in $J_{1},\dotsc,J_{N}$, \textit{i.e} if $J$ can be written as
  \begin{equation*}
    J(\vv) = P(J_{1}(\vv), \dotsc ,J_{N}(\vv)), \qquad \vv \in \VV,
  \end{equation*}
  where $P$ is a polynomial function in $N$ variables. An integrity basis is \emph{minimal} if no proper subset of it is an integrity basis.
\end{defn}

\begin{rem}
  A minimal integrity basis of homogeneous invariants is not unique, several choices are possible but its cardinality, as well as the degree of the generators are independent of the choice of a particular basis~\cite{DL1985/86}.
\end{rem}

\begin{defn}
  An homogeneous polynomial invariant is called \emph{reducible} if it can be written as the product of two (non constant) homogeneous polynomial invariants, or more generally as a sum of products of two (non constant) homogeneous polynomial invariants. Otherwise, it is called \emph{irreducible}.
\end{defn}

\begin{lem}\label{lem:reducible-invariants}
  Let $\mathcal{J} := \set{J_{1}, \dotsc , J_{N}}$ be a set of homogeneous polynomial invariants which generates $\inv(\VV, G)$. If some $J_{r} \in \mathcal{J}$ is \emph{reducible}, then $\mathcal{J} \setminus \set{J_{r}}$ is still a generating set of $\inv(\VV, G)$.
\end{lem}

\begin{proof}
  Suppose that $J_{r} \in \mathcal{J}$ is reducible. Then, it can be written as a sum of products of two (non constant) homogeneous polynomial invariants.
  \begin{equation*}
    J_{r} = \sum_{p,q} I_{p}I_{q},
  \end{equation*}
  where $\deg ( I_{p}I_{q})= \deg J_{r}$, for each pair $(p,q)$. Thus, for each $k$, $\deg I_{k} < \deg J_{r}$ (because $I_{k}$ is not constant). Besides, each $I_{k}$ writes as
  \begin{equation*}
    I_{k} = P_{k}(J_{1}, \dotsc , J_{N}).
  \end{equation*}
  But $P_{k}$ cannot depends on $J_{r}$ since $\deg I_{k} < \deg J_{r}$. The conclusion follows, since each $I_{k}$, and thus $J_{r}$, can then be rewritten as polynomial functions of the homogeneous invariants in $\mathcal{J} \setminus \set{J_{r}}$.
\end{proof}

\begin{cor}
  A minimal integrity basis constituted of homogeneous invariants contains only irreducible invariants.
\end{cor}

\subsection{Covariant algebra}

There is a useful extension of the concept of invariant which is called a \emph{covariant}~\cite{KP2000,OKDD2022}. Its definition involves two representations $\rho_{\VV}$ and $\rho_{\WW}$ of $G$ of the same group $G$ (see definition~\ref{def:linear-representation}).

\begin{defn}\label{def:covariants}
  A mapping $\vv \mapsto \bc(\vv)$ from $\VV$ to $\WW$ is a \emph{covariant} of $\vv\in \VV$ of type $\WW$, if
  \begin{equation*}
    \bc(\rho_{\VV}(g)\vv)= \rho_{\WW}(g) \bc(\vv), \qquad \forall g\in G.
  \end{equation*}
  It is called a \emph{polynomial covariant} of type $\WW$ if moreover the mapping $\vv \mapsto \bc(\vv)$ is polynomial in $\vv$.
\end{defn}

\begin{rem}
  In the situations we are usually concerned with in mechanics, $\VV$ and $\WW$ are tensor spaces and the condition $\bc(\rho_{\VV}(g)\vv)= \rho_{\WW}(g) \bc(\vv)$ is written generally as $\bc(g\star\bT)= g\star\bc(\bT)$, where it is understood that the action $\star$ is the usual action of $G$ on tensors. For $G=\SO(2)$, it simply means that the covariant $\bc(\bT)$ of $\bT$ is rotated by $g$ if $\bT$ is rotated by $g$.
\end{rem}

\begin{exa}\label{ex:covC}
  The harmonic components of $\bC\in \Ela$, $\bh$ and $\bH$ (see~remark \ref{rem:dec-harm-Ela}), as well as the symmetric second-order tensor $\bd_2(\bH):=\bH \3dots \bH$ introduced in~\cite{BKO1994}, are polynomial covariants of $\bC$ of respective type $\HT{2}$, $\HT{4}$ and $\Sym^{2}(\RR^2)$, for the actions of both groups $\SO(2)$ and $\OO(2)$.
\end{exa}

The concept of polynomial covariant is particularly useful when restricted to covariants of type $\Sym^{n}(\RR^{2})$ endowed with the tensorial representation
\begin{equation*}
  (\rho_{n}(g)\bS)(\xx_{1},\dotsc ,\xx_{n}) = \bS(g^{-1}\xx_{1},\dotsc ,g^{-1}\xx_{n}), \qquad g \in \OO(2).
\end{equation*}
In that case, to each polynomial covariant $\bc$ of $\vv$ of type $\Sym^{n}(\RR^{2})$, corresponds an homogeneous polynomial $\phi(\bc(\vv))$ of degree $n$ (see~\autoref{sec:tensors}), and
\begin{equation*}
  \rp_{\bc} (\vv, \xx) := \phi(\bc(\vv))(\xx) = \bc(\vv)(\xx, \dotsc, \xx)
\end{equation*}
is a polynomial function of both $\vv$ and $\xx$. Moreover, we get
\begin{align*}
  \rp_{\bc}(\rho_{\VV}(g)\vv, g \xx) & = \bc(\rho_{\VV}(g)\vv)(g \xx, \dotsc, g \xx)
  \\
                                     & = (\rho_{n}(g)\bc(\vv))(g \xx, \dotsc, g \xx)
  \\
                                     & = \bc(\vv)(g^{-1} (g \xx), \dotsc, g^{-1}(g \xx))
  \\
                                     & = \bc(\vv)(\xx, \dotsc, \xx)
  \\
                                     & = \rp_{\bc}(\vv, \xx).
\end{align*}

Therefore, to every covariant $\bc$ of type $\Sym^{n}(\RR^{2})$, corresponds a unique invariant polynomial $\rp_{\bc}$ of $(\vv, \xx)$. In other words, the polynomial covariants of type $\Sym^{n}(\RR^2)$ can be identified with elements of the invariant algebra
\begin{equation*}
  \RR[\VV\oplus \RR^{2}]^{G}.
\end{equation*}
This justifies the following definition.

\begin{defn}
  The \emph{covariant algebra} of $\VV$, \textit{i.e.} the algebra generated by the \emph{polynomial G-covariants} of $\VV$ of type $\Sym^{n}(\RR^2)$, denoted by $\cov{(\VV,G)}$, is defined as
  \begin{equation*}
    \cov{(\VV, G)} := \inv(\VV \oplus \RR^{2}, G),
  \end{equation*}
  where $G$ acts on $\VV \oplus \RR^{2}$ as
  \begin{equation*}
    (\vv,\xx) \mapsto (\rho(g)\vv,g\xx), \qquad v \in \VV,\, \xx \in \RR^{2},\, g \in G.
  \end{equation*}
\end{defn}

\begin{rem}
  Polynomial covariants appear to be much more useful than polynomial invariants to solve many problems in Invariant Theory such as the characterization of symmetry classes for instance~\cite{AR2016,OKDD2022} or the characterization of geometric properties \cite{Stu2008}.
\end{rem}

The covariant algebra $\cov{(\VV, G)}$ is naturally \emph{bi-graded}, by the degree $d$ in $\vv\in \VV$, on one hand, called the \emph{degree} of the covariant, and by the degree $o$ in $\xx = (x,y)\in \RR^{2}$, on second hand, called the \emph{order} of the covariant. The set of covariants of order $0$ is a subalgebra of $\cov{(\VV,G)}$ which corresponds exactly to $\inv(\VV, G)$ (this justifies the fact that the covariant algebra is an extension of the invariant algebra and contains more information). The set of covariants of degree $d$ and order $o$ is a finite dimensional vector subspace of $\cov{(\VV,G)}$ which is denoted by $\cov_{d, o}(\VV, G)$.

\begin{exa}\label{ex:covapoly}
  The following expressions are $\SO(2)$-covariants of $\ba \in \Sym^{2}(\RR^{2})$ of respective type $\Sym^{2}(\RR^{2})$, $\Sym^{2}(\RR^{2})$ and $\Sym^{6}(\RR^{2})$:
  \begin{equation*}
    \ba^{2}, \qquad \id \times \ba, \qquad \ba \odot \ba \odot \ba^{3},
  \end{equation*}
  where the notation $\ba^{n+1} = \ba^{n}\1dot\ba$ has been used. There polynomial counterparts write
  \begin{equation*}
    \ba^{2}(\xx, \xx)=\xx\cdot \ba^{2}\cdot \xx,  \qquad
    (\id \times \ba)(\xx, \xx)=\det(\ba\cdot \xx, \xx) ,
    \qquad (\xx\cdot \ba\cdot\xx)^{2} \,  (\xx\cdot\ba^{3}\cdot \xx),
  \end{equation*}
  and belong respectively to
  \begin{equation*}
    \cov_{2,2}(\Sym^{2}(\RR^{2}), \SO(2)), \qquad \cov_{1,2}(\Sym^{2}(\RR^{2}), \SO(2)), \qquad \cov_{5,6}(\Sym^{2}(\RR^{2}),\SO(2)).
  \end{equation*}
  The first one and the third one are also $\OO(2)$-covariants, but not the second one.
\end{exa}

\section{Computing integrity bases}
\label{sec:integrity-bases}

The approach developed in this section has already been applied to plane elasticity by Vianello~\cite{Via1997,FV2014}, following a work of Pierce~\cite{Pie1994/95}, and in a related way by Verchery some years before~\cite{Ver1982}. Our goal is to explain a systematic way to obtain integrity bases for the invariant algebra $\inv(\VV,G)$ of any linear representation $\VV$ of $G=\SO(2)$ or $G=\OO(2)$, the computation of an integrity basis for $\cov(\VV, G)$ being a particular case, since it is just the invariant algebra of $\VV\oplus\RR^{2}$.

\subsection{$\SO(2)$ invariant algebras}

We will start by studying the case of a representation $\VV$ of the rotation group $\SO(2)$. To compute an integrity basis for $\VV$, the first step is to split $\VV$ into irreducible components:
\begin{equation}\label{eq:SO2_Harm_Decomp}
  \VV \simeq \nu_{0}\HT{0} \oplus \HT{n_{1}} \oplus \dotsb \oplus \HT{n_{r}},  \quad
  \nu_{0}\in \NN, \quad n_{k}\in \NN^{*},
\end{equation}
where some $n_{k}$ may be equal (multiplicities of two-dimensional components $\HT{n_{k}}$ are allowed). An explicit way to accomplish this task is detailed in~\autoref{sec:explicit-harmonic-decomposition}. Using this decomposition, a polynomial on $\VV$ writes
\begin{equation*}
  p(\lambda_{1}, \dotsc , \lambda_{\nu_{0}},a_{1},b_{1},\dotsc,a_{r},b_{r}),
\end{equation*}
where $\lambda_{j}$ belongs to $\HT{0} = \RR$ and $(a_{k}, b_{k})\in \RR^2$ are the components of $\bH_{k} \in \HT{n_{k}}$ in some basis.

\begin{rem}\label{rem:lambda-reduction}
  Since each $\lambda_{j}$ is itself an invariant, every invariant polynomial which contains $\lambda_{j}$, and which is not reduced to it, is necessarily reducible. Our goal being to compute a minimal integrity basis, and thus irreducible invariants of $\VV$, we can thus consider only invariant polynomials which depend on
  \begin{equation*}
    (a_{1},b_{1},\dotsc,a_{r},b_{r}).
  \end{equation*}
  Indeed, a minimal integrity basis for $\VV$ consists of a minimal integrity basis of
  \begin{equation*}
    \HT{n_{1}} \oplus \dotsb \oplus\HT{n_{r}},
  \end{equation*}
  to which we must add $\lambda_{1}, \dotsc , \lambda_{\nu_{0}}$.
\end{rem}

Following Vianello~\cite{Via1997}, let us now introduce the complex variable $z_{k}: = a_{k} + ib_{k}$. Then, any real polynomial in $(a_{1},b_{1},\dotsc,a_{r},b_{r})$ can be recast as
\begin{equation*}
  \sum_{\alpha_{k}, \beta_{k}} c_{\alpha_{1}, \dotsc , \alpha_{r}, \beta_{1}, \dotsc , \beta_{r}} z_{1}^{\alpha_{1}}\dotsm z_{r}^{\alpha_{r}}\zb_{1}^{\beta_{1}}\dotsm \zb_{r}^{\beta_{r}},
  \qquad
  \alpha_{i}, \beta_{i} \in \NN,
\end{equation*}
in which the condition of being real writes
\begin{equation*}
  c_{\beta_{1}, \dotsc , \beta_{r}, \alpha_{1}, \dotsc , \alpha_{r}} = \overline{c_{\alpha_{1}, \dotsc , \alpha_{r}, \beta_{1}, \dotsc , \beta_{r}}},
\end{equation*}
where $\overline{c}$ means the complex conjugate of $c$. The advantage of this choice of variables is that the action of $\SO(2)$ preserves the monomials, since
\begin{equation*}
  \rho(r_{\theta}) z_{k} = e^{in_{k}\theta}z_{k}, \qquad \rho(r_{\theta}) \zb_{k} = e^{-in_{k}\theta}\zb_{k},
\end{equation*}
and thus
\begin{equation*}
  r_{\theta} \star \left(z_{1}^{\alpha_{1}}\dotsm z_{r}^{\alpha_{r}}\zb_{1}^{\beta_{1}}\dotsm \zb_{r}^{\beta_{r}} \right) = e^{i(n_{1}(\alpha_{1} -\beta_{1}) + \dotsb + n_{r}(\alpha_{r} - \beta_{r}))} z_{1}^{\alpha_{1}}\dotsm z_{r}^{\alpha_{r}}\zb_{1}^{\beta_{1}}\dotsm \zb_{r}^{\beta_{r}}.
\end{equation*}
We need, therefore, only to compute invariant monomials.

\begin{lem}\label{lem:Diophantine-equation}
  A monomial
  \begin{equation}\label{eq:defm}
    \bm : =  z_{1}^{\alpha_{1}}\dotsm z_{r}^{\alpha_{r}}\zb_{1}^{\beta_{1}}\dotsm \zb_{r}^{\beta_{r}}
  \end{equation}
  is $\SO(2)$-invariant if and only if $(\alpha_{1}, \dotsc , \alpha_{r}, \beta_{1}, \dotsc , \beta_{r})$ is solution of the \emph{linear Diophantine equation}
  \begin{equation}\label{eq:Diophantine-equation}
    n_{1}\alpha_{1} + \dotsb + n_{r}\alpha_{r} - n_{1}\beta_{1} - \dotsb - n_{r}\beta_{r} = 0.
  \end{equation}
\end{lem}

A solution $(\alpha_{1}, \dotsc , \alpha_{r}, \beta_{1}, \dotsc , \beta_{r})$ of~\eqref{eq:Diophantine-equation} is called \emph{irreducible} if it is not the sum of two non-trivial solutions, and \emph{reducible} otherwise. It was shown by Gordan~\cite{Gor1873} (see also~\cite[Section 6.5]{KR1984} and~\cite[Section 1.4]{Stu2008}) that there is only a finite number of irreducible solutions of~\eqref{eq:Diophantine-equation}. Algorithms to compute these irreducible solutions can be found in~\cite{Kry2006,BI2010}. As one can expect, such minimal solutions lead directly to a minimal integrity basis of $\inv(\VV, \SO(2))$.

\begin{thm}\label{thm:SO2-minimal-integrity-basis}
  Let $(\VV,\rho)$ be a real linear representation of $\SO(2)$ which decomposes as
  \begin{equation}
    \VV \simeq \nu_{0}\HT{0}\oplus \HT{n_{1}} \oplus \dotsb \oplus \HT{n_{r}}, \quad \nu_{0}\in \NN, \quad n_{k}\in \NN^{*}.
  \end{equation}
  Then, a minimal integrity basis of $\inv(\VV, \SO(2))$ consists of the homogeneous invariants
  \begin{equation}\label{eq:SO2-minimal-integrity-basis}
    \lambda_{i}, \quad \abs{z_{k}}^{2}, \quad \Re(\bm_{l}), \quad \Im(\bm_{l}),
  \end{equation}
  where $1 \le i \leq \nu_{0}$, $1 \le k \le r$, and $\bm_{l}$ are the irreducible solutions of~\eqref{eq:Diophantine-equation} such that $\bm_{l} \ne \mathbf{\overline{m}}_{l}$.
\end{thm}

\begin{proof}
  Consider first the algebra of \emph{complex} invariant polynomials
  \begin{equation*}
    \mathcal{A} : =  \CC[z_{1},\zb_{1},\dotsc,z_{r},\zb_{r}]^{\SO(2)}.
  \end{equation*}
  It follows from~\cite[Lemma 1.4.2]{Stu2008}, that a minimal integrity basis for $\mathcal{A}$ is given by
  \begin{equation*}
    z_{k}\zb_{k}, \qquad \bm_{l}, \qquad \overline{\bm}_{l},
  \end{equation*}
  where $1 \le k \le r$ and $\bm_{l}$ are the irreducible solutions of~\eqref{eq:Diophantine-equation} such that $\bm_{l} \ne \mathbf{\overline{m}}_{l}$. Thus, $\mathcal{A}$ is also generated by
  \begin{equation*}
    z_{k}\zb_{k}, \qquad \Re(\bm_{l}) = \frac{1}{2}(\bm_{l} + \overline{\bm}_{l}), \qquad \Im(\bm_{l}) = \frac{1}{2i}(\bm_{l} - \overline{\bm}_{l}),
  \end{equation*}
  which is still minimal. Now every real polynomial in
  \begin{equation*}
    \inv(\VV, \SO(2)) = \RR[a_{1},b_{1}, \dotsc , a_{r},b_{r}]^{\SO(2)}
  \end{equation*}
  is a real polynomial in
  \begin{equation*}
    z_{k}\zb_{k}, \qquad \Re(\bm_{l}) = \frac{1}{2}(\bm_{l} + \overline{\bm}_{l}), \qquad \Im(\bm_{l}) = \frac{1}{2i}(\bm_{l} - \overline{\bm}_{l}).
  \end{equation*}
  Hence, this set is a generating set of $\RR[a_{1},b_{1},\dotsc, a_{r},b_{r}]^{\SO(2)}$ which is also minimal. Otherwise, one of these invariants could be written as a real polynomial in the others and this would contradict the fact that this set is minimal as a generating set of $\mathcal{A}$. As already stated (see remark~\ref{rem:lambda-reduction}), we conclude that
  \begin{equation*}
    \lambda_{i}, \quad \abs{z_{k}}^{2}, \quad \Re(\bm_{l}), \quad \Im(\bm_{l}),
  \end{equation*}
  is a minimal integrity basis of $\inv(\VV,\SO(2))$.
\end{proof}

\begin{exa}\label{ex:ElaSO2}
  A minimal integrity basis for the action of $\SO(2)$ on $\Ela$ has been computed in~\cite{Via1997}, using the harmonic decomposition, and representing an elasticity tensor $\bC=(\lambda, \mu, \bh, \bH)$ (see remark \ref{rem:dec-harm-Ela}) in complex form. The basis writes
  \begin{equation*}
    \lambda, \quad \mu, \quad \abs{z_{2}}^{2}, \quad \abs{z_{4}}^{2}, \quad \Re(z_{2}^{2}\overline{z_{4}}), \quad \Im(z_{2}^{2}\overline{z_{4}}),
  \end{equation*}
  where $z_{2} = h_{11} + i h_{12}$ and $z_{4} = H_{1111} + iH_{1112}$ are the components of $\bH_{2} := \bh\in \HT{2}$ and $\bH_{4} := \bH\in \HT{4}$ in some orthonormal basis of $\RR^{2}$. The tensorial expressions of these invariants are provided in example~\ref{exa:tensors-invariants-elasticity-tensor}.
\end{exa}

\subsection{$\OO(2)$ invariant algebras}

Consider now a representation $\VV$ of the orthogonal group $\OO(2)$, which decomposes as
\begin{equation*}
  \VV \simeq m_{-1}\HT{-1}\oplus m_{0}\HT{0}\oplus \HT{n_{1}} \oplus \dotsb \oplus\HT{n_{r}},
\end{equation*}
where $m_{-1} \ge 0$, $m_{0} \ge 0$ and $n_{i} \ge 1$. An integrity basis of $\inv(\VV,\OO(2))$ will be obtained from one of the invariant algebra of the restriction of the representation of $\OO(2)$ on $\VV$ to its subgroup $\SO(2)$. This integrity basis will \emph{not be minimal in general} and further computations will be necessary to extract from it a minimal integrity basis.

\begin{thm}\label{thm:O2-integrity-basis}
  Let $(\VV,\rho)$ be a real linear representation of $\OO(2)$ which decomposes as
  \begin{equation*}
    \VV \simeq m_{-1}\HT{-1}\oplus m_{0}\HT{0}\oplus \HT{n_{1}} \oplus \dotsb \oplus\HT{n_{r}},
  \end{equation*}
  where $m_{-1},m_{0} \in \NN$ and $n_{k}\in \NN^{*}$. Then, an integrity basis for $\inv(\VV,\OO(2))$ consists of the homogeneous invariants
  \begin{equation}\label{eq:O2-generators}
    \lambda_{k}, \quad \abs{z_{l}}^{2}, \quad \Re(\bm_{l}), \quad \xi_{i}\xi_{j}, \quad \xi_{i}\Im(\bm_{l}), \quad \Im(\bm_{p})\Im(\bm_{q}),
  \end{equation}
  where $\lambda_{k}\in \HT{0}$, $\xi_{i} \in \HT{-1}$ and where $\bm_{l}$ are the irreducible solutions of~\eqref{eq:Diophantine-equation} such that $\bm_{l} \ne \mathbf{\overline{m}}_{l}$ and only remains the terms $\Im(\bm_{p})\Im(\bm_{q})$ for which neither $\bm_{p}\bm_{q}$, nor $\bm_{p}\overline{\bm}_{q}$ contains a factor $z_{s}\overline{z}_{s}$ for some $s\in\set{1,\dotsc,r}$.
\end{thm}

The proof of theorem~\ref{thm:O2-integrity-basis} requires a useful tool in invariant theory called the \emph{Reynolds operator}, which is defined as follows.

\begin{defn}
  Given a compact group $G$ and a linear representation $\VV$ of $G$, the \emph{Reynolds operator} is the linear projector from $\RR[\VV]$ onto the invariant algebra $\inv(\VV,G)$, defined as
  \begin{equation}\label{eq:Reynolds-projector}
    R_{G}(\rp) := \int_{G} (g \star \rp)\, \mathrm{d} \mu, \qquad \rp \in \RR[\VV],
  \end{equation}
  where $\mathrm{d} \mu$ is the Haar measure on $G$.
\end{defn}

\begin{defn}\label{def:Haar-measure}
  Given a compact group $G$, the \emph{Haar measure} is a (bi-invariant) probability measure on $G$ and is uniquely defined~\cite{Ste1994}. For $G = \SO(2)$, it writes as
  \begin{equation*}
    \int_{\SO(2)} f(g)\, \mathrm{d} \mu = \frac{1}{2\pi}\int_{0}^{2\pi} f(r_{\theta})\,\mathrm{d}\theta ,
  \end{equation*}
  for every continuous function $f$ on $\SO(2)$, whereas for $G = \OO(2)$, it writes as
  \begin{equation*}
    \int_{\OO(2)} f(g)\, \mathrm{d} \mu = \frac{1}{4\pi}\int_{0}^{2\pi} f(r_{\theta})\,\mathrm{d}\theta + \frac{1}{4\pi}\int_{0}^{2\pi} f(\sigma r_{\theta})\,\mathrm{d}\theta ,
  \end{equation*}
  for every continuous function $f$ on $\OO(2)$.
\end{defn}

\begin{proof}[Proof of theorem~\ref{thm:O2-integrity-basis}]
  Consider an $\OO(2)$-invariant polynomial $\rp$. It is obviously invariant under $\SO(2)$ and we get thus
  \begin{equation}\label{eq:RO2p}
    \rp = R_{\OO(2)}(\rp) = \frac{1}{2} \left(R_{\SO(2)}(\rp) + \sigma \star R_{\SO(2)}(\rp)\right) = \frac{1}{2} (\rp + \sigma \star \rp).
  \end{equation}
  Now as an element of $\inv(\VV,\SO(2))$ and using theorem~\ref{thm:SO2-minimal-integrity-basis}, $\rp$ can be written as a polynomial expression
  \begin{equation*}
    \rp = P(\lambda_{k}, \xi_{i}, \abs{z_{l}}^{2}, \Re(\bm_{l}), \Im(\bm_{l})),
  \end{equation*}
  and we have moreover
  \begin{equation*}
    \sigma \star \lambda_{k} = \lambda_{k}, \qquad \sigma \star \xi_{i} = - \xi_{i},
  \end{equation*}
  and
  \begin{equation*}
    \sigma \star \abs{z_{l}}^{2} = \abs{z_{l}}^{2}, \quad \sigma \star \Re(\bm_{l}) = \Re(\bm_{l}), \quad \sigma \star \Im(\bm_{l}) = - \Im(\bm_{l}).
  \end{equation*}
  We get thus
  \begin{equation*}
    \sigma \star  \rp=P(\lambda_{k}, -\xi_{i}, \abs{z_{l}}^{2}, \Re(\bm_{l}), -\Im(\bm_{l})).
  \end{equation*}
  Now, using~\eqref{eq:RO2p}, we have
  \begin{equation*}
    \rp = \frac{1}{2} \left(P(\lambda_{k}, \xi_{i}, \abs{z_{l}}^{2}, \Re(\bm_{l}), \Im(\bm_{l})) + P(\lambda_{k}, -\xi_{i}, \abs{z_{l}}^{2}, \Re(\bm_{l}), -\Im(\bm_{l}))\right),
  \end{equation*}
  and expanding $P$, we deduce thereby that $\inv(\VV,\OO(2))$ is generated by the homogeneous invariants
  \begin{equation*}
    \lambda_{k}, \quad \abs{z_{l}}^{2}, \quad \Re(\bm_{l}), \quad \xi_{i}\xi_{j}, \quad \xi_{i}\Im(\bm_{l}), \quad \Im(\bm_{p})\Im(\bm_{q}).
  \end{equation*}
  Note however that
  \begin{align*}
    \Im(\bm_{p})\Im(\bm_{q}) & = \Re(\bm_{p})\Re(\bm_{q}) - \Re(\bm_{p}\bm_{q})                        \\
                             & = \Re(\bm_{p}\overline{\bm}_{q}) - \Re(\bm_{p})\Re(\overline{\bm}_{q}).
  \end{align*}
  Hence, we can remove $\Im(\bm_{p})\Im(\bm_{q})$ from the list of generators, each time $\bm_{p}\bm_{q}$ or $\bm_{p}\overline{\bm}_{q}$ can be recast as
  \begin{equation*}
    (z_{s}\overline{z}_{s})\bm,
  \end{equation*}
  for some $s \in \set{1, \dotsc , r}$ and $\bm$ is a monomial which satisfies~\eqref{eq:Diophantine-equation}. Indeed, then $\Im(\bm_{p})\Im(\bm_{q})$ is reducible and can be removed from the set of generators by lemma~\ref{lem:reducible-invariants}. This applies, in particular, to each invariant $(\Im(\bm_{l}))^{2}$.
\end{proof}

The elimination of $\Im(\bm_{p})\Im(\bm_{q})$, each time $\bm_{p}\bm_{q}$ or $\bm_{p}\overline{\bm}_{q}$ contains a factor $z_{s}\overline{z}_{s}$ in the list of generators, does not lead, in general, to a minimal basis, even if it reduces \textit{a priori} the number of generators, sometimes drastically.

\begin{rem}
  For those of you who have been involved in similar calculations, the problem of whether such invariants as products $\Im\bm_{p}\Im\bm_{q}$ could always be removed \textit{a priori} from a minimal basis of $\inv(\VV,\OO(2))$ founds here a definitive answer. Indeed, there are examples  in~\autoref{sec:applications} where such products cannot be removed (even in the case of totally symmetric tensors, see~\autoref{subsec:piezoelectricity-law}, for instance).
\end{rem}

A reduction procedure, which we call \emph{cleaning} and described in~\autoref{sec:algorithm} is thus required to obtain a minimal integrity basis or to check that a given basis is already minimal. In practice, and the argument will be used when applying the cleaning procedure, it is enough to reduce the integrity basis
\begin{equation}\label{eq:B}
  \mathcal{B} := \set{\abs{z_{k}}^{2}, \Re(\bm_{l}),\Im(\bm_{i})\Im(\bm_{j})}
\end{equation}
of $\inv(\HT{n_{1}} \oplus \dotsb \oplus\HT{n_{r}}, \OO(2))$, to obtain a minimal integrity basis of the full space
\begin{equation*}
  m_{-1}\HT{-1}\oplus m_{0}\HT{0}\oplus \HT{n_{1}} \oplus \dotsb \oplus\HT{n_{r}}.
\end{equation*}
The argument is formalized as the following theorem.

\begin{thm}\label{thm:important-reduction}
  Let $\mathcal{MB}$ be a minimal integrity basis of
  \begin{equation*}
    \inv(\HT{n_{1}} \oplus \dotsb \oplus\HT{n_{r}}, \OO(2)),
  \end{equation*}
  extracted from $\mathcal{B}$~\eqref{eq:B}. Then, a minimal integrity basis for
  \begin{equation*}
    \inv(m_{-1}\HT{-1}\oplus m_{0}\HT{0}\oplus \HT{n_{1}} \oplus \dotsb \oplus\HT{n_{r}}, \OO(2)),
  \end{equation*}
  is given by
  \begin{equation*}
    \mathcal{MB} \cup \set{\lambda_{k}, \xi_{i}\xi_{j},\xi_{i}\Im(\bm_{l})},
  \end{equation*}
  where $0 \le i,j \leq m_{-1}$, $0 \le k \leq m_{0}$, and $\bm_{l}$ are the irreducible solutions of~\eqref{eq:Diophantine-equation} such that $\bm_{l} \ne \mathbf{\overline{m}}_{l}$
\end{thm}

\begin{proof}
  The invariant algebra $\inv(\VV,\OO(2))$ is multi-graded; each invariant writes uniquely as a sum of invariants which are multi-homogeneous relatively to the decomposition
  \begin{equation*}
    \VV \simeq m_{-1}\HT{-1}\oplus m_{0}\HT{0}\oplus \HT{n_{1}} \oplus \dotsb \oplus\HT{n_{r}}.
  \end{equation*}
  In other words,
  \begin{equation*}
    \inv(\VV,\OO(2)) = \bigoplus_{K} \inv_{K}(\VV,\OO(2)),
  \end{equation*}
  where
  \begin{equation*}
    K := (s_{1}, \dotsc s_{m_{-1}}, e_{1}, \dotsc e_{m_{0}}, k_{1}, \dotsc, k_{r}),
  \end{equation*}
  is a multi-index in which $s_{i}$ indicates the degree in $\xi_{i}$, $e_{k}$ indicates the degree in $\lambda_{k}$, $k_{i}$ indicates the degree in $\bH_{n_i}$ and $\inv_{K}(\VV)$ is the vector space of multi-homogeneous invariants of multi-degree $K$. Any relation among multi-homogeneous invariants happens in one vector space $\inv_{K}(\VV,\OO(2))$. Thus, neither $\lambda_{k}$, nor $\xi_{i}\xi_{j}$, can be recast using other invariants from the set~\eqref{eq:O2-generators}. This is because the spaces $\inv_{K}(\VV,\OO(2))$ which contains either $\lambda_{k}$ or $\xi_{i}\xi_{j}$ are one-dimensional. This is also true for $\xi_{i}\Im(\bm_{l})$, not because the corresponding space $\inv_{K}(\VV,\OO(2))$, to which it belongs is one-dimensional, but because if it could be recast using other invariants from the set~\eqref{eq:O2-generators}, then $\Im(\bm_{l})$ could be re-written using $\abs{z_{k}}^{2}$, $\Re(\bm_{i})$ and $\Im(\bm_{j})$ ($j \ne l$), which would lead to a contradiction.
\end{proof}

\begin{exa}\label{ex:ElaO2}
  A minimal integrity basis for the action of $\OO(2)$ on $\Ela$ has been computed in~\cite{Via1997}. It consists in the following five invariants
  \begin{equation*}
    \lambda, \quad \mu, \quad z_{2}\overline{z_{2}}, \quad z_{4}\overline{z_{4}}, \quad \Re(z_{2}^{2}\overline{z_{4}}),
  \end{equation*}
  where $z_{2} = h_{11} + i h_{12}$ and $z_{4} = H_{1111} + iH_{1112}$. The tensorial expressions of these invariants are provided in example~\ref{exa:tensors-invariants-elasticity-tensor}.
\end{exa}

We finally formulate as a theorem another reduction result, which avoids useless computations.

\begin{thm}\label{thm:subspace-reduction}
  Let $\VV = \HT{n_{1}} \oplus \dotsb \oplus\HT{n_{p}}$, where $n_{k} \ge -1$. Then, any stable subspace $\WW$ of $\VV$ writes
  \begin{equation*}
    \WW = \HT{n_{k_{1}}} \oplus \dotsb \oplus\HT{n_{k_{p}}},
  \end{equation*}
  where $\set{n_{k_{1}}, \dotsc , n_{k_{p}}}$ is a subset of $\set{n_{1}, \dotsc , n_{p}}$. Moreover, given any minimal integrity basis $\mathcal{MB}$ of $\inv(\VV,\OO(2))$, which consists only of \emph{multi-homogenous} invariants, a \emph{minimal integrity basis} of $\inv(\WW,\OO(2))$ is obtained by extracting, from $\mathcal{MB}$, multi-homogenous invariants which depend only on the variables
  \begin{equation*}
    \bH_{n_{k_{1}}}, \dotsc , \bH_{n_{k_{p}}}.
  \end{equation*}
\end{thm}

\begin{proof}
  Let $\WW$ be a stable subspace of $\VV$. Since each $\HT{n}$ is an irreducible representation, each stable subspace $\WW \cap \HT{n}$ of $\HT{n}$ is either $\set{0}$ or $\HT{n}$. This proves the first assertion. Consider now a \emph{multi-homogenous invariant} $J$ on $\VV$, then evaluated on $\WW$ it either vanishes if it depends on more variables than $\bH_{n_{k_{1}}}, \dotsc , \bH_{n_{k_{p}}}$ or is equal to itself overwise. Finally, given $I \in \inv(\WW,\OO(2))$, observe that it extends naturally as a element of $\inv(\VV,\OO(2))$ by setting variables other than $\bH_{n_{k_{1}}}, \dotsc , \bH_{n_{k_{p}}}$ to $0$. It can thus be recast as a polynomial in the multi-homogeneous invariants of $\mathcal{MB}$, but when evaluated on $\WW$, each term of $\mathcal{MB}$ which depends on more variables than $\bH_{n_{k_{1}}}, \dotsc , \bH_{n_{k_{p}}}$ vanishes, which concludes the proof.
\end{proof}

\begin{rem}\label{rem:Snmoins2k}
  Theorem~\ref{thm:subspace-reduction} applies, in particular to any stable subspace $\WW$ of $\VV=\TT^{n}(\RR^{2})$, defined by some index symmetries, and thus in particular to $\WW = \Sym^{n}(\RR^{2})$. It applies also to the space $\WW=\Sym^{n-2k}(\RR^{2})$, which can be considered as a subspace of $\VV=\Sym^{n}(\RR^{2})$, because there is a natural and equivariant linear embedding
  \begin{equation*}
    \Sym^{n-2k}(\RR^{2}) \to \Sym^{n}(\RR^{2}), \qquad \bS \mapsto \underbrace{\id\odot \dotsm \odot \id}_\textrm{$k$ copies} \odot \bS .
  \end{equation*}
\end{rem}

\section{Cleaning algorithm}
\label{sec:algorithm}

Starting from {a finite generating set $\mathcal{B}$~\eqref{eq:B} of the invariant algebra
\begin{equation}\label{eq:def-A}
  \mathcal{A} := \inv(\HT{n_{1}} \oplus \dotsb \oplus\HT{n_{r}}, \OO(2)),
  \qquad
  n_{k_{i}}\geq 1,
\end{equation}
the \emph{cleaning algorithm} produces a minimal integrity basis $\mathcal{MB}$ extracted from $\mathcal{B}$. The invariant algebra $\mathcal{A}$ is multi-graded by the fact that each polynomial invariant can be uniquely decomposed into a sum of polynomial invariants which are homogeneous into each factor (multiplicities allowed)
\begin{equation*}
  \HT{n_{1}}, \dotsc, \HT{n_{r}},
\end{equation*}
of respective degrees $k_{1},k_{2}, \dotsc , k_{r}$. This information will be encoded into the \emph{multi-index}
\begin{equation*}
  K := (k_{1},k_{2}, \dotsc , k_{r}).
\end{equation*}
Therefore, the invariant algebra $\mathcal{A}$ can be decomposed as the direct sum
\begin{equation*}
  \mathcal{A} = \bigoplus_{K} \mathcal{A}_{K},
\end{equation*}
where each $\mathcal{A}_{K}$ is the finite dimensional subspace of $\mathcal{A}$ consisting of multi-homogeneous invariants of multi-degree $K$. The remarkable fact is that the} dimension $a_{K}$ of $\mathcal{A}_{K}$ can be computed \textit{a priori} using the \emph{Hilbert series} of $\mathcal{A}$ (see~\autoref{sec:Hilbert-series}), which writes, by theorem \ref{thm:O2-Hilbert-series} and remark \ref{rem:betak},
\begin{equation*}
  H(t_{1},\dotsc,t_{r}) = \sum_{K} a_{K} t_{1}^{k_{1}}\dotsm t_{r}^{k_r},
  \qquad
  a_{K} = \frac{1}{2}(b_{K}+\beta_{K}),
\end{equation*}
where $\beta_{K}=0$ each time one of the $k_{i}$ is odd and $\beta_{K}=1$ otherwise, and where $b_{K}$ is the number of solutions $(\alpha_{1}, \dotsc, \alpha_{r})$ of the linear Diophantine equation
\begin{equation}\label{eq:diopheqtext}
  2\alpha_{1}n_{1} + \dotsb + 2\alpha_{r}n_{r} = k_{1}n_{1} + \dotsb + k_{r}n_{r},\quad \alpha_{i}\geq 0 .
\end{equation}
Let now $\mathcal{B}_{K} := \mathcal{B} \cap \mathcal{A}_{K}$ be the subset of $\mathcal{B}$ of homogeneous polynomials with multi-degree $K = (k_{1},k_{2}, \dotsc , k_{r})$. Choosing a total order $\preceq$ on the set of multi-index $K$, leads to a partitioning of $\mathcal{B}$ as
\begin{equation*}
  \mathcal{B} = \bigsqcup_{i=0}^{N} \mathcal{B}_{K_{i}}, \quad \text{where} \quad K_{i} \prec K_{j}, \quad \text{if} \quad i<j.
\end{equation*}

\begin{rem}
  Any finite set $\mathcal{S}$ of $p$ homogeneous polynomials in $\mathcal{A}_{K}$ is a family of \emph{vectors} $\vv_{1}, \dotsc,\vv_p$ in the finite dimensional space $\mathcal{A}_{K}$ of dimension $a_{K}$ and we can thus define its rank, $\rg(\mathcal{S})$.
\end{rem}

The proposed cleaning algorithm with
\begin{itemize}
  \item inputs: $\mathcal{B}_{K_{i}}$, $a_{K_{i}}$, with $K_{i} \prec K_{j}$, for all $i<j$,
  \item output : $\mathcal{MB}$,
\end{itemize}
consists in:
\begin{enumerate}
  \item \textbf{Initialization} : determine a subfamily $\mathcal{F}^0\subset \mathcal{B}_{K_{0}}$ of linearly independent polynomials such that $\rg(\mathcal{F}^0)=a_{K_{0}}$.
  \item \textbf{Iteration step $n$ ($1 \leq n \leq N$)}: suppose that we have obtained, at step $n-1$, the family $\mathcal{F}^{n-1} = \set{I_{1},\dotsc, I_{s}}$ and note that $\mathcal{F}^{n-1}$ may contain homogeneous polynomials with different multi-indices $K(I_{1}), \dotsc , K(I_{s})$ but all are strictly lower than $K_{n}$, where $K(I)$ stands for the multi-index of homogeneous polynomial $I$.
        \begin{enumerate}
          \item Determine the finite set $\mathcal{R}_{K_{n}}$ of all \emph{reducible} homogeneous polynomials of multi-degree $K_{n}$ that can be constructed, in two steps, using elements of $\mathcal{F}^{n-1}$:
                \begin{enumerate}
                  \item Find the $p$ solutions $\alpha_{1}^j, \alpha_{2}^j \dotsc, \alpha_{s}^j$ ($1\leq j\leq p$) of the linear Diophantine system
                        \begin{equation}\label{eq:Dioph_Equ_Hom_MDeg_m+1}
                          \alpha_{1}K(I_{1})+\dotsc +\alpha_{s}K(I_{s})=K_{n},
                        \end{equation}
                  \item If $p>0$, $\mathcal{R}_{K_{n}} = \set{I_{1}^{\alpha_{1}^j}I_{2}^{\alpha_{2}^j} \dotsm I_{s}^{\alpha_{s}^j};\;
                          1\le j\le p}$, else $\mathcal{R}_{K_{n}} = \emptyset$,
                \end{enumerate}
          \item if $\rg(\mathcal{R}_{K_{n}})=a_{K_{n}}$, $\mathcal{X}_{K_{n}}=\emptyset$, go to (d),
          \item Determine a subset $\mathcal{X}_{K_{n}} \subset \mathcal{B}_{K_{n}}$ of \emph{minimal cardinal} such that
                \begin{equation*}
                  \rg(\mathcal{R}_{K_{n}}\cup \mathcal{X}_{K_{n}})= a_{K_{n}},
                \end{equation*}
                \textit{i.e.} check one by one the invariants of $\mathcal{B}_{K_{n}}$ that need to be added to match the dimension $a_{K_{n}}$. This requires to compute the rank of the new set of vectors in $\mathcal{A}_{K_{n}}$, each time we add a new element,
          \item $\mathcal{F}^{n} := \mathcal{F}^{n-1}\cup \mathcal{X}_{K_{n}}$,
        \end{enumerate}
  \item  \textbf{Termination:} $\mathcal{MB}:=\mathcal{F}^{N}$.
\end{enumerate}

\medskip

The cleaning algorithm was applied with the following specifications.
\begin{itemize}
  \item All minimal solutions of the Diophantine equation~(\ref{eq:Diophantine-equation}) were obtained using the software of algebraic geometry \emph{Normaliz}~\cite{BI2010};
  \item The Diophantine equation \eqref{eq:diopheqtext} as well as the linear Diophantine system \eqref{eq:Dioph_Equ_Hom_MDeg_m+1}
        were solved using \emph{Mathematica}~\cite{Mathematica};
  \item The adopted total order on multi-index $K=(k_{1}, \dotsc, k_{r})$ was the lexicographic order
        \begin{equation*}
          K_{i} \preceq K_{j} \Leftrightarrow k^{i}_{1}=k^{j}_{1},\dotsc,k^{i}_{q}=k^{j}_{q},\quad k^{i}_{q+1}<k^{j}_{q+1},
          \quad \text{ for some } 1\leq q\leq r-1;
        \end{equation*}
  \item In step (2)(c) of the algorithm, it is necessary to order the element of $\mathcal{B}_{K_{n}}=\mathcal{B}\cap\mathcal{A}_{K_{n}}$ to be tested. We have used the Mathematica built-in function \textbf{Sort}.
\end{itemize}

\begin{rem}
  When the covariant algebra is involved, the cleaning algorithm is applied to the invariant algebra
  \begin{equation*}
    \inv(\RR^2 \oplus \HT{n_{1}} \oplus \dotsb \oplus\HT{n_{r}}, \OO(2)).
  \end{equation*}
  In that case, the multi-index has been numbered as
  \begin{equation*}
    K := (k_{0},k_{1},k_{2}, \dotsc , k_{r}),
  \end{equation*}
  where $k_{0}$ represents the degree in $\xx$, i.e. the order of the associated covariant. The choice of the adopted lexicographic order implies that the cleaning is processed by increasing the order of covariants, first. Of course, other total orders on the set of multi-index are possible and they may be more adapted to other situations considered.
\end{rem}

\section{From complex monomials to tensor covariants}
\label{sec:monomials-versus-tensors-invariants}

Regarding mechanical applications, an integrity basis should be expressed in terms of tensors invariants, rather than in terms of complex monomials. A translation of the real and imaginary parts $\Re(\bm_{l})$, $\Im(\bm_{l})$ of the previous monomials is thus mandatory. As recalled in~\autoref{sec:tensors}, there is a natural correspondence
\begin{equation*}
  \phi: \Sym^{n}(\RR^{2})\to \Pn{n},
\end{equation*}
which associates to any totally symmetric tensor $\bS$ of order $n$, an homogeneous polynomial $\rp$ of degree $n$, which writes as
\begin{equation*}
  \rp(\xx) = \bS(\xx, \dotsc, \xx), \qquad \xx = (x,y).
\end{equation*}
Under this isomorphism, which is $\OO(2)$-equivariant, the subspace $\HT{n}$ of harmonic tensors of order $n$ (traceless tensors) is sent to the subspace of homogeneous harmonic polynomials $\HP{n}$ (polynomials with vanishing Laplacian). A natural basis for $\HP{n}$ is given by the real and imaginary parts of the complex function $z^{n}=(x+iy)^{n}$,
\begin{equation*}
  \rp^{(n)}_{1} = \Re(x + iy)^{n}, \qquad \rp^{(n)}_{2} = \Im(x + iy)^{n}.
\end{equation*}
This basis corresponds to the image under $\phi$ of the following basis of $\HT{n}$
\begin{equation}\label{eq:Hn-tensorial-basis}
  \begin{aligned}
    \bK^{(n)}_{1} & = \sum_{k=0}^{\lfloor\frac{n}{2}\rfloor} \binom{n}{2k} (-1)^{k} \ee_{1}^{n-2k} \odot \ee_{2}^{2k},
    \\
    \bK^{(n)}_{2} & = \sum_{k=0}^{\lfloor\frac{n-1}{2}\rfloor} \binom{n}{2k+1} (-1)^{k} \ee_{1}^{n-(2k+1)} \odot \ee_{2}^{2k+1},
  \end{aligned}
\end{equation}
where $\odot$ stands for the symmetric tensor product \eqref{eq:odot} and
\begin{equation*}
  \ee_{i}^{p} := \ee_{i}\odot \dotsb \odot \ee_{i}
\end{equation*}
means the tensor product of $p$ copies of vector $\ee_{i}$. Thus, any harmonic polynomial $\rh$ in $\HP{n}$ writes
\begin{equation*}
  \rh = a \rp^{(n)}_{1} + b \rp^{(n)}_{2},
\end{equation*}
and the harmonic tensor $\bH = \phi^{-1}(\rh)$ in $\HT{n}$ writes
\begin{equation*}
  \bH =a\, \bK^{(n)}_{1} + b\,\bK^{(n)}_{2},
  \quad\textrm{where}\quad
  \begin{cases}
    a=H_{11 \cdots 11}
    \\
    b=H_{11 \cdots 12}
  \end{cases}
\end{equation*}
while the other components~\cite{Kan1984} of $\bH$ are
\begin{equation*}
  H_{\underbrace{{\scriptstyle 1 \cdots 1}}_{n-2p}\underbrace{{\scriptstyle 2 \cdots 2}}_{2p}} = (-1)^{p} H_{{\scriptstyle 1 \cdots 11}}
  \quad
  \text{and}
  \quad
  H_{\underbrace{{\scriptstyle 1 \cdots 1}}_{n-2p-1}\underbrace{{\scriptstyle 2 \cdots 2}}_{2p+1}} = (-1)^{p} H_{{\scriptstyle 1 \cdots 12}}.
\end{equation*}

\begin{rem}
  In both cases, the matrix form of $\rho_{n}$ in these bases is
  \begin{equation*}
    [\rho_{n}(r_{\theta})] =
    \begin{pmatrix}
      \cos n\theta & - \sin n\theta \\
      \sin n\theta & \cos n\theta
    \end{pmatrix}.
  \end{equation*}
  Note however, that none of the defined bases are \emph{orthonormal} for the natural scalar products on both spaces. There are however orthogonal and their norms are equal  {(see remark~\ref{rem:scaling-factors})}. Normalizing the bases, will not change the matrix representation and is thus inessential.
\end{rem}

Consider now a representation
\begin{equation}\label{eq:direct-sum}
  \VV \simeq \HT{n_{1}} \oplus \dotsb \oplus \HT{n_{r}},
\end{equation}
of $G = \SO(2)$ or $G = \OO(2)$, where $n_{k}\ge 1$ for each $k$. Generating sets for $\inv(\VV, G)$ have been provided in~\autoref{sec:integrity-bases}, but in terms of monomials
\begin{equation*}
  \bm =  z_{1}^{\alpha_{1}}\dotsm z_{r}^{\alpha_{r}}\zb_{1}^{\beta_{1}}\dotsm \zb_{r}^{\beta_{r}},
\end{equation*}
whose exponents satisfy the linear Diophantine equation~\eqref{lem:Diophantine-equation}. Here, $z_{k}=a_{k}+ i b_{k}$ corresponds to the components $(a_{k}, b_{k})$ of the factor $\bH_{k}\in \HT{n_{k}}$ in the direct sum~\eqref{eq:direct-sum} relative to the basis $(\bK^{(n)}_{1},\bK^{(n)}_{2})$, where we have used the correspondence
\begin{equation*}
  \rh_{k}(\xx) = \bH_{k}(\xx, \dotsc, \xx)\in \HP{n_{k}},
\end{equation*}
and recast $\rh_{k}$ as a polynomial function of the complex variables $z=x+iy$ and $\overline{z}=x-iy$.

It is the goal of this section to translate real and imaginary parts of the monomials $\bm$ into tensors invariants. To do so, observe first that
\begin{equation*}
  \rh_{k} = \Re (\overline{z}_{k}z^{n_{k}}) = \Re (z_{k}\overline{z}^{n_{k}}) = \frac{1}{2}\left(\overline{z}_{k}z^{n_{k}} + z_{k}\overline{z}^{n_{k}}\right),
\end{equation*}
and its conjugate harmonic function $\widetilde{\rh}_{k}$ writes as
\begin{equation*}
  \widetilde{\rh}_{k} = \Im (\overline{z}_{k}z^{n_{k}}) = -\Im (z_{k}\overline{z}^{n_{k}}) = \frac{1}{2i}\left(\overline{z}_{k}z^{n_{k}} - z_{k}\overline{z}^{n_{k}}\right),
\end{equation*}
where $z = x+iy$ and $z_{k} = a_{k}+ib_{k}$. We will now provide three theorems which enable to translate real and imaginary parts of monomials $\bm$ into tensors invariants. Their proofs are provided in~\autoref{sec:proofs}.

\begin{thm}\label{thm:1xH}
  Let $\rh_{k} = \phi(\bH_{k}) \in \HP{n_{k}}$, where $n_{k} \ge 1$. Then,
  \begin{equation*}
    \widetilde \rh_{k} = \Im (\overline{z}_{k}z^{n_{k}}) = \phi(\id \times \bH_{k}) .
  \end{equation*}
\end{thm}

\begin{thm}\label{thm:harmonic-product}
  Let $\bH_{j}\in \HT{n_{j}}$ be harmonic tensors where $n_{j} \ge 1$ and set $\phi(\bH_{j}) = \Re(\overline{z}_{j} z^{n_{j}})$. Then
  \begin{equation*}
    \Re\left(\overline{z}_{1}\dotsm \overline{z}_{p} z^{n_{1}+ \dotsb + n_{p}}\right) = 2^{p-1} \phi\left( (\bH_{1} \odot \dotsb \odot \bH_{p})^{\prime} \right).
  \end{equation*}
\end{thm}

\begin{thm}\label{thm:monomials-translation}
  Let $\bH_{j}\in \HT{n_{j}}$ be harmonic tensors where $n_{j} \ge 1$ and set $\phi(\bH_{j}) = \Re(\overline{z}_{j} z^{n_{j}})$. Let $N_{1} = n_{1} + \dotsb + n_{p}$, $N_{2} = n_{p+1} + \dotsb + n_{p+s}$ and assume that $N_{1} \le N_{2}$. Then
  \begin{multline*}
    \Re\left(z_{1}\dotsm z_{p} \zb_{p+1} \dotsm \zb_{p+s} z^{N_{2}-N_{1}}\right)
    \\
    = 2^{(p + s-1-N_{1})} \phi\left((\bH_{1} \odot \dotsb \odot \bH_{p})^{\prime} \rdots{N_{1}}(\bH_{p+1} \odot \dotsb \odot \bH_{p+s})^{\prime}\right),
  \end{multline*}
  and
  \begin{align*}
     & \Im\left(z_{1}\dotsm z_{p} \zb_{p+1} \dotsm \zb_{p+s} z^{N_{2}-N_{1}}\right)
    \\
     & \quad = \frac{2^{(p + s - 2N_{1})}(N_{1}+N_{2}-2)!}{(N_{1}-1)!(N_{2}-1)!} \phi \left(\tr^{(N_{1}-1)} \left((\bH_{1} \odot \dotsb \odot \bH_{p})^{\prime} \times (\bH_{p + 1} \odot \dotsb \odot \bH_{p + s})^{\prime}\right)\right)
    \\
     & \quad = 2^{(p + s-1-N_{1})} \phi\left((\bH_{1} \odot \dotsb \odot \bH_{p})^{\prime} \rdots{N_{1}}
    ([\id\times\bH_{p + 1}] \odot \dotsb \odot \bH_{p + s})^{\prime}]\right)
    \\
     & \quad = -2^{(p + s-1-N_{1})}
    \phi\left(([\id\times\bH_{1}] \odot \dotsb \odot \bH_{p})^{\prime} \rdots{N_{1}}
    (\bH_{p + 1} \odot \dotsb \odot \bH_{p + s})^{\prime}]\right).
  \end{align*}
\end{thm}

\begin{exa}\label{exa:tensors-invariants-elasticity-tensor}
  Let $\bC\in \Ela$ be a bidimensional elasticity tensor. Its  harmonic decomposition writes $\bC \simeq (\lambda, \mu, \bh, \bH)$, where $\bh\in \HT{2}$ and $\bH\in \HT{4}$ (see example \ref{ex:dec-harm-4-order} and remark \ref{rem:dec-harm-Ela}). Writing
  \begin{equation*}
    \phi(\bh) = \Re(\overline{z}_{2} z^{2}), \quad \text{and} \quad \phi(\bH) = \Re(\overline{z}_{4} z^4),
  \end{equation*}
  the translation of monomial invariants given in~\ref{ex:ElaSO2} and~\ref{ex:ElaO2}, namely
  \begin{equation*}
    z_{2}\overline{z_{2}}, \quad z_{4}\overline{z_{4}}, \quad \Re(z_{2}^{2}\overline{z_{4}}), \quad \Im(z_{2}^{2}\overline{z_{4}})
  \end{equation*}
  write
  \begin{align*}
    z_{2}\overline{z_{2}}          & = \frac{1}{2}(\bh \rdots{2} \bh)=\frac{1}{2} \bh:\bh,
    \\
    z_{4}\overline{z_{4}}          & = \frac{1}{2^{3} } (\bH \rdots{4} \bH)=\frac{1}{8} \bH::\bH,
    \\
    \Re(z_{2}^{2}\overline{z_{4}}) & = \frac{1}{4} ((\bh\odot\bh)' \rdots{4} \bH) = \frac{1}{4} (\bh : \bH : \bh),
  \end{align*}
  while there are several possibilities to translate $\Im(z_{2}^{2}\overline{z_{4}})$:
  \begin{align*}
    \Im(z_{2}^{2}\overline{z_{4}}) & = \frac{5}{8} \tr^{3} ((\bh\odot\bh)' \times \bH)
    \\
                                   & =\frac{1}{4} \big((\bh\odot\bh)' \rdots{4} (\id \times \bH)\big)=\frac{1}{4} \bh: (\id \times \bH):\bh
    \\
                                   & =-\frac{1}{4} \big(((\id\times\bh)\odot\bh)' \rdots{4} \bH\big)=-\frac{1}{4} \bh : \bH :(\id\times\bh).
  \end{align*}
  We deduce thus the following results.
  \begin{enumerate}
    \item A minimal $\SO(2)$-integrity basis for $\bC\in \Ela$ is
          \begin{equation*}
            \lambda, \quad \mu, \quad \bh:\bh, \quad  \bH::\bH, \quad \bh:\bH :\bh,
            \quad
            \bh : \bH :(\id\times\bh).
          \end{equation*}
    \item A minimal $\OO(2)$-integrity basis for $\bC\in \Ela$ is
          \begin{equation*}
            \lambda, \quad \mu, \quad \bh:\bh,\quad  \bH::\bH, \quad \bh:\bH :\bh.
          \end{equation*}
  \end{enumerate}
\end{exa}

\section{Minimal covariant bases for most common constitutive tensors and laws}
\label{sec:applications}

In this section, we illustrate the power of the method explained in this paper by providing a minimal integrity basis for an exhaustive list of constitutive tensors and laws which involve several tensors. More precisely, applying theorems \ref{thm:SO2-minimal-integrity-basis}, \ref{thm:O2-integrity-basis}, and \ref{thm:important-reduction}, and using the cleaning algorithm detailed in~\autoref{sec:algorithm}, we obtain explicit results in 2D for:
\begin{itemize}
  \item Third-order tensors with no index symmetry $\TT^{3}(\RR^{2})$, thus for third-order tensors with any kind of index symmetries (by theorem~\ref{thm:subspace-reduction}), and in particular for the piezoelectricity tensor $\bP\in \Piez$;
  \item Fourth-order tensors with no index symmetry $\TT^{4}(\RR^{2})$, thus for fourth-order tensors with any kind of index symmetries (by example \ref{ex:dec-harm-4-order} and
        theorem~\ref{thm:subspace-reduction}), and in particular for the elasticity tensor $\bC\in \Ela$ and the photoelasticity/Eshelby tensor $\bPi\in \Gel$;
  \item The complex viscoelasticity tensor, or more precisely its de-complexification, $\Ela \oplus \Ela$,
  \item The Hill elasto-plasticity constitutive equations;
  \item The linear piezoelectricity constitutive law, which involves three constitutive tensors, the dielectric permittivity tensor (of order two), the piezoelectricity tensor (of order three) and the elasticity tensor (of order four);
  \item Twelfth-order totally symmetric tensors $\bS\in\Sym^{12}(\RR^{2})$ \and thus for totally symmetric fabric tensors~\cite{Kan1984} of order 4, 6, 8 and 10 (by remark \ref{rem:Snmoins2k}).
\end{itemize}

In each case, we provide an harmonic decomposition and a minimal integrity basis of the covariant algebra (except for $\Sym^{12}(\RR^{2})$, for which we provide only an integrity basis for its invariant algebra due to its very large cardinal), and this both for $G=\OO(2)$ and $G=\SO(2)$.

\begin{rem}
  For each produced $\SO(2)$-integrity basis, the generators $I$ satisfy either $\sigma \star I = I$ or $\sigma \star I = -I$. In the first case, we shall refer to $I$ as an \emph{isotropic invariant} (since it is $\OO(2)$-invariant), and in the second case, as an \emph{hemitropic invariant}. Besides, the following notation has been adopted. For each factor $\HT{n}$ which occurs in the harmonic decomposition provided, the corresponding variable is written as $z_{n}$, if there is only one component $\HT{n}$ in this decomposition or $z_{na}$, $z_{nb}$, $z_{nc}$, \ldots if the component $\HT{n}$ appears with multiplicity.
\end{rem}

\subsection{Third-order tensors}
\label{subsec:third-order-tensors}

The harmonic decomposition of $\TT^{3}(\RR^{2})$ is the same for $\SO(2)$ and $\OO(2)$ and writes
\begin{equation*}
  \TT^{3}(\RR^{2}) \simeq 3 \HT{1} \oplus \HT{3}.
\end{equation*}
We will thus write $\bT = (z_{1a},z_{1b},z_{1c},z_{3})$, after the choice of an explicit harmonic decomposition (such as example~\ref{ex:T3-harmonic-decomposition}) and we have the following result.

\begin{thm}\label{thm:third-order-tensors}
  A minimal integrity basis for $\cov(\TT^{3}(\RR^{2}),\SO(2))$ consists in the 57 covariants (30 invariants) of~\autoref{tab:3-order-SO2-isotropic} and~\autoref{tab:3-order-SO2-hemitropic}. A minimal integrity basis for $\cov(\TT^{3}(\RR^{2}),\OO(2))$ is provided by the 31 covariants (17 invariants) of~\autoref{tab:3-order-SO2-isotropic}.
\end{thm}

\begin{table}[h]
  \caption{Isotropic covariants of $\TT^{3}(\RR^{2})$}
  \label{tab:3-order-SO2-isotropic}
  \begin{subtable}[T]{.5\linewidth}
  \centering
  \scriptsize
  \begin{tabular}{cccl}
    \toprule
    \# & order & degree & Formula                                                            \\
    \midrule
    1  & 0     & 2      & $z_{1a}\overline{z}_{1a}$                                         \\
    2  & 0     & 2      & $z_{1b}\overline{z}_{1b}$                                         \\
    3  & 0     & 2      & $z_{1c}\overline{z}_{1c}$                                         \\
    4  & 0     & 2      & $z_{3}\overline{z}_{3}$                                           \\
    5  & 0     & 2      & $\Re[z_{1a} \overline{z}_{1b}]$                                    \\
    6  & 0     & 2      & $\Re[z_{1a} \overline{z}_{1c}]$                                    \\
    7  & 0     & 2      & $\Re[z_{1b} \overline{z}_{1c}]$                                    \\
    8  & 0     & 4      & $\Re[z_{1a}^{3} \overline{z}_{3}]$                                 \\
    9  & 0     & 4      & $\Re[z_{1b}^{3} \overline{z}_{3}]$                                 \\
    10 & 0     & 4      & $\Re[z_{1c}^{3} \overline{z}_{3}]$                                 \\
    11 & 0     & 4      & $\Re[z_{1a}z_{1b}^{2} \overline{z}_{3}]$               \\
    12 & 0     & 4      & $\Re[z_{1b}z_{1c}^{2} \overline{z}_{3}]$               \\
    13 & 0     & 4      & $\Re[z_{1b}^{2}z_{1c} \overline{z}_{3}]$               \\
    14 & 0     & 4      & $\Re[z_{1a}^{2}z_{1c} \overline{z}_{3}]$               \\
    15 & 0     & 4      & $\Re[z_{1a}^{2}z_{1b} \overline{z}_{3}]$               \\
    16 & 0     & 4      & $\Re[z_{1a}z_{1c}^{2} \overline{z}_{3}]$               \\
    17 & 0     & 4      & $\Re[z_{1a}z_{1b}z_{1c} \overline{z}_{3}]$ \\
    \bottomrule
  \end{tabular}
\end{subtable}%
\begin{subtable}[T]{.5\linewidth}
  \centering
  \scriptsize
  \begin{tabular}{cccl}
    \toprule
    \# & order & degree & Formula                                                       \\
    \midrule
    18 & 1     & 1      & $\Re[\overline{z}_{1a}z]$                                    \\
    19 & 1     & 1      & $\Re[\overline{z}_{1b}z]$                                    \\
    20 & 1     & 1      & $\Re[\overline{z}_{1c}z]$                                    \\
    21 & 1     & 3      & $\Re[z_{1a}^{2} \overline{z}_{3}z]$               \\
    22 & 1     & 3      & $\Re[z_{1b}^{2} \overline{z}_{3}z]$               \\
    23 & 1     & 3      & $\Re[z_{1c}^{2} \overline{z}_{3}z]$               \\
    24 & 1     & 3      & $\Re[z_{1a}z_{1b} \overline{z}_{3}z]$ \\
    25 & 1     & 3      & $\Re[z_{1a}z_{1c} \overline{z}_{3}z]$ \\
    26 & 1     & 3      & $\Re[z_{1b}z_{1c} \overline{z}_{3}z]$ \\
    \midrule
    27 & 2     & 0      & $z \overline{z}$                                               \\
    28 & 2     & 2      & $\Re[z_{1a} \overline{z}_{3}z^{2}]$               \\
    29 & 2     & 2      & $\Re[z_{1b} \overline{z}_{3}z^{2}]$               \\
    30 & 2     & 2      & $\Re[z_{1c} \overline{z}_{3}z^{2}]$               \\
    \midrule
    31 & 3     & 1      & $\Re[\overline{z}_{3}z^{3}]$                                  \\
    \bottomrule
  \end{tabular}
\end{subtable}

\end{table}

\begin{table}[h]
  \caption{Hemitropic covariants of $\TT^{3}(\RR^{2})$}
  \label{tab:3-order-SO2-hemitropic}
  \setcounter{mtligne}{31}

\begin{subtable}[T]{.5\linewidth}
  \centering
  \scriptsize
  \begin{tabular}{>{\stepcounter{mtligne} \themtligne}cccl}
    \toprule
    \multicolumn{1}{c}{\#} & order & degree & Formula                                                            \\
    \midrule
                           & 0     & 2      & $\Im[z_{1a} \overline{z}_{1b}]$                                    \\
                           & 0     & 2      & $\Im[z_{1a} \overline{z}_{1c}]$                                    \\
                           & 0     & 2      & $\Im[z_{1b} \overline{z}_{1c}]$                                    \\
                           & 0     & 4      & $\Im[z_{1a}^{3} \overline{z}_{3}]$                                 \\
                           & 0     & 4      & $\Im[z_{1b}^{3} \overline{z}_{3}]$                                 \\
                           & 0     & 4      & $\Im[z_{1c}^{3} \overline{z}_{3}]$                                 \\
                           & 0     & 4      & $\Im[z_{1a} z_{1b}^{2} \overline{z}_{3}]$               \\
                           & 0     & 4      & $\Im[z_{1b} z_{1c}^{2} \overline{z}_{3}]$               \\
                           & 0     & 4      & $\Im[z_{1b}^{2} z_{1c} \overline{z}_{3}]$               \\
                           & 0     & 4      & $\Im[z_{1a}^{2} z_{1c} \overline{z}_{3}]$               \\
                           & 0     & 4      & $\Im[z_{1a}^{2} z_{1b} \overline{z}_{3}]$               \\
                           & 0     & 4      & $\Im[z_{1a} z_{1c}^{2} \overline{z}_{3}]$               \\
                           & 0     & 4      & $\Im[z_{1a} z_{1b} z_{1c} \overline{z}_{3}]$ \\
    \bottomrule
  \end{tabular}
\end{subtable}%
\begin{subtable}[T]{.5\linewidth}
  \centering
  \scriptsize
  \begin{tabular}{>{\stepcounter{mtligne} \themtligne}cccl}
    \toprule
    \multicolumn{1}{c}{\#} & order & degree & Formula                                                       \\
    \midrule
                           & 1     & 1      & $\Im[\overline{z}_{1a} z]$                                    \\
                           & 1     & 1      & $\Im[\overline{z}_{1b} z]$                                    \\
                           & 1     & 1      & $\Im[\overline{z}_{1c} z]$                                    \\
                           & 1     & 3      & $\Im[z_{1a}^{2} \overline{z}_{3} z]$               \\
                           & 1     & 3      & $\Im[z_{1b}^{2} \overline{z}_{3} z]$               \\
                           & 1     & 3      & $\Im[z_{1c}^{2} \overline{z}_{3} z]$               \\
                           & 1     & 3      & $\Im[z_{1a} z_{1b} \overline{z}_{3} z]$ \\
                           & 1     & 3      & $\Im[z_{1a} z_{1c} \overline{z}_{3} z]$ \\
                           & 1     & 3      & $\Im[z_{1b} z_{1c} \overline{z}_{3} z]$ \\
    \midrule
                           & 2     & 2      & $\Im[z_{1a} \overline{z}_{3} z^{2}]$               \\
                           & 2     & 2      & $\Im[z_{1b} \overline{z}_{3} z^{2}]$               \\
                           & 2     & 2      & $\Im[z_{1c} \overline{z}_{3} z^{2}]$               \\
    \midrule
                           & 3     & 1      & $\Im[\overline{z}_{3} z^{3}]$                                 \\
    \bottomrule
  \end{tabular}
\end{subtable}

\end{table}

An application of theorem~\ref{thm:third-order-tensors} concerns the bidimensional piezoelectricity third-order tensor $\bP$ (also denoted $\pieze$ in the IEEE Standard on Piezoelectricity, ANSI/IEEE 176 -1987), with index symmetry $P_{ijk}=P_{ikj}$. It relates the electric displacement $\vec D$ to the stress tensor $\bsigma\in \Sym^{2}(\RR^{2})$, at vanishing electric field, as~\cite{RD1999}
\begin{equation*}
  \vec D = \bP:\bsigma,
  \qquad
  D_{i}= P_{ijk} \sigma_{jk}.
\end{equation*}
The space of 2D piezoelectricity tensors, noted $\Piez$, has the same harmonic decomposition for both $\OO(2)$ and $\SO(2)$ and writes~\cite{GW2002a} $2\HT{1} \oplus \HT{3}$. By theorem~\ref{thm:subspace-reduction}, minimal integrity bases for $\cov(\Piez,G)$ ($G=\OO(2)$ or $G=\SO(2)$) are obtained by setting $z_{1c}=0$ in theorem \ref{thm:third-order-tensors} (with $z_{1a}, z_{1b}, z_{3}$ defined as in example \ref{ex:P3-harmonic-decomposition}). Using translation formulas of section~\ref{sec:monomials-versus-tensors-invariants}, we deduce the following corollary, which completes partial results obtained by Vannucci in~\cite{Van2007}.

\begin{cor}
  A minimal integrity basis of $\cov(\Piez,\SO(2))$, where we have set $\bP= (\vv, \ww, \bH)$, with $\vv, \ww \in \HT{1}$ and $\bH\in \HT{3}$, consists in the 30 covariants (13 invariants) of~\autoref{tab:Piez-MB-SO2}. A minimal integrity basis of $\cov(\Piez,\OO(2))$ consists in the 17 covariants (8 invariants) of~\autoref{tab:Piez-MB-O2}.
\end{cor}

\begin{table}[h]
  \caption{A minimal integrity basis for $\cov(\Piez,\SO(2))$}
  \label{tab:Piez-MB-SO2}
  \begin{subtable}[T]{.5\linewidth}
  \centering
  \scriptsize
  \begin{tabular}{cccl}
    \toprule
    \# & order & degree & Formula                                          \\
    \midrule
    1  & 0     & 2      & $\vv \cdot \vv$                                  \\
    2  & 0     & 2      & $\ww \cdot \ww$                                  \\
    3  & 0     & 2      & $\vv \cdot \ww$                                  \\
    4  & 0     & 2      & $\vv \times \ww$                                 \\
    5  & 0     & 2      & $\bH \3dots \bH$                                 \\
    6  & 0     & 4      & $(\vv \cdot \bH \cdot \vv) \cdot \vv$            \\
    7  & 0     & 4      & $(\vv \cdot \bH \cdot \vv)\cdot \ww$             \\
    8  & 0     & 4      & $(\ww \cdot \bH \cdot \ww)\cdot \ww$             \\
    9  & 0     & 4      & $(\ww \cdot \bH \cdot \ww)\cdot \vv$             \\
    10 & 0     & 4      & $(\vv\cdot  \bH \cdot \vv)\cdot (\id \times\vv)$ \\
    11 & 0     & 4      & $(\vv\cdot \bH \cdot \vv)\cdot (\id \times\ww)$  \\
    12 & 0     & 4      & $(\ww\cdot \bH \cdot \ww)\cdot (\id \times\vv)$  \\
    13 & 0     & 4      & $(\ww\cdot \bH \cdot \ww)\cdot (\id \times\ww)$  \\
    \bottomrule
  \end{tabular}
\end{subtable}%
\begin{subtable}[T]{.5\linewidth}
  \centering
  \scriptsize
  \begin{tabular}{cccl}
    \toprule
    \# & order & degree & Formula                               \\
    \midrule
    14 & 1     & 1      & $\vv$                                 \\
    15 & 1     & 1      & $\ww$                                 \\
    16 & 1     & 1      & $\id \times \vv$                      \\
    17 & 1     & 1      & $\id \times \ww$                      \\
    18 & 1     & 3      & $\vv \cdot \bH \cdot \vv$             \\
    19 & 1     & 3      & $\ww \cdot \bH \cdot \ww$             \\
    20 & 1     & 3      & $\vv \cdot \bH \cdot \ww$             \\
    21 & 1     & 3      & $\vv \cdot \bH \cdot (\id \times\vv)$ \\
    22 & 1     & 3      & $\ww \cdot \bH\cdot (\id \times\ww)$  \\
    23 & 1     & 3      & $\vv \cdot \bH \cdot (\id \times\ww)$ \\
    \midrule
    24 & 2     & 0      & $\id$                                 \\
    25 & 2     & 2      & $\bH \cdot \vv$                       \\
    26 & 2     & 2      & $\bH \cdot \ww$                       \\
    27 & 2     & 2      & $\bH \cdot (\id \times \vv)$          \\
    28 & 2     & 2      & $\bH \cdot (\id \times\ww)$           \\
    \midrule
    29 & 3     & 1      & $\bH$                                 \\
    30 & 3     & 1      & $\id \times \bH$                      \\
    \bottomrule
  \end{tabular}
\end{subtable}

\end{table}

\begin{table}[h]
  \caption{A minimal integrity basis for $\cov(\Piez,\OO(2))$}
  \label{tab:Piez-MB-O2}
  \begin{subtable}[T]{.5\linewidth}
  \centering
  \scriptsize
  \begin{tabular}{cccl}
    \toprule
    \# & order & degree & Formula                               \\
    \midrule
    1  & 0     & 2      & $\vv \cdot \vv$                       \\
    2  & 0     & 2      & $\ww \cdot \ww$                       \\
    3  & 0     & 2      & $\vv \cdot \ww$                       \\
    4  & 0     & 2      & $\bH \3dots \bH$                      \\
    5  & 0     & 4      & $(\vv \cdot \bH \cdot \vv) \cdot \vv$ \\
    6  & 0     & 4      & $(\vv \cdot \bH \cdot \vv)\cdot \ww$  \\
    7  & 0     & 4      & $(\ww \cdot \bH \cdot \ww)\cdot \ww$  \\
    8  & 0     & 4      & $(\ww \cdot \bH \cdot \ww)\cdot \vv$  \\
    \bottomrule
  \end{tabular}
\end{subtable}%
\begin{subtable}[T]{.5\linewidth}
  \centering
  \scriptsize
  \begin{tabular}{cccl}
    \toprule
    \# & order & degree & Formula                   \\
    \midrule
    9  & 1     & 1      & $\vv$                     \\
    10 & 1     & 1      & $\ww$                     \\
    11 & 1     & 3      & $\vv \cdot \bH \cdot \vv$ \\
    12 & 1     & 3      & $\ww \cdot \bH \cdot \ww$ \\
    13 & 1     & 3      & $\vv \cdot \bH \cdot \ww$ \\
    \midrule
    14 & 2     & 0      & $\id$                     \\
    15 & 2     & 2      & $\bH \cdot \vv$           \\
    16 & 2     & 2      & $\bH \cdot \ww$           \\
    \midrule
    17 & 3     & 1      & $\bH$                     \\
    \bottomrule
  \end{tabular}
\end{subtable}

\end{table}

Finally, we will complete our investigations of third-order tensors, in order to be fully exhaustive, by adding the space of totally symmetric tensors $\Sym^{3}(\RR^3)$. Its harmonic decomposition writes $\Sym^{3}(\RR^3) \simeq \HT{1}\oplus \HT{3}$ and corresponds to the subspace $\ww=0$ of $\Piez$.

\begin{cor}
  A minimal integrity basis for $\cov(\Sym^{3}(\RR^2),\SO(2))$, where we have set $\bS= (\vv, \bH)$, with $\vv \in \HT{1}$ and $\bH\in \HT{3}$, consists in the 13 covariants (4 invariants) of~\autoref{tab:S3-MB-SO2}. A minimal integrity basis for $\cov(\Sym^{3}(\RR^2),\OO(2))$ consists in the 8 covariants (3 invariants) of~\autoref{tab:S3-MB-O2}.
\end{cor}

\begin{table}[h]
  \caption{A minimal integrity basis for $\cov(\Sym^3(\RR^2),\SO(2))$}
  \label{tab:S3-MB-SO2}
  \setcounter{mtligne}{0}

\begin{subtable}[T]{.5\linewidth}
  \centering
  \scriptsize
  \begin{tabular}{>{\stepcounter{mtligne} \themtligne}cccl}
    \toprule
    \multicolumn{1}{c}{\#} & order & degree & Formula                             \\
    \midrule
      & 0     & 2      & $\vv \cdot \vv$                                  \\
      & 0     & 2      & $\bH \3dots \bH$                                 \\
      & 0     & 4      & $(\vv \cdot \bH \cdot \vv) \cdot \vv$            \\
     & 0     & 4      & $(\vv\cdot  \bH \cdot \vv)\cdot (\id \times\vv)$ \\
    \bottomrule
  \end{tabular}
\end{subtable}%
\begin{subtable}[T]{.5\linewidth}
  \centering
  \scriptsize
  \begin{tabular}{>{\stepcounter{mtligne} \themtligne}cccl}
    \toprule
    \multicolumn{1}{c}{\#} & order & degree & Formula                                           \\    
    \midrule
     & 1     & 1      & $\vv$                                 \\
     & 1     & 1      & $\id \times \vv$                      \\
     & 1     & 3      & $\vv \cdot \bH \cdot \vv$             \\
     & 1     & 3      & $\vv \cdot \bH \cdot (\id \times\vv)$ \\
    \midrule
     & 2     & 0      & $\id$                                 \\
     & 2     & 2      & $\bH \cdot \vv$                       \\
     & 2     & 2      & $\bH \cdot (\id \times \vv)$          \\
    \midrule
     & 3     & 1      & $\bH$                                 \\
     & 3     & 1      & $\id \times \bH$                      \\
    \bottomrule
  \end{tabular}
\end{subtable}

\end{table}

\begin{table}[h]
  \caption{A minimal integrity basis for $\cov(\Sym^3(\RR^2),\OO(2))$}
  \label{tab:S3-MB-O2}
  \setcounter{mtligne}{0}

\begin{subtable}[T]{.5\linewidth}
  \centering
  \scriptsize
  \begin{tabular}{>{\stepcounter{mtligne} \themtligne}cccl}
    \toprule
    \multicolumn{1}{c}{\#} & order & degree & Formula                                           \\    
     \midrule
      & 0     & 2      & $\vv \cdot \vv$                       \\
      & 0     & 2      & $\bH \3dots \bH$                      \\
      & 0     & 4      & $(\vv \cdot \bH \cdot \vv) \cdot \vv$ \\
    \bottomrule
  \end{tabular}
\end{subtable}%
\begin{subtable}[T]{.5\linewidth}
  \centering
  \scriptsize
  \begin{tabular}{>{\stepcounter{mtligne} \themtligne}cccl}
    \toprule
    \multicolumn{1}{c}{\#} & order & degree & Formula                                           \\    
     \midrule
      & 1     & 1      & $\vv$                     \\
     & 1     & 3      & $\vv \cdot \bH \cdot \vv$ \\
     \midrule
     & 2     & 0      & $\id$                     \\
     & 2     & 2      & $\bH \cdot \vv$           \\
    \midrule
     & 3     & 1      & $\bH$                     \\
    \bottomrule
  \end{tabular}
\end{subtable}

\end{table}


\subsection{Fourth-order tensors}
\label{subsec:fourth-order-tensors}

The harmonic decomposition of $\TT^{4}(\RR^{2})$ relative to $\SO(2)$ writes
\begin{equation*}
  6\HT{0} \oplus 4\HT{2} \oplus \HT{4}.
\end{equation*}
We will thus write
\begin{equation*}
  \bT = (\lambda_{1},\lambda_{2},\lambda_{3},\lambda_{4},\lambda_{5},\lambda_{6},z_{2a},z_{2b},z_{2c},z_{2d},z_{4}),
\end{equation*}
after the choice of an explicit harmonic decomposition. For $\OO(2)$, we get
\begin{equation*}
  3\HT{-1} \oplus 3\HT{0} \oplus 4\HT{2} \oplus \HT{4},
\end{equation*}
and we will have
\begin{equation*}
  \bT = (\xi_{1},\xi_{2},\xi_{3},\lambda_{1},\lambda_{2},\lambda_{3},z_{2a},z_{2b},z_{2c},z_{2d},z_{4}).
\end{equation*}
Note that in the present case, there are pseudo-scalars which reduce to additional $\HT{0}$ components when restricted to $\SO(2)$: $\xi_{1}=\lambda_{4}$, $\xi_{2}=\lambda_{5}$ and $\xi_{3}=\lambda_{6}$.

\begin{thm}\label{thm:fourth-order-tensors}
  A minimal integrity basis for $\cov(\TT^{4}(\RR^{2}),\SO(2))$ consists in the 62 covariants (43 invariants) of~\autoref{tab:4-order-SO2-isotropic} and~\autoref{tab:4-order-SO2-hemitropic}. A minimal integrity basis for $\cov(\TT^{4}(\RR^{2}),\OO(2))$ consists in the 115 covariants (78 invariants) of~\autoref{tab:4-order-SO2-isotropic} and \autoref{tab:4-order-SO2-hemitropic-products}.
\end{thm}

\begin{rem}
  In the case of $\OO(2)$, all products $\Im(\bm_{p})\Im(\bm_{q})$ disappear after cleaning but the products $\xi_{i} \Im(\bm_{l})$ remain, of course, and are listed in table~\autoref{tab:4-order-SO2-hemitropic-products}.
\end{rem}

\begin{table}[h]
  \caption{Isotropic covariants of $\TT^{4}(\RR^{2})$}
  \label{tab:4-order-SO2-isotropic}
  \setcounter{mtligne}{0}

\begin{subtable}[T]{.5\linewidth}
  \centering
  \scriptsize
  \begin{tabular}{>{\stepcounter{mtligne} \themtligne}cccl}
    \toprule
    \multicolumn{1}{c}{\#} & order & degree & Formula                             \\
    \midrule
                           & 0     & 1      & $\lambda_{1}$                       \\
                           & 0     & 1      & $\lambda_{2}$                       \\
                           & 0     & 1      & $\lambda_{3}$                       \\
                           & 0     & 2      & $z_{2a} \overline{z}_{2a}$          \\
                           & 0     & 2      & $z_{2b} \overline{z}_{2b}$          \\
                           & 0     & 2      & $z_{2c} \overline{z}_{2c}$          \\
                           & 0     & 2      & $z_{2d} \overline{z}_{2d}$          \\
                           & 0     & 2      & $z_{4} \overline{z}_{4}$            \\
                           & 0     & 2      & $ \Re[z_{2a} \overline{z}_{2b}]$    \\
                           & 0     & 2      & $\Re[z_{2a} \overline{z}_{2c}]$     \\
                           & 0     & 2      & $ \Re[z_{2a} \overline{z}_{2d}]$    \\
                           & 0     & 2      & $\Re[z_{2b} \overline{z}_{2c}]$     \\
                           & 0     & 2      & $ \Re[z_{2b} \overline{z}_{2d}]$    \\
                           & 0     & 2      & $\Re[z_{2c} \overline{z}_{2d}]$     \\
                           & 0     & 3      & $\Re[z_{2a}^{2} \overline{z}_{4}]$  \\
                           & 0     & 3      & $ \Re[z_{2b}^{2} \overline{z}_{4}]$ \\
                           & 0     & 3      & $ \Re[z_{2c}^{2} \overline{z}_{4}]$ \\
                           & 0     & 3      & $ \Re[z_{2d}^{2} \overline{z}_{4}]$ \\
    \bottomrule
  \end{tabular}
\end{subtable}%
\begin{subtable}[T]{.5\linewidth}
  \centering
  \scriptsize
  \begin{tabular}{>{\stepcounter{mtligne} \themtligne}cccl}
    \toprule
    \multicolumn{1}{c}{\#} & order & degree & Formula                                \\
    \midrule
                           & 0     & 3      & $ \Re[z_{2a} z_{2b} \overline{z}_{4}]$ \\
                           & 0     & 3      & $ \Re[z_{2a} z_{2c} \overline{z}_{4}]$ \\
                           & 0     & 3      & $\Re[z_{2a} z_{2d} \overline{z}_{4}]$  \\
                           & 0     & 3      & $\Re[z_{2b} z_{2c} \overline{z}_{4}]$  \\
                           & 0     & 3      & $ \Re[z_{2b} z_{2d} \overline{z}_{4}]$ \\
                           & 0     & 3      & $ \Re[z_{2c} z_{2d} \overline{z}_{4}]$ \\
    \midrule
                           & 2     & 0      & $z\overline{z}$                        \\
                           & 2     & 1      & $\Re[\overline{z}_{2a} z^{2}]$         \\
                           & 2     & 1      & $ \Re[\overline{z}_{2b} z^{2}]$        \\
                           & 2     & 1      & $ \Re[\overline{z}_{2c} z^{2}]$        \\
                           & 2     & 1      & $\Re[\overline{z}_{2d} z^{2}]$         \\
                           & 2     & 2      & $ \Re[z_{2a} \overline{z}_{4} z^{2}]$  \\
                           & 2     & 2      & $\Re[z_{2b} \overline{z}_{4} z^{2}]$   \\
                           & 2     & 2      & $ \Re[z_{2c} \overline{z}_{4} z^{2}]$  \\
                           & 2     & 2      & $ \Re[z_{2d} \overline{z}_{4} z^{2}]$  \\
    \midrule
                           & 4     & 1      & $\Re[\overline{z}_{4} z^{4}]$          \\
    \bottomrule
  \end{tabular}
\end{subtable}

\end{table}

\begin{table}[h]
  \caption{Hemitropic covariants of $\TT^{4}(\RR^{2})$}
  \label{tab:4-order-SO2-hemitropic}
  \setcounter{mtligne}{34}

\begin{subtable}[T]{.5\linewidth}
  \centering
  \scriptsize
  \begin{tabular}{>{\stepcounter{mtligne} \themtligne}cccl}
    \toprule
    \multicolumn{1}{c}{\#} & order & degree & Formula                                           \\
    \midrule
                           & 0     & 1      & $\xi_{1}=\lambda_{4}$                             \\
                           & 0     & 1      & $\xi_{2}=\lambda_{5}$                             \\
                           & 0     & 1      & $\xi_{3}=\lambda_{6}$                             \\
                           & 0     & 2      & $\Im[z_{2a} \overline{z}_{2b}]$                   \\
                           & 0     & 2      & $\Im[z_{2a} \overline{z}_{2c}]$                   \\
                           & 0     & 2      & $ \Im[z_{2a} \overline{z}_{2d}]$                  \\
                           & 0     & 2      & $\Im[z_{2b} \overline{z}_{2c}]$                   \\
                           & 0     & 2      & $ \Im[z_{2b} \overline{z}_{2d}]$                  \\
                           & 0     & 2      & $\Im[z_{2c} \overline{z}_{2d}]$                   \\
                           & 0     & 3      & $\Im[z_{2a}^{2} \overline{z}_{4}]$                \\
                           & 0     & 3      & $\Im[z_{2b}^{2} \overline{z}_{4}]$                \\
                           & 0     & 3      & $\Im[z_{2c}^{2} \overline{z}_{4}]$                \\
                           & 0     & 3      & $\Im[z_{2d}^{2} \overline{z}_{4}]$                \\
                           & 0     & 3      & $ \Im[z_{2a} z_{2b} \overline{z}_{4}]$ \\
                           & 0     & 3      & $ \Im[z_{2a} z_{2c} \overline{z}_{4}]$ \\
    \bottomrule
  \end{tabular}
\end{subtable}%
\begin{subtable}[T]{.5\linewidth}
  \centering
  \scriptsize
  \begin{tabular}{>{\stepcounter{mtligne} \themtligne}cccl}
    \toprule
    \multicolumn{1}{c}{\#} & order & degree & Formula                                           \\
    \midrule
                           & 0     & 3      & $\Im[z_{2a} z_{2d} \overline{z}_{4}]$  \\
                           & 0     & 3      & $\Im[z_{2b} z_{2c} \overline{z}_{4}]$  \\
                           & 0     & 3      & $ \Im[z_{2b} z_{2d} \overline{z}_{4}]$ \\
                           & 0     & 3      & $\Im[z_{2c} z_{2d} \overline{z}_{4}]$  \\
    \midrule
                           & 2     & 1      & $\Im[\overline{z}_{2a} z^{2}]$                    \\
                           & 2     & 1      & $ \Im[\overline{z}_{2b} z^{2}]$                   \\
                           & 2     & 1      & $ \Im[\overline{z}_{2c} z^{2}]$                   \\
                           & 2     & 1      & $\Im[\overline{z}_{2d} z^{2}]$                    \\
                           & 2     & 2      & $\Im[z_{2a} \overline{z}_{4} z^{2}]$   \\
                           & 2     & 2      & $\Im[z_{2b} \overline{z}_{4} z^{2}]$   \\
                           & 2     & 2      & $ \Im[z_{2c} \overline{z}_{4} z^{2}]$  \\
                           & 2     & 2      & $\Im[z_{2d} \overline{z}_{4} z^{2}]$   \\
    \midrule
                           & 4     & 1      & $\Im[\overline{z}_{4} z^{4}]$                     \\
    \bottomrule
  \end{tabular}
\end{subtable}

\end{table}

\begin{table}[h]
  \caption{Isotropic products of hemitropic covariants of $\TT^{4}(\RR^{2})$ ($i=1,2,3$)}
  \label{tab:4-order-SO2-hemitropic-products}
  \begin{subtable}[T]{.5\linewidth}
  \centering
  \scriptsize
  \begin{tabular}{lccl}
    \toprule
    \#         & order & degree & Formula                                                   \\
    \midrule
    35         & 0     & 2      & $\xi_{1}^{\,2}$                                         \\
    36         & 0     & 2      & $\xi_{2}^{\,2}$                                         \\
    37         & 0     & 2      & $\xi_{3}^{\,2}$                                         \\
    38         & 0     & 2      & $\xi_{1} \xi_{2}$                                         \\
    39         & 0     & 2      & $\xi_{1} \xi_{3}$                                         \\
    40         & 0     & 2      & $\xi_{2} \xi_{3}$                                         \\
    41, 42, 43 & 0     & 3      & $\xi_{i} \Im[z_{2a} \overline{z}_{2b}]$                   \\
     44, 45, 46 & 0     & 3      & $\xi_{i} \Im[z_{2a} \overline{z}_{2c}]$                   \\
    47, 48, 49 & 0     & 3      & $\xi_{i} \Im[z_{2a} \overline{z}_{2d}]$                   \\
    50, 51, 52  & 0     & 3      & $\xi_{i} \Im[z_{2b} \overline{z}_{2c}]$                   \\
   53, 54, 55 & 0     & 3      & $\xi_{i}  \Im[z_{2b} \overline{z}_{2d}]$                  \\
   56, 57, 58  & 0     & 3      & $\xi_{i}  \Im[z_{2c} \overline{z}_{2d}]$                  \\
   59, 60, 61  & 0     & 4      & $ \xi_{i} \Im[z_{2a}^{2} \overline{z}_{4}]$               \\
    62, 63, 64 & 0     & 4      & $ \xi_{i} \Im[z_{2b}^{2} \overline{z}_{4}]$               \\
    65, 66, 67 & 0     & 4      & $\xi_{i} \Im[z_{2c}^{2} \overline{z}_{4}]$                \\
     68, 69, 70& 0     & 4      & $\xi_{i}  \Im[z_{2d}^{2} \overline{z}_{4}]$               \\
    \bottomrule
  \end{tabular}
\end{subtable}%
\begin{subtable}[T]{.5\linewidth}
  \centering
  \scriptsize
  \begin{tabular}{cccl}
    \toprule
    \#            & order & degree & Formula                                                   \\
    \midrule
    71, 72, 73 & 0     & 4      & $\xi_{i}  \Im[z_{2a} z_{2b} \overline{z}_{4}]$ \\
     74, 75, 76 & 0     & 4      & $\xi_{i} \Im[z_{2a} z_{2c} \overline{z}_{4}]$  \\
     77, 78, 79  & 0     & 4      & $\xi_{i} \Im[z_{2a} z_{2d} \overline{z}_{4}]$  \\
      80, 81, 82  & 0     & 4      & $\xi_{i} \Im[z_{2b} z_{2c} \overline{z}_{4}],$ \\
       83, 84, 85   & 0     & 4      & $ \xi_{i} \Im[z_{2b} z_{2d} \overline{z}_{4}]$ \\
    86, 87, 88  & 0     & 4     & $\xi_{i} \Im[z_{2c} z_{2d} \overline{z}_{4}]$
    \\
    \midrule
     89, 90, 91   & 2     & 2      & $ \xi_{i} \Im(\overline{z}_{2a}z^2)$                            \\
     92, 93, 94   & 2     & 2      & $\xi_{i} \Im(\overline{z}_{2b}z^2)$                             \\
     95, 96, 97  & 2     & 2      & $\xi_{i} \Im(\overline{z}_{2c}z^2)$                             \\
     98, 99, 100   & 2     & 2      & $\xi_{i} \Im(\overline{z}_{2d}z^2)$                             \\
    101, 102, 103 & 2     & 3      & $\xi_{i} \Im[z_{2a} \overline{z}_{4} z^{2}]$   \\
    104, 105, 106 & 2     & 3      & $\xi_{i} \Im[z_{2b} \overline{z}_{4} z^{2}]$   \\
    107, 108, 109 & 2     & 3      & $\xi_{i} \Im[z_{2c} \overline{z}_{4} z^{2}]$   \\
    110, 111, 112 & 2     & 3      & $\xi_{i} \Im[z_{2d} \overline{z}_{4} z^{2}]$   \\
    \midrule
    113, 114, 115 & 4     & 2      & $\xi_{i} \Im(\overline{z}_{4}z^4)$                              \\
    \bottomrule
  \end{tabular}
\end{subtable}

\end{table}

\subsubsection*{Photoelasticity tensor}

The 2D photoelasticity tensor $\bPi$ \cite{CF1957}, like the 2D Eshelby tensor, has the following index symmetry $\Pi_{ijkl}=\Pi_{jikl}=\Pi_{ijlk}$. The corresponding tensor space, noted $\Gel$ in \cite{FV1997}, has the following harmonic decomposition under $\SO(2)$
\begin{equation*}
  3\HT{0} \oplus 2\HT{2} \oplus \HT{4},
\end{equation*}
and
\begin{equation*}
  \HT{-1}\oplus 2\HT{0} \oplus 2\HT{2} \oplus \HT{4},
\end{equation*}
under $\OO(2)$. It corresponds to the subspace of $\TT^4(\RR^2)$, where
\begin{equation*}
  \lambda_{1}=\lambda, \quad \lambda_{2}=\mu, \quad \xi_{1}=\xi, \quad \lambda_{3}=\xi_{2}=\xi_{3}=0, \quad z_{2c}=z_{2d}=0.
\end{equation*}
In the following corollary, we provide for it a minimal integrity basis of its covariant algebra and correct, by the way, an error in~\cite{MZC2019}. Indeed, the $\OO(2)$-integrity basis of its invariant algebra provided there, is of cardinal 10, and omits all the irreducible invariants of~\autoref{tab:Photo-SO2-hemitropic-products}.

\begin{cor}
  A minimal integrity basis for $\cov(\Gel,\SO(2)))$, where we have set
  \begin{equation*}
    \bPi= (\xi, \lambda, \mu, \bh_{1}, \bh_{2}, \bH),
  \end{equation*}
  with $\xi \in \HT{-1}$, $\lambda, \mu \in \HT{0}$, $\bh_{1}, \bh_{2} \in \HT{2}$, $\bH\in \HT{4}$, consists in the 25 covariants of~\autoref{tab:Photo-SO2-isotropic} and \autoref{tab:Photo-SO2-hemitropic} (14 invariants). A minimal integrity basis for $\cov(\Gel,\OO(2)))$ consists in the 25 covariants (14 invariants) of~\autoref{tab:Photo-SO2-isotropic} and \autoref{tab:Photo-SO2-hemitropic-products}.
\end{cor}

\begin{rem}
  In \cite{MZC2019}, the authors have forgotten the 4 products $\xi_{i} \Im \bm_{l}$ of~\autoref{tab:Photo-SO2-hemitropic-products} which cannot be removed from any minimal basis of the $\OO(2)$-invariant algebra, by theorem~\ref{thm:O2-integrity-basis}.
\end{rem}

\begin{table}[h]
  \caption{Isotropic covariants of the photoelasticity tensor}
  \label{tab:Photo-SO2-isotropic}
  \setcounter{mtligne}{0}

\begin{subtable}[T]{.5\linewidth}
  \centering
  \scriptsize
  \begin{tabular}{>{\stepcounter{mtligne} \themtligne}cccll}
    \toprule
    \multicolumn{1}{c}{\#} & order & degree & Formula                                           & Formula           \\
    \midrule
                           & 0     & 1      & $\lambda$                                         & $\lambda$         \\
                           & 0     & 1      & $\mu$                                             & $\mu$             \\
                           & 0     & 2      & $z_{2a} \overline{z}_{2a}$                        & $\bh_1:\bh_1$     \\
                           & 0     & 2      & $z_{2b} \overline{z}_{2b}$                        & $\bh_2:\bh_2$     \\
                           & 0     & 2      & $z_{4} \overline{z}_{4}$                          & $\bH::\bH$        \\
                           & 0     & 2      & $ \Re[z_{2a} \overline{z}_{2b}]$                  & $\bh_1:\bh_2$     \\
                           & 0     & 3      & $\Re[z_{2a}^{2} \overline{z}_{4}]$                & $\bh_1:\bH:\bh_1$ \\
                           & 0     & 3      & $ \Re[z_{2b}^{2} \overline{z}_{4}]$               & $\bh_2:\bH:\bh_2$ \\
                           & 0     & 3      & $ \Re[z_{2a} z_{2b} \overline{z}_{4}]$ & $\bh_1:\bH:\bh_2$ \\
    \bottomrule
  \end{tabular}
\end{subtable}%
\begin{subtable}[T]{.5\linewidth}
  \centering
  \scriptsize
  \begin{tabular}{>{\stepcounter{mtligne} \themtligne}cccll}
    \toprule
    \multicolumn{1}{c}{\#} & order & degree & Formula                                          & Formula     \\
    \midrule
                           & 2     & 0      & $z\overline{z}$                                  & $\id$       \\
                           & 2     & 1      & $\Re[\overline{z}_{2a} z^{2}]$                   & $\bh_1$     \\
                           & 2     & 1      & $ \Re[\overline{z}_{2b} z^{2}]$                  & $\bh_2$     \\
                           & 2     & 2      & $ \Re[z_{2a} \overline{z}_{4} z^{2}]$ & $\bH:\bh_1$ \\
                           & 2     & 2      & $\Re[z_{2b} \overline{z}_{4} z^{2}]$  & $\bH:\bh_2$ \\
    \midrule
                           & 4     & 1      & $\Re[\overline{z}_{4} z^{4}]$                    & $\bH$       \\
    \bottomrule
  \end{tabular}
\end{subtable}

\end{table}

\begin{table}[h]
  \caption{Hemitropic covariants of the photoelasticity tensor}
  \label{tab:Photo-SO2-hemitropic}
  \setcounter{mtligne}{15}

\begin{subtable}[T]{.5\linewidth}
  \centering
  \scriptsize
  \begin{tabular}{>{\stepcounter{mtligne} \themtligne}cccll}
    \toprule
    \multicolumn{1}{c}{\#} & order & degree & Formula                                & Formula                       \\
    \midrule
                           & 0     & 1      & $\xi$                                  & $\xi$                         \\
                           & 0     & 2      & $ \Im[z_{2a} \overline{z}_{2b}]$       & $\bh_1:(\id \times\bh_2)$     \\
                           & 0     & 3      & $\Im[z_{2a}^{2} \overline{z}_{4}]$     & $\bh_1:\bH:(\id \times\bh_1)$ \\
                           & 0     & 3      & $ \Im[z_{2b}^{2} \overline{z}_{4}]$    & $\bh_2:\bH:(\id \times\bh_2)$ \\
                           & 0     & 3      & $ \Im[z_{2a} z_{2b} \overline{z}_{4}]$ & $\bh_1:\bH:(\id \times\bh_2)$ \\
    \bottomrule
  \end{tabular}
\end{subtable}%
\begin{subtable}[T]{.5\linewidth}
  \centering
  \scriptsize
  \begin{tabular}{>{\stepcounter{mtligne} \themtligne}cccll}
    \toprule
    \multicolumn{1}{c}{\#} & order & degree & Formula                               & Formula                 \\
    \midrule
                           & 2     & 1      & $\Im[\overline{z}_{2a} z^{2}]$        & $(\id \times\bh_1)$     \\
                           & 2     & 1      & $ \Im[\overline{z}_{2b} z^{2}]$       & $(\id \times\bh_2)$     \\
                           & 2     & 2      & $ \Im[z_{2a} \overline{z}_{4} z^{2}]$ & $\bH:(\id \times\bh_1)$ \\
                           & 2     & 2      & $\Im[z_{2b} \overline{z}_{4} z^{2}]$  & $\bH:(\id \times\bh_2)$ \\
    \midrule
                           & 4     & 1      & $\Im[\overline{z}_{4} z^{4}]$         & $(\id \times\bH)$       \\
    \bottomrule
  \end{tabular}
\end{subtable}

\end{table}

\begin{table}[h]
  \caption{Isotropic products of hemitropic covariants of the photoelasticity tensor}
  \label{tab:Photo-SO2-hemitropic-products}
  \setcounter{mtligne}{15}

\begin{subtable}[T]{.5\linewidth}
  \centering
  \scriptsize
  \begin{tabular}{>{\stepcounter{mtligne} \themtligne}cccll}
    \toprule
    \multicolumn{1}{c}{\#} & order & degree & Formula                                    & Formula                               \\
    \midrule
                           & $0$   & $2$    & $\xi^{2}$                                  & $\xi^{2}$                             \\
                           & 0     & 3      & $\xi \Im[z_{2a} \overline{z}_{2b}]$        & $ \xi \, \bh_1:(\id\times \bh_2)$     \\
                           & 0     & 4      & $ \xi \Im[z_{2a}^{2} \overline{z}_{4}]$    & $\xi\, \bh_1 :\bH:(\id\times \bh_1)$  \\
                           & 0     & 4      & $ \xi \Im[z_{2b}^{2} \overline{z}_{4}]$    & $\xi\, \bh_2 :\bH:(\id\times \bh_2)$  \\
                           & 0     & 4      & $\xi  \Im[z_{2a} z_{2b} \overline{z}_{4}]$ & $\xi \, \bh_1 :\bH:(\id\times \bh_2)$ \\
    \bottomrule
  \end{tabular}
\end{subtable}%
\begin{subtable}[T]{.5\linewidth}
  \centering
  \scriptsize
  \begin{tabular}{>{\stepcounter{mtligne} \themtligne}cccll}
    \toprule
    \multicolumn{1}{c}{\#} & order & degree & Formula                                  & Formula                           \\
    \midrule
                           & 2     & 2      & $ \xi \Im(\overline{z}_{2a}z^2)$         & $\xi \, \id \times \bh_1$         \\
                           & 2     & 2      & $\xi \Im(\overline{z}_{2b}z^2)$          & $\xi \, \id \times \bh_2$         \\
                           & 2     & 3      & $\xi \Im[z_{2a} \overline{z}_{4} z^{2}]$ & $\xi \,  \bH:(\id \times \bh_1)$  \\
                           & 2     & 3      & $\xi \Im[z_{2b} \overline{z}_{4} z^{2}]$ & $\xi \,   \bH:(\id \times \bh_2)$ \\
    \midrule
                           & 4     & 2      & $\xi \Im(\overline{z}_{4}z^4)$           & $\xi \,  \id \times \bH$          \\
    \bottomrule
  \end{tabular}
\end{subtable}

\end{table}

\subsubsection*{Elasticity tensor}

The 2D elasticity tensor $\bC$ has the index symmetry $C_{ijkl}=C_{jikl}=C_{ijlk}=C_{klij}$. The corresponding tensor space $\Ela$ has the following harmonic decomposition
\begin{equation*}
  2\HT{0} \oplus \HT{2}  \oplus \HT{4},
\end{equation*}
both for $\OO(2)$ and $\SO(2)$. It corresponds to the subspace of $\TT^4(\RR^2)$, where
\begin{equation*}
  \lambda_{3}=\xi_{1}=\xi_{2}=\xi_{3}=0, \qquad  z_{2b}=z_{2c}=z_{2d}=0,
\end{equation*}
and to the subspace of $\Gel$, where
\begin{equation*}
  \xi=0, \qquad \bh_{1}=\bh, \qquad \bh_{2}=0.
\end{equation*}
The following corollary completes the already known integrity basis of the invariant algebra of $\Ela$ (see examples \ref{ex:ElaSO2}, \ref{ex:ElaO2} and~\ref{exa:tensors-invariants-elasticity-tensor}) into a minimal integrity basis of its covariant algebra.

\begin{cor}
  A minimal integrity basis for $\cov(\Ela,\SO(2)))$, where we have set $\bC= (\lambda, \mu, \bh, \bH)$, with $\lambda, \mu \in \HT{0}$, $\bh\in \HT{2}$, $\bH\in \HT{4}$, consists in the 13 covariants (6 invariants) of~\autoref{tab:Ela}~(A) and \autoref{tab:Ela}~(B). A minimal integrity basis for $\cov(\Ela,\OO(2))$ consists in the 9 covariants (5 invariants) of~\autoref{tab:Ela}~(A).
\end{cor}

\begin{table}[h]
  \caption{Covariants of $\Ela$}
  \label{tab:Ela}
  \setcounter{mtligne}{0}

\begin{subtable}[T]{.5\linewidth}
  \caption{Isotropic covariants}
  \centering
  \scriptsize
  \begin{tabular}{>{\stepcounter{mtligne} \themtligne}cccl}
    \toprule
    \multicolumn{1}{c}{\#} & order & degree & Formula        \\
    \midrule
                           & 0     & 1      & $\lambda$      \\
                           & 0     & 1      & $\mu$          \\
                           & 0     & 2      & $\bh:\bh$      \\
                           & 0     & 2      & $ \bH::\bH$    \\
                           & 0     & 3      & $\bh:\bH :\bh$ \\
    \midrule
                           & 2     & 0      & $\id$          \\
                           & 2     & 1      & $\bh$          \\
                           & 2     & 2      & $\bH:\bh$      \\
    \midrule
                           & 4     & 1      & $\bH$          \\
    \bottomrule
  \end{tabular}
\end{subtable}%
\begin{subtable}[T]{.5\linewidth}
  \caption{Hemitropic covariants}
  \centering
  \scriptsize
  \begin{tabular}{>{\stepcounter{mtligne} \themtligne}cccl}
    \toprule
    \multicolumn{1}{c}{\#} & order & degree & Formula                    \\
    \midrule
                           & 0     & 3      & $\bh:\bH :(\id \times\bh)$ \\
    \midrule
                           & 2     & 1      & $\id \times \bh$           \\
                           & 2     & 2      & $ \bH :(\id \times\bh)$    \\
    \midrule
                           & 4     & 1      & $\id \times \bH$           \\
    \bottomrule
  \end{tabular}
\end{subtable}%
\end{table}


\subsection{Viscoelasticity law and Hill elasto-plasticity}
\label{subsec:viscoelasticity-law}

In linear viscoelasticity, the application of a periodic strain tensor at frequency $f$, seen as the imaginary part of $\bepsilon =\bepsilon_{a} \exp(2 i \pi f t)$, generates a periodic stress tensor which is the imaginary part of $\bsigma =\bsigma_{a} \exp(i(2\pi f t+\varphi))$; $\bepsilon_{a}$ and $\bsigma_{a}$ are the assumed constant strain and stress amplitude (symmetric)
tensors and $\varphi$ is the phase shift. A frequency dependent anisotropic viscoelasticity behaviour can then be formulated as
\begin{equation}\label{eq:ViscoElasLaw}
  \bsigma = \bC^{*}:\bepsilon,
\end{equation}
where $\bC^{*}=\bC^{*}(f)=\bC_{1}+ i\, \bC_{2}$ (with $\bC_{1}, \bC_{2}\in \Ela$) is the \emph{complex viscoelasticity tensor}. For the purpose we are concerned in, we will however still consider this representation as a real representation of either $\SO(2)$ or $\OO(2)$ and represent it as the (de-complexified) vector space $\VV=\Ela \oplus \Ela$.

On the other hand, Hill elasto-plasticity constitutive equations~\cite{Hil1948} can be summarized into the linear elasticity law
\begin{equation*}
  \bsigma=\bC:(\bepsilon-\bepsilon^p),
\end{equation*}
where $\bC\in \Ela$ is the elasticity tensor, $\bepsilon^p\in \Sym^2(\RR^2)$ is the plastic strain tensor, and into the (plasticity) yield criterion
\begin{equation*}
  f = \bsigma:\bP^H:\bsigma-R(p) \leq 0,
\end{equation*}
where $\bP^H\in \Ela$ is the Hill fourth-order tensor and $R$ is the hardening function ($p$ being the so-called accumulated plastic strain). The evolution laws are obtained by generalized normality (see~\cite{LC1985}). Hill elasto-plasticity law is also represented by a pair $(\bC_1,\bC_2)$ of tensors of the elasticity type, but this time, with $\bC_1=\bC$, and $\bC_2=\bP^H$.

The harmonic decomposition of $\VV=\Ela \oplus \Ela$, which is the same for $\SO(2)$ and $\OO(2)$, is
\begin{equation*}
  4\HT{0} \oplus  2\HT{2} \oplus 2\HT{4},
\end{equation*}
and we will write
\begin{equation*}
  (\bC_{1},\bC_{2}) = (\lambda_{1},\lambda_{2},\mu_{1}, \mu_{2},z_{2a},z_{2b},z_{4a},z_{4b}) = (\lambda_{1},\lambda_{2},\mu_{1}, \mu_{2},\bh_{1},\bh_{2},\bH_{1},\bH_{2}).
\end{equation*}

\begin{thm}\label{thm:ViscoEla}
  A minimal integrity basis for $\cov(\Ela \oplus \Ela,\SO(2))$ consists in the 41 covariants (24 invariants) of~\autoref{tab:ViscoEla-SO2-isotropic} and~\autoref{tab:ViscoEla-SO2-hemitropic}. A minimal integrity basis for $\cov(\Ela \oplus \Ela,\OO(2))$ consists in the 28 covariants (17 invariants) of~\autoref{tab:ViscoEla-SO2-isotropic} and~\autoref{tab:ViscoEla-SO2-hemitropic-products}.
\end{thm}

\begin{rem}
  Note that, in the minimal covariant basis for $\cov(\Ela \oplus \Ela,\OO(2))$, it remains 3 products $\Im(\bm_{p})\Im(\bm_{q})$ (see~\autoref{tab:ViscoEla-SO2-hemitropic-products}), which cannot be eliminated after cleaning.
\end{rem}

\begin{table}[h]
  \caption{Isotropic covariants of $\VV=\Ela \oplus \Ela$}
  \label{tab:ViscoEla-SO2-isotropic}
  \setcounter{mtligne}{0}

\begin{subtable}[T]{.5\linewidth}
  \centering
  \scriptsize
  \begin{tabular}{>{\stepcounter{mtligne} \themtligne}cccll}
    \toprule
    \multicolumn{1}{c}{\#} & order & degree & Formula                                & Formula                     \\
    \midrule
                           & 0     & 1      & $\lambda_1$                            & $\lambda_1$                 \\
                           & 0     & 1      & $\mu_1$                                & $\mu_1$                     \\
                           & 0     & 1      & $\lambda_2$                            & $\lambda_2$                 \\
                           & 0     & 1      & $\mu_2$                                & $\mu_2$                     \\
                           & 0     & 2      & $ z_{2a} \overline{z}_{2a}$            & $\bh_{1}:\bh_{1}$           \\
                           & 0     & 2      & $z_{2b} \overline{z}_{2b}$             & $\bh_{2}:\bh_{2}$           \\
                           & 0     & 2      & $z_{4a} \overline{z}_{4a}$             & $ \bH_{1}::\bH_{1}$         \\
                           & 0     & 2      & $z_{4b} \overline{z}_{4b}$             & $\bH_{2}::\bH_{2}$          \\
                           & 0     & 2      & $ \Re[z_{2a} \overline{z}_{2b}]$       & $ \bh_{1}:\bh_{2} $         \\
                           & 0     & 2      & $ \Re[z_{4a} \overline{z}_{4b}]$       & $\bH_{1}::\bH_{2}$          \\
                           & 0     & 3      & $ \Re[z_{2a}^{2} \overline{z}_{4a}]$   & $ \bh_{1}:\bH_{1} :\bh_{1}$ \\
                           & 0     & 3      & $ \Re[z_{2b}^{2} \overline{z}_{4a}]$   & $\bh_{2}:\bH_{1} :\bh_{2}$  \\
                           & 0     & 3      & $\Re[z_{2a}^{2} \overline{z}_{4b}]$    & $\bh_{1}:\bH_{2} :\bh_{1}$  \\
                           & 0     & 3      & $\Re[z_{2b}^{2} \overline{z}_{4b}]$    & $\bh_{2}:\bH_{2} :\bh_{2}$  \\
                           & 0     & 3      & $\Re[z_{2a} z_{2b} \overline{z}_{4a}]$ & $\bh_{1}:\bH_{1} :\bh_{2}$  \\
                           & 0     & 3      & $\Re[z_{2a} z_{2b} \overline{z}_{4b}]$ & $\bh_{1}:\bH_{2} :\bh_{2}$  \\
    \bottomrule
  \end{tabular}
\end{subtable}%
\begin{subtable}[T]{.5\linewidth}
  \centering
  \scriptsize
  \begin{tabular}{>{\stepcounter{mtligne} \themtligne}cccll}
    \toprule
    \multicolumn{1}{c}{\#} & order & degree & Formula                                & Formula           \\
    \midrule
                           & 2     & 0      & $z \overline{z}$                       & $\id$             \\
                           & 2     & 1      & $\Re[\overline{z}_{2a} z^{2}]$         & $\bh_{1}$         \\
                           & 2     & 1      & $ \Re[\overline{z}_{2b} z^{2}]$        & $ \bh_{2}$        \\
                           & 2     & 2      & $ \Re[z_{2a} \overline{z}_{4a} z^{2}]$ & $\bH_{1}:\bh_{1}$ \\
                           & 2     & 2      & $\Re[z_{2a} \overline{z}_{4b} z^{2}]$  & $\bH_{2}:\bh_{1}$ \\
                           & 2     & 2      & $\Re[z_{2b} \overline{z}_{4a} z^{2}]$  & $\bH_{1}:\bh_{2}$ \\
                           & 2     & 2      & $ \Re[z_{2b} \overline{z}_{4b} z^{2}]$ & $\bH_{2}:\bh_{2}$ \\
    \midrule
                           & 4     & 1      & $\Re[\overline{z}_{4a} z^{4}]$         & $\bH_1$           \\
                           & 4     & 1      & $\Re[\overline{z}_{4b} z^{4}]$         & $\bH_2$           \\
    \bottomrule
  \end{tabular}
\end{subtable}

\end{table}

\begin{table}[h]
  \caption{Hemitropic covariants of $\VV=\Ela \oplus \Ela$}
  \label{tab:ViscoEla-SO2-hemitropic}
  \setcounter{mtligne}{25}

\begin{subtable}[T]{.5\linewidth}
  \centering
  \scriptsize
  \begin{tabular}{>{\stepcounter{mtligne} \themtligne}cccll}
    \toprule
    \multicolumn{1}{c}{\#} & order & degree & Formula                                & Formula                            \\
    \midrule
                           & 0     & 2      & $\Im[z_{2a} \overline{z}_{2b}]$        & $\bh_1 : (\id \times \bh_2)$       \\
                           & 0     & 2      & $ \Im[z_{4a} \overline{z}_{4b}]$       & $\bH_1 : (\id \times \bH_2)$       \\
                           & 0     & 3      & $\Im[z_{2a}^{2} \overline{z}_{4a}]$    & $\bh_1:\bH_1 : (\id \times \bh_1)$ \\
                           & 0     & 3      & $ \Im[z_{2b}^{2} \overline{z}_{4a}]$   & $\bh_2:\bH_1 : (\id \times \bh_2)$ \\
                           & 0     & 3      & $ \Im[z_{2a}^{2} \overline{z}_{4b}]$   & $\bh_1:\bH_2 : (\id \times \bh_1)$ \\
                           & 0     & 3      & $ \Im[z_{2b}^{2} \overline{z}_{4b}]$   & $\bh_2:\bH_2 : (\id \times \bh_2)$ \\
                           & 0     & 3      & $\Im[z_{2a} z_{2b} \overline{z}_{4a}]$ & $\bh_1:\bH_1 : (\id \times \bh_2)$ \\
                           & 0     & 3      & $\Im[z_{2a} z_{2b} \overline{z}_{4b}]$ & $\bh_1:\bH_2 : (\id \times \bh_2)$ \\
    \bottomrule
  \end{tabular}
\end{subtable}%
\begin{subtable}[T]{.5\linewidth}
  \centering
  \scriptsize
  \begin{tabular}{>{\stepcounter{mtligne} \themtligne}cccll}
    \toprule
    \multicolumn{1}{c}{\#} & order & degree & Formula                                & Formula                      \\
    \midrule
                           & 2     & 1      & $\Im[\overline{z}_{2a} z^{2}]$         & $\id \times \bh_1$           \\
                           & 2     & 1      & $\Im[\overline{z}_{2b} z^{2}]$         & $\id \times \bh_2$           \\
                           & 2     & 2      & $\Im[z_{2a} \overline{z}_{4a} z^{2}]$  & $\bH_1 : (\id \times \bh_1)$ \\
                           & 2     & 2      & $\Im[z_{2a} \overline{z}_{4b} z^{2}]$  & $\bH_2 : (\id \times \bh_1)$ \\
                           & 2     & 2      & $ \Im[z_{2b} \overline{z}_{4a} z^{2}]$ & $\bH_1 : (\id \times \bh_2)$ \\
                           & 2     & 2      & $\Im[z_{2b} \overline{z}_{4b} z^{2}]$  & $\bH_2 : (\id \times \bh_2)$ \\
    \midrule
                           & 4     & 1      & $\Im[\overline{z}_{4a} z^{4}]$         & $\id \times \bH_1$           \\
                           & 4     & 1      & $\Im[\overline{z}_{4b} z^{4}]$         & $\id \times \bH_2$           \\
    \bottomrule
  \end{tabular}
\end{subtable}

\end{table}

\begin{table}[h]
  \caption{Isotropic products of hemitropic covariants of $\VV=\Ela \oplus \Ela$}
  \label{tab:ViscoEla-SO2-hemitropic-products}
  \setcounter{mtligne}{25}
\centering
\scriptsize
\begin{tabular}{>{\stepcounter{mtligne} \themtligne}cccll}
  \toprule
  \multicolumn{1}{c}{\#} & order & degree & Formula                                                       & Formula                                                                \\
  \midrule
                         & 0     & 4      & $\Im[z_{2a} \overline{z}_{2b}] \Im[z_{4a} \overline{z}_{4b}]$ & $\left(\bh_{1}:(\id\times\bh_{2}))(\bH_{1}::(\id\times\bH_{2})\right)$ \\
  \midrule
                         & 2     & 3      & $ \Im[z_{4a} \overline{z}_{4b}] \Im[\overline{z}_{2a} z^{2}]$ & $\left(\bH_{1}::(\id\times\bH_{2})\right)  (\id \times \bh_{1})$       \\
                         & 2     & 3      & $ \Im[z_{4a} \overline{z}_{4b}] \Im[\overline{z}_{2b} z^{2}]$ & $\left(\bH_{1}::(\id\times\bH_{2})\right)  (\id \times \bh_{2})$       \\
  \bottomrule
\end{tabular}

\end{table}


\subsection{Piezoelectricity law}
\label{subsec:piezoelectricity-law}

The linear piezoelectricity law~\cite{RD1999} is a linear relation between the (symmetric) second-order strain tensor $\bepsilon$, the (symmetric) second-order stress tensor $\bsigma$, the electric field $\EE$, and the electric displacement $\vec D$. It writes
\begin{equation*}
  \begin{cases}
    \bsigma= \bC:\bepsilon-  \EE \cdot \bP,
    \\
    \vec D=\bP:\bsigma + \pmb\varepsilon_0^{\sigma} \cdot \EE,
  \end{cases}
  \qquad
  \begin{cases}
    \sigma_{ij}= C_{ijkl}\epsilon_{kl} - P_{kij} E_{k},
    \\
    D_{i}=P_{ikl}\sigma_{kl}+ \varepsilon_{0 ik}^\sigma E_{k},
  \end{cases}
\end{equation*}
where $\bC\in \Ela$ is the elasticity fourth-order tensor ($C_{ijkl}=C_{jikl}=C_{klij}$), $\bP\in \Piez$ is the piezoelectricity third-order tensor ($P_{kij}=P_{kji}$), and $\pmb\varepsilon_0^{\sigma}$ is the (symmetric) second-order dielectric permittivity tensor (which should not be confused with the Levi-Civita tensor $\pmb\varepsilon$ or the strain tensor $\bepsilon$). The harmonic decomposition of the constitutive tensor $(\bC,\bP,\pmb\varepsilon_0^{\sigma})$, with respect to either $\SO(2)$ or $\OO(2)$ writes as
\begin{equation*}
  3\HT{0} \oplus 2\HT{1} \oplus 2\HT{2} \oplus \HT{3} \oplus \HT{4},
\end{equation*}
and we will write
\begin{equation*}
  (\bC,\bP,\pmb\varepsilon_0^{\sigma} ) = (\lambda_{1},\lambda_{2},\lambda_{3},z_{1a},z_{1b},z_{2a},z_{2b},z_{3},z_{4}).
\end{equation*}

\begin{thm}\label{thm:FullPiezo}
  A minimal integrity basis for the $\SO(2)$-covariant algebra for the triplet $(\bC, \bP, \pmb\varepsilon_0^{\sigma})$ consists in the 206 covariants (121 invariants) of~\autoref{tab:FullPiezo-SO2-isotropic} and~\autoref{tab:FullPiezo-SO2-hemitropic}. A minimal integrity basis for the $\OO(2)$-covariant algebra for the triplet $(\bC, \bP, \pmb\varepsilon_0^{\sigma})$ consists in the 123 covariants (71 invariants) of~\autoref{tab:FullPiezo-SO2-isotropic} and \autoref{tab:FullPiezo-SO2-hemitropic-products}.
\end{thm}

\begin{rem}
  In this case also, it remains many products $\Im(\bm_{p})\Im(\bm_{q})$ (see~\autoref{tab:FullPiezo-SO2-hemitropic-products}) which cannot be eliminated from the $\OO(2)$-integrity basis .
\end{rem}

\begin{table}[h]
  \caption{Isotropic covariants for the triplet $(\bC, \bP, \pmb\varepsilon_0^{\sigma})$}
  \label{tab:FullPiezo-SO2-isotropic}
  \setcounter{mtligne}{0}

\begin{subtable}[T]{.5\linewidth}
  \centering
  \scriptsize
  \begin{tabular}{>{\stepcounter{mtligne} \themtligne}cccl}
    \toprule
    \multicolumn{1}{c}{\#} & order & degree & Formula                                                  \\
    \midrule
                           & 0     & 1      & $\lambda_{1}$                                            \\
                           & 0     & 1      & $\lambda_{2}$                                            \\
                           & 0     & 1      & $\lambda_3$                                              \\
                           & 0     & 2      & $z_{1a} \overline{z}_{1a}$                               \\
                           & 0     & 2      & $z_{1b} \overline{z}_{1b}$                               \\
                           & 0     & 2      & $z_{2a} \overline{z}_{2a}$                               \\
                           & 0     & 2      & $z_{2b} \overline{z}_{2b}$                               \\
                           & 0     & 2      & $z_{3} \overline{z}_{3}$                                 \\
                           & 0     & 2      & $z_{4} \overline{z}_{4}$                                 \\
                           & 0     & 2      & $\Re[z_{1a} \overline{z}_{1b}]$                          \\
                           & 0     & 2      & $ \Re[z_{2a} \overline{z}_{2b}]$                         \\
                           & 0     & 3      & $ \Re[z_{1a}^{2} \overline{z}_{2a}]$                     \\
                           & 0     & 3      & $\Re[z_{1a}^{2} \overline{z}_{2b}]$                      \\
                           & 0     & 3      & $ \Re[z_{1b}^{2} \overline{z}_{2a}]$                     \\
                           & 0     & 3      & $\Re[z_{1b}^{2} \overline{z}_{2b}]$                      \\
                           & 0     & 3      & $ \Re[z_{2a}^{2} \overline{z}_{4}]$                      \\
                           & 0     & 3      & $ \Re[z_{2b}^{2} \overline{z}_{4}]$                      \\
                           & 0     & 3      & $\Re[z_{1a} z_{1b} \overline{z}_{2a}]$                   \\
                           & 0     & 3      & $ \Re[z_{1a} z_{1b} \overline{z}_{2b}]$                  \\
                           & 0     & 3      & $\Re[z_{1a} z_{2a} \overline{z}_{3}]$                    \\
                           & 0     & 3      & $\Re[z_{1b} z_{2a} \overline{z}_{3}]$                    \\
                           & 0     & 3      & $ \Re[z_{1b} z_{2b} \overline{z}_{3}]$                   \\
                           & 0     & 3      & $ \Re[z_{1a} z_{2b} \overline{z}_{3}]$                   \\
                           & 0     & 3      & $\Re[z_{2a} z_{2b} \overline{z}_{4}]$                    \\
                           & 0     & 3      & $\Re[z_{1a} z_{3} \overline{z}_{4}]$                     \\
                           & 0     & 3      & $ \Re[z_{1b} z_{3} \overline{z}_{4}]$                    \\
                           & 0     & 4      & $\Re[z_{1a}^{3} \overline{z}_{3}]$                       \\
                           & 0     & 4      & $\Re[z_{1a}^{2} z_{1b} \overline{z}_{3}]$                \\
                           & 0     & 4      & $ \Re[z_{1a} z_{1b}^{2} \overline{z}_{3}]$               \\
                           & 0     & 4      & $ \Re[z_{1b}^{3} \overline{z}_{3}]$                      \\
                           & 0     & 4      & $ \Re[z_{1a} \overline{z}_{2a}^{2} z_{3}]$               \\
                           & 0     & 4      & $\Re[z_{1b} \overline{z}_{2a}^{2} z_{3}]$                \\
                           & 0     & 4      & $\Re[z_{1a} \overline{z}_{2a} \overline{z}_{2b} z_{3}]$  \\
                           & 0     & 4      & $ \Re[z_{1b} \overline{z}_{2a} \overline{z}_{2b} z_{3}]$ \\
                           & 0     & 4      & $ \Re[z_{1a} \overline{z}_{2b}^{2} z_{3}]$               \\
                           & 0     & 4      & $\Re[z_{1b} \overline{z}_{2b}^{2} z_{3}]$                \\
                           & 0     & 4      & $\Re[z_{1a}^{2} z_{2a} \overline{z}_{4}]$                \\
                           & 0     & 4      & $\Re[z_{1a} z_{1b} z_{2a} \overline{z}_{4}]$             \\
                           & 0     & 4      & $ \Re[z_{1b}^{2} z_{2a} \overline{z}_{4}]$               \\
                           & 0     & 4      & $ \Re[z_{1a}^{2} z_{2b} \overline{z}_{4}]$               \\
                           & 0     & 4      & $ \Re[z_{1a} z_{1b} z_{2b} \overline{z}_{4}]$            \\
                           & 0     & 4      & $\Re[z_{1b}^{2} z_{2b} \overline{z}_{4}]$                \\
                           & 0     & 4      & $\Re[z_{1a} \overline{z}_{2a} \overline{z}_{3} z_{4}]$   \\
                           & 0     & 4      & $ \Re[z_{1b} \overline{z}_{2a} \overline{z}_{3} z_{4}]$  \\
                           & 0     & 4      & $\Re[z_{1a} \overline{z}_{2b} \overline{z}_{3} z_{4}]$   \\
                           & 0     & 4      & $ \Re[z_{1b} \overline{z}_{2b} \overline{z}_{3} z_{4}]$  \\
                           & 0     & 4      & $\Re[z_{2a} \overline{z}_{3}^{2} z_{4}]$                 \\
                           & 0     & 4      & $\Re[z_{2b} \overline{z}_{3}^{2} z_{4}]$                 \\
                           & 0     & 5      & $ \Re[z_{2a}^{3} \overline{z}_{3}^{2}]$                  \\
                           & 0     & 5      & $\Re[z_{2a}^{2} z_{2b} \overline{z}_{3}^{2}]$            \\
                           & 0     & 5      & $ \Re[z_{2a} z_{2b}^{2} \overline{z}_{3}^{2}]$           \\
                           & 0     & 5      & $\Re[z_{2b}^{3} \overline{z}_{3}^{2}]$                   \\
                           & 0     & 5      & $ \Re[z_{1a}^{4} \overline{z}_{4}]$                      \\
                           & 0     & 5      & $ \Re[z_{1a}^{3} z_{1b} \overline{z}_{4}]$               \\
                           & 0     & 5      & $ \Re[z_{1a}^{2} z_{1b}^{2} \overline{z}_{4}]$           \\
    \bottomrule
  \end{tabular}
\end{subtable}%
\begin{subtable}[T]{.5\linewidth}
  \centering
  \scriptsize
  \begin{tabular}{>{\stepcounter{mtligne} \themtligne}cccl}
    \toprule
    \multicolumn{1}{c}{\#} & order & degree & Formula                                            \\
    \midrule
                           & 0     & 5      & $ \Re[z_{1a} z_{1b}^{3} \overline{z}_{4}]$         \\
                           & 0     & 5      & $\Re[z_{1b}^{4} \overline{z}_{4}]$                 \\
                           & 0     & 5      & $\Re[z_{2a} z_{3}^{2} \overline{z}_{4}^{2}]$       \\
                           & 0     & 5      & $\Re[z_{2b} z_{3}^{2} \overline{z}_{4}^{2}]$       \\
                           & 0     & 5      & $ \Re[z_{1a}^{2} \overline{z}_{3}^{2} z_{4}]$      \\
                           & 0     & 5      & $\Re[z_{1a} z_{1b} \overline{z}_{3}^{2} z_{4}]$    \\
                           & 0     & 5      & $\Re[z_{1b}^{2} \overline{z}_{3}^{2} z_{4}]$       \\
                           & 0     & 6      & $\Re[z_{1a} \overline{z}_{3}^{3} z_{4}^{2}]$       \\
                           & 0     & 6      & $ \Re[z_{1b} \overline{z}_{3}^{3} z_{4}^{2}]$      \\
                           & 0     & 7      & $\Re[z_{3}^{4} \overline{z}_{4}^{3}]$              \\
    \midrule
                           & 1     & 1      & $\Re[\overline{z}_{1a} z]$                         \\
                           & 1     & 1      & $\Re[\overline{z}_{1b} z]$                         \\
                           & 1     & 2      & $ \Re[z_{1a} \overline{z}_{2a} z]$                 \\
                           & 1     & 2      & $\Re[z_{1b} \overline{z}_{2a} z]$                  \\
                           & 1     & 2      & $\Re[z_{1a} \overline{z}_{2b} z]$                  \\
                           & 1     & 2      & $\Re[z_{1b} \overline{z}_{2b} z]$                  \\
                           & 1     & 2      & $ \Re[z_{2a} \overline{z}_{3} z]$                  \\
                           & 1     & 2      & $\Re[z_{2b} \overline{z}_{3} z]$                   \\
                           & 1     & 2      & $\Re[z_{3} \overline{z}_{4} z]$                    \\
                           & 1     & 3      & $ \Re[z_{1a}^{2} \overline{z}_{3} z]$              \\
                           & 1     & 3      & $ \Re[\overline{z}_{2b}^{2} z_{3} z]$              \\
                           & 1     & 3      & $\Re[z_{1b}^{2} \overline{z}_{3} z]$               \\
                           & 1     & 3      & $\Re[\overline{z}_{2a}^{2} z_{3} z]$               \\
                           & 1     & 3      & $\Re[\overline{z}_{2a} \overline{z}_{2b} z_{3} z]$ \\
                           & 1     & 3      & $ \Re[z_{1a} z_{1b} \overline{z}_{3} z]$           \\
                           & 1     & 3      & $\Re[\overline{z}_{2a} \overline{z}_{3} z_{4} z]$  \\
                           & 1     & 3      & $\Re[\overline{z}_{2b} \overline{z}_{3} z_{4} z]$  \\
                           & 1     & 3      & $\Re[z_{1a} z_{2a} \overline{z}_{4} z]$            \\
                           & 1     & 3      & $\Re[z_{1b} z_{2a} \overline{z}_{4} z]$            \\
                           & 1     & 3      & $\Re[z_{1a} z_{2b} \overline{z}_{4} z]$            \\
                           & 1     & 3      & $\Re[z_{1b} z_{2b} \overline{z}_{4} z]$            \\
                           & 1     & 4      & $\Re[z_{1a}^{3} \overline{z}_{4} z]$               \\
                           & 1     & 4      & $\Re[z_{1b}^{3} \overline{z}_{4} z]$               \\
                           & 1     & 4      & $\Re[z_{1a} \overline{z}_{3}^{2} z_{4} z]$         \\
                           & 1     & 4      & $\Re[z_{1b} \overline{z}_{3}^{2} z_{4} z]$         \\
                           & 1     & 4      & $\Re[z_{1a}^{2} z_{1b} \overline{z}_{4} z]$        \\
                           & 1     & 4      & $ \Re[z_{1a} z_{1b}^{2} \overline{z}_{4} z]$       \\
                           & 1     & 5      & $\Re[\overline{z}_{3}^{3} z_{4}^{2} z]$            \\
    \midrule
                           & 2     & 0      & $z \overline{z}$                                   \\
                           & 2     & 1      & $\Re[\overline{z}_{2a} z^{2}]$                     \\
                           & 2     & 1      & $ \Re[\overline{z}_{2b} z^{2}]$                    \\
                           & 2     & 2      & $\Re[z_{1a} \overline{z}_{3} z^{2}]$               \\
                           & 2     & 2      & $\Re[z_{1b} \overline{z}_{3} z^{2}]$               \\
                           & 2     & 2      & $\Re[z_{2a} \overline{z}_{4} z^{2}]$               \\
                           & 2     & 2      & $\Re[z_{2b} \overline{z}_{4} z^{2}]$               \\
                           & 2     & 3      & $\Re[z_{1a}^{2} \overline{z}_{4} z^{2}]$           \\
                           & 2     & 3      & $\Re[z_{1b}^{2} \overline{z}_{4} z^{2}]$           \\
                           & 2     & 3      & $\Re[\overline{z}_{3}^{2} z_{4} z^{2}]$            \\
                           & 2     & 3      & $\Re[z_{1a} z_{1b} \overline{z}_{4} z^{2}]$        \\
    \midrule
                           & 3     & 1      & $\Re[\overline{z}_{3} z^{3}]$                      \\
                           & 3     & 2      & $\Re[z_{1a} \overline{z}_{4} z^{3}]$               \\
                           & 3     & 2      & $\Re[z_{1b} \overline{z}_{4} z^{3}]$               \\
    \midrule
                           & 4     & 1      & $\Re[\overline{z}_{4} z^{4}]$                      \\
    \bottomrule
  \end{tabular}
\end{subtable}

\end{table}

\begin{table}[h]
  \caption{Hemitropic covariants for the triplet $(\bC, \bP, \pmb\varepsilon_0^{\sigma})$}
  \label{tab:FullPiezo-SO2-hemitropic}
  \setcounter{mtligne}{108}

\begin{subtable}[T]{.5\linewidth}
  \centering
  \scriptsize
  \begin{tabular}{>{\stepcounter{mtligne} \themtligne}cccl}
    \toprule
    \multicolumn{1}{c}{\#} & order & degree & Formula                                                  \\
    \midrule
                           & 0     & 2      & $\Im[z_{1a} \overline{z}_{1b}]$                          \\
                           & 0     & 2      & $ \Im[z_{2a} \overline{z}_{2b}]$                         \\
                           & 0     & 3      & $ \Im[z_{1a}^{2} \overline{z}_{2a}]$                     \\
                           & 0     & 3      & $\Im[z_{1a}^{2} \overline{z}_{2b}]$                      \\
                           & 0     & 3      & $ \Im[z_{1b}^{2} \overline{z}_{2a}]$                     \\
                           & 0     & 3      & $\Im[z_{1b}^{2} \overline{z}_{2b}]$                      \\
                           & 0     & 3      & $ \Im[z_{2a}^{2} \overline{z}_{4}]$                      \\
                           & 0     & 3      & $ \Im[z_{2b}^{2} \overline{z}_{4}]$                      \\
                           & 0     & 3      & $\Im[z_{1a} z_{1b} \overline{z}_{2a}]$                   \\
                           & 0     & 3      & $ \Im[z_{1a} z_{1b} \overline{z}_{2b}]$                  \\
                           & 0     & 3      & $\Im[z_{1a} z_{2a} \overline{z}_{3}]$                    \\
                           & 0     & 3      & $\Im[z_{1b} z_{2a} \overline{z}_{3}]$                    \\
                           & 0     & 3      & $ \Im[z_{1b} z_{2b} \overline{z}_{3}]$                   \\
                           & 0     & 3      & $ \Im[z_{1a} z_{2b} \overline{z}_{3}]$                   \\
                           & 0     & 3      & $\Im[z_{2a} z_{2b} \overline{z}_{4}]$                    \\
                           & 0     & 3      & $\Im[z_{1a} z_{3} \overline{z}_{4}]$                     \\
                           & 0     & 3      & $ \Im[z_{1b} z_{3} \overline{z}_{4}]$                    \\
                           & 0     & 4      & $\Im[z_{1a}^{3} \overline{z}_{3}]$                       \\
                           & 0     & 4      & $\Im[z_{1a}^{2} z_{1b} \overline{z}_{3}]$                \\
                           & 0     & 4      & $ \Im[z_{1a} z_{1b}^{2} \overline{z}_{3}]$               \\
                           & 0     & 4      & $ \Im[z_{1b}^{3} \overline{z}_{3}]$                      \\
                           & 0     & 4      & $ \Im[z_{1a} \overline{z}_{2a}^{2} z_{3}]$               \\
                           & 0     & 4      & $\Im[z_{1b} \overline{z}_{2a}^{2} z_{3}]$                \\
                           & 0     & 4      & $\Im[z_{1a} \overline{z}_{2a} \overline{z}_{2b} z_{3}]$  \\
                           & 0     & 4      & $ \Im[z_{1b} \overline{z}_{2a} \overline{z}_{2b} z_{3}]$ \\
                           & 0     & 4      & $ \Im[z_{1a} \overline{z}_{2b}^{2} z_{3}]$               \\
                           & 0     & 4      & $\Im[z_{1b} \overline{z}_{2b}^{2} z_{3}]$                \\
                           & 0     & 4      & $\Im[z_{1a}^{2} z_{2a} \overline{z}_{4}]$                \\
                           & 0     & 4      & $\Im[z_{1a} z_{1b} z_{2a} \overline{z}_{4}]$             \\
                           & 0     & 4      & $ \Im[z_{1b}^{2} z_{2a} \overline{z}_{4}]$               \\
                           & 0     & 4      & $ \Im[z_{1a}^{2} z_{2b} \overline{z}_{4}]$               \\
                           & 0     & 4      & $ \Im[z_{1a} z_{1b} z_{2b} \overline{z}_{4}]$            \\
                           & 0     & 4      & $\Im[z_{1b}^{2} z_{2b} \overline{z}_{4}]$                \\
                           & 0     & 4      & $\Im[z_{1a} \overline{z}_{2a} \overline{z}_{3} z_{4}]$   \\
                           & 0     & 4      & $ \Im[z_{1b} \overline{z}_{2a} \overline{z}_{3} z_{4}]$  \\
                           & 0     & 4      & $\Im[z_{1a} \overline{z}_{2b} \overline{z}_{3} z_{4}]$   \\
                           & 0     & 4      & $ \Im[z_{1b} \overline{z}_{2b} \overline{z}_{3} z_{4}]$  \\
                           & 0     & 4      & $\Im[z_{2a} \overline{z}_{3}^{2} z_{4}]$                 \\
                           & 0     & 4      & $\Im[z_{2b} \overline{z}_{3}^{2} z_{4}]$                 \\
                           & 0     & 5      & $ \Im[z_{2a}^{3} \overline{z}_{3}^{2}]$                  \\
                           & 0     & 5      & $\Im[z_{2a}^{2} z_{2b} \overline{z}_{3}^{2}]$            \\
                           & 0     & 5      & $ \Im[z_{2a} z_{2b}^{2} \overline{z}_{3}^{2}]$           \\
                           & 0     & 5      & $\Im[z_{2b}^{3} \overline{z}_{3}^{2}]$                   \\
                           & 0     & 5      & $ \Im[z_{1a}^{4} \overline{z}_{4}]$                      \\
                           & 0     & 5      & $ \Im[z_{1a}^{3} z_{1b} \overline{z}_{4}]$               \\
                           & 0     & 5      & $ \Im[z_{1a}^{2} z_{1b}^{2} \overline{z}_{4}]$           \\
                           & 0     & 5      & $ \Im[z_{1a} z_{1b}^{3} \overline{z}_{4}]$               \\
                           & 0     & 5      & $\Im[z_{1b}^{4} \overline{z}_{4}]$                       \\
                           & 0     & 5      & $\Im[z_{2a} z_{3}^{2} \overline{z}_{4}^{2}]$             \\
                           & 0     & 5      & $\Im[z_{2b} z_{3}^{2} \overline{z}_{4}^{2}]$             \\
    \bottomrule
  \end{tabular}
\end{subtable}%
\begin{subtable}[T]{.5\linewidth}
  \centering
  \scriptsize
  \begin{tabular}{>{\stepcounter{mtligne} \themtligne}cccl}
    \toprule
    \multicolumn{1}{c}{\#} & order & degree & Formula                                            \\
    \midrule
                           & 0     & 5      & $ \Im[z_{1a}^{2} \overline{z}_{3}^{2} z_{4}]$      \\
                           & 0     & 5      & $\Im[z_{1a} z_{1b} \overline{z}_{3}^{2} z_{4}]$    \\
                           & 0     & 5      & $\Im[z_{1b}^{2} \overline{z}_{3}^{2} z_{4}]$       \\
                           & 0     & 6      & $\Im[z_{1a} \overline{z}_{3}^{3} z_{4}^{2}]$       \\
                           & 0     & 6      & $ \Im[z_{1b} \overline{z}_{3}^{3} z_{4}^{2}]$      \\
                           & 0     & 7      & $\Im[z_{3}^{4} \overline{z}_{4}^{3}]$              \\
    \midrule
                           & 1     & 1      & $\Im[\overline{z}_{1a} z]$                         \\
                           & 1     & 1      & $\Im[\overline{z}_{1b} z]$                         \\
                           & 1     & 2      & $ \Im[z_{1a} \overline{z}_{2a} z]$                 \\
                           & 1     & 2      & $\Im[z_{1b} \overline{z}_{2a} z]$                  \\
                           & 1     & 2      & $\Im[z_{1a} \overline{z}_{2b} z]$                  \\
                           & 1     & 2      & $\Im[z_{1b} \overline{z}_{2b} z]$                  \\
                           & 1     & 2      & $ \Im[z_{2a} \overline{z}_{3} z]$                  \\
                           & 1     & 2      & $\Im[z_{2b} \overline{z}_{3} z]$                   \\
                           & 1     & 2      & $\Im[z_{3} \overline{z}_{4} z]$                    \\
                           & 1     & 3      & $ \Im[z_{1a}^{2} \overline{z}_{3} z]$              \\
                           & 1     & 3      & $ \Im[\overline{z}_{2b}^{2} z_{3} z]$              \\
                           & 1     & 3      & $\Im[z_{1b}^{2} \overline{z}_{3} z]$               \\
                           & 1     & 3      & $\Im[\overline{z}_{2a}^{2} z_{3} z]$               \\
                           & 1     & 3      & $\Im[\overline{z}_{2a} \overline{z}_{2b} z_{3} z]$ \\
                           & 1     & 3      & $ \Im[z_{1a} z_{1b} \overline{z}_{3} z]$           \\
                           & 1     & 3      & $\Im[\overline{z}_{2a} \overline{z}_{3} z_{4} z]$  \\
                           & 1     & 3      & $\Im[\overline{z}_{2b} \overline{z}_{3} z_{4} z]$  \\
                           & 1     & 3      & $\Im[z_{1a} z_{2a} \overline{z}_{4} z]$            \\
                           & 1     & 3      & $\Im[z_{1b} z_{2a} \overline{z}_{4} z]$            \\
                           & 1     & 3      & $\Im[z_{1a} z_{2b} \overline{z}_{4} z]$            \\
                           & 1     & 3      & $\Im[z_{1b} z_{2b} \overline{z}_{4} z]$            \\
                           & 1     & 4      & $\Im[z_{1a}^{3} \overline{z}_{4} z]$               \\
                           & 1     & 4      & $\Im[z_{1b}^{3} \overline{z}_{4} z]$               \\
                           & 1     & 4      & $\Im[z_{1a} \overline{z}_{3}^{2} z_{4} z]$         \\
                           & 1     & 4      & $\Im[z_{1b} \overline{z}_{3}^{2} z_{4} z]$         \\
                           & 1     & 4      & $\Im[z_{1a}^{2} z_{1b} \overline{z}_{4} z]$        \\
                           & 1     & 4      & $ \Im[z_{1a} z_{1b}^{2} \overline{z}_{4} z]$       \\
                           & 1     & 5      & $\Im[\overline{z}_{3}^{3} z_{4}^{2} z]$            \\
    \midrule
                           & 2     & 1      & $\Im[\overline{z}_{2a} z^{2}]$                     \\
                           & 2     & 1      & $ \Im[\overline{z}_{2b} z^{2}]$                    \\
                           & 2     & 2      & $\Im[z_{1a} \overline{z}_{3} z^{2}]$               \\
                           & 2     & 2      & $\Im[z_{1b} \overline{z}_{3} z^{2}]$               \\
                           & 2     & 2      & $\Im[z_{2a} \overline{z}_{4} z^{2}]$               \\
                           & 2     & 2      & $\Im[z_{2b} \overline{z}_{4} z^{2}]$               \\
                           & 2     & 3      & $\Im[z_{1a}^{2} \overline{z}_{4} z^{2}]$           \\
                           & 2     & 3      & $\Im[z_{1b}^{2} \overline{z}_{4} z^{2}]$           \\
                           & 2     & 3      & $\Im[\overline{z}_{3}^{2} z_{4} z^{2}]$            \\
                           & 2     & 3      & $\Im[z_{1a} z_{1b} \overline{z}_{4} z^{2}]$        \\
    \midrule
                           & 3     & 1      & $\Im[\overline{z}_{3} z^{3}]$                      \\
                           & 3     & 2      & $\Im[z_{1a} \overline{z}_{4} z^{3}]$               \\
                           & 3     & 2      & $\Im[z_{1b} \overline{z}_{4} z^{3}]$               \\
    \midrule
                           & 4     & 1      & $\Im[\overline{z}_{4} z^{4}]$                      \\
    \bottomrule
  \end{tabular}
\end{subtable}

\end{table}

\begin{table}[h]
  \caption{Isotropic products of hemitropic covariants for the triplet $(\bC, \bP, \pmb\varepsilon_0^{\sigma})$}
  \label{tab:FullPiezo-SO2-hemitropic-products}
  \setcounter{mtligne}{108}

\begin{subtable}[T]{.5\linewidth}
  \centering
  \scriptsize
  \begin{tabular}{>{\stepcounter{mtligne} \themtligne}cccl}
    \toprule
    \multicolumn{1}{c}{\#} & order & degree & Formula                                                            \\
    \midrule
                           & 0     & 4      & $\Im[z_{1a} \overline{z}_{1b}] \Im[z_{2a} \overline{z}_{2b}]$      \\
                           & 0     & 5      & $\Im[z_{1a} \overline{z}_{1b}] \Im[z_{2a}^{2} \overline{z}_{4}]$   \\
                           & 0     & 5      & $\Im[z_{1a} \overline{z}_{1b}] \Im[z_{2b}^{2} \overline{z}_{4}]$   \\
                           & 0     & 5      & $\Im[z_{1a} \overline{z}_{1b}] \Im[z_{2a} z_{2b}\overline{z}_{4}]$ \\
                           & 0     & 5      & $ \Im[z_{2a} \overline{z}_{2b}] \Im[z_{1a} z_{3}\overline{z}_{4}]$ \\
                           & 0     & 5      & $\Im[z_{2a} \overline{z}_{2b}] \Im[z_{1b} z_{3} \overline{z}_{4}]$ \\
    \bottomrule
  \end{tabular}
\end{subtable}%
\begin{subtable}[T]{.5\linewidth}
  \centering
  \scriptsize
  \begin{tabular}{>{\stepcounter{mtligne} \themtligne}cccl}
    \toprule
    \multicolumn{1}{c}{\#} & order & degree & Formula                                                         \\               \midrule
                           & 2     & 3      & $\Im[z_{2a} \overline{z}_{2b}] \Im[\overline{z}_{1a} z]$        \\
                           & 2     & 3      & $\Im[z_{2a} \overline{z}_{2b}] \Im[\overline{z}_{1b} z]$        \\
                           & 2     & 4      & $\Im[z_{2a}^{2} \overline{z}_{4}] \Im[\overline{z}_{1a} z]$     \\
                           & 2     & 4      & $\Im[z_{2a}^{2} \overline{z}_{4}] \Im[\overline{z}_{1b} z]$     \\
                           & 2     & 4      & $\Im[z_{2b}^{2} \overline{z}_{4}] \Im[\overline{z}_{1a} z]$     \\
                           & 2     & 4      & $\Im[z_{2b}^{2} \overline{z}_{4}] \Im[\overline{z}_{1b} z]$     \\
                           & 2     & 4      & $\Im[z_{2a} z_{2b} \overline{z}_{4}] \Im[\overline{z}_{1a} z]$  \\
                           & 2     & 4      & $ \Im[z_{2a} z_{2b} \overline{z}_{4}] \Im[\overline{z}_{1b} z]$ \\
                           & 2     & 4      & $\Im[z_{2a} \overline{z}_{2b}] \Im[z_{3} \overline{z}_{4} z]$   \\
    \bottomrule
  \end{tabular}
\end{subtable}

\end{table}

\clearpage

\subsection{Twelfth-order totally symmetric tensor}
\label{subsec:twelth-order-symmetric-tensors}

Totally symmetric tensors are encountered for the intrinsic description of directional data~\cite{Kan1984}. They may represent the directional density of spatial contacts and grains orientations within granular materials~\cite{Oda1982}, the directional crack density~\cite{Kan1984,Ona1984,LK1993,TNS2001}, or the directional description of microstructure degradation by rafting in single crystal superalloys at high temperature~\cite{CDC2018}.

The harmonic decomposition of $\Sym^{12}(\RR^{2})$ is the same under $\SO(2)$ or $\OO(2)$ and writes
\begin{equation*}
  \HT{0} \oplus \HT{2} \oplus \HT{4} \oplus \HT{6} \oplus \HT{8} \oplus \HT{10} \oplus \HT{12},
\end{equation*}
and we will write
\begin{equation*}
  \bS = (\lambda,z_{2},z_{4},z_{6},z_{8},z_{10},z_{12}),
\end{equation*}
where $\lambda\in \RR$ and $z_{k}\in \CC$. Due to the large number of two-dimensional harmonic components, only minimal integrity bases for its invariant algebra are detailed (and, as for piezoelectricity law, translations into tensorial expressions will not be provided).

\begin{thm}\label{thm:S12-tensors-O2-SO2}
  A minimal integrity basis for $\inv(\Sym^{12}(\RR^{2}),\SO(2))$ consists in the 211 invariants of~\autoref{tab:S12-SO2-isotropic} and~\autoref{tab:S12-SO2-hemitropic}.
  A minimal integrity basis for $\inv(\Sym^{12}(\RR^{2}),\OO(2))$ consists in the 113 invariants of~\autoref{tab:S12-SO2-isotropic} and \autoref{tab:S12-SO2-hemitropic-products}.
\end{thm}

\begin{rem}
  Again, it remains four products $\Im(\bm_{p})\Im(\bm_{q})$ (see~\autoref{tab:S12-SO2-hemitropic-products}) which cannot be eliminated from the $\OO(2)$-integrity basis .
\end{rem}

\begin{table}[h]
  \caption{Isotropic invariants of $\Sym^{12}(\RR^2)$}
  \label{tab:S12-SO2-isotropic}
  \setcounter{mtligne}{0}

\begin{subtable}[T]{.5\linewidth}
  \centering
  \scriptsize
  \begin{tabular}{>{\stepcounter{mtligne} \themtligne}cccl}
    \toprule
    \multicolumn{1}{c}{\#} & order & degree & Formula                                                               \\
    \midrule
                           & 0     & 1      & $\lambda$                                                             \\
                           & 0     & 2      & $z_{2} \overline{z}_{2}$                                              \\
                           & 0     & 2      & $z_{4} \overline{z}_{4}$                                              \\
                           & 0     & 2      & $z_{6} \overline{z}_{6}$                                              \\
                           & 0     & 2      & $z_{8} \overline{z}_{8}$                                              \\
                           & 0     & 2      & $ z_{10} \overline{z}_{10}$                                           \\
                           & 0     & 2      & $z_{12} \overline{z}_{12}$                                            \\
                           & 0     & 3      & $\Re[z_{4} \overline{z}_{2}^{2}]$                                     \\
                           & 0     & 3      & $ \Re[z_{8} \overline{z}_{4}^{2}]$                                    \\
                           & 0     & 3      & $ \Re[z_{12} \overline{z}_{6}^{2}]$                                   \\
                           & 0     & 3      & $\Re[z_{12} \overline{z}_{10} \overline{z}_{2}]$                      \\
                           & 0     & 3      & $ \Re[z_{6} \overline{z}_{2} \overline{z}_{4}]$                       \\
                           & 0     & 3      & $ \Re[z_{8} \overline{z}_{2} \overline{z}_{6}]$                       \\
                           & 0     & 3      & $\Re[z_{10} \overline{z}_{4} \overline{z}_{6}]$                       \\
                           & 0     & 3      & $\Re[z_{10} \overline{z}_{2} \overline{z}_{8}]$                       \\
                           & 0     & 3      & $ \Re[z_{12} \overline{z}_{4} \overline{z}_{8}]$                      \\
                           & 0     & 4      & $\Re[z_{6} \overline{z}_{2}^{3}]$                                     \\
                           & 0     & 4      & $ \Re[z_{12} \overline{z}_{4}^{3}]$                                   \\
                           & 0     & 4      & $\Re[z_{6}^{2} \overline{z}_{10} \overline{z}_{2}]$                   \\
                           & 0     & 4      & $\Re[z_{8}^{2} \overline{z}_{12} \overline{z}_{4}]$                   \\
                           & 0     & 4      & $ \Re[z_{8} \overline{z}_{2}^{2} \overline{z}_{4}]$                   \\
                           & 0     & 4      & $\Re[z_{10} \overline{z}_{2} \overline{z}_{4}^{2}]$                   \\
                           & 0     & 4      & $ \Re[z_{4}^{2} \overline{z}_{2} \overline{z}_{6}]$                   \\
                           & 0     & 4      & $\Re[z_{10} \overline{z}_{2}^{2} \overline{z}_{6}]$                   \\
                           & 0     & 4      & $\Re[z_{10}^{2} \overline{z}_{12} \overline{z}_{8}]$                  \\
                           & 0     & 4      & $\Re[z_{12} \overline{z}_{2}^{2} \overline{z}_{8}]$                   \\
                           & 0     & 4      & $\Re[z_{6}^{2} \overline{z}_{4} \overline{z}_{8}]$                    \\
                           & 0     & 4      & $ \Re[z_{8}^{2} \overline{z}_{10} \overline{z}_{6}]$                  \\
                           & 0     & 4      & $\Re[z_{4} z_{8} \overline{z}_{10} \overline{z}_{2}]$                 \\
                           & 0     & 4      & $\Re[z_{10} z_{4} \overline{z}_{12} \overline{z}_{2}]$                \\
                           & 0     & 4      & $\Re[z_{6} z_{8} \overline{z}_{12} \overline{z}_{2}]$                 \\
                           & 0     & 4      & $ \Re[z_{6} z_{8} \overline{z}_{10} \overline{z}_{4}]$                \\
                           & 0     & 4      & $\Re[z_{10} z_{6} \overline{z}_{12} \overline{z}_{4}]$                \\
                           & 0     & 4      & $\Re[z_{10} z_{8} \overline{z}_{12} \overline{z}_{6}]$                \\
                           & 0     & 4      & $\Re[z_{12} \overline{z}_{2} \overline{z}_{4} \overline{z}_{6}]$      \\
                           & 0     & 4      & $\Re[z_{4} z_{6} \overline{z}_{2} \overline{z}_{8}]$                  \\
                           & 0     & 5      & $ \Re[z_{8} \overline{z}_{2}^{4}]$                                    \\
                           & 0     & 5      & $ \Re[z_{6}^{2} \overline{z}_{4}^{3}]$                                \\
                           & 0     & 5      & $\Re[z_{12}^{2} \overline{z}_{8}^{3}]$                                \\
                           & 0     & 5      & $\Re[z_{4}^{3} \overline{z}_{10} \overline{z}_{2}]$                   \\
                           & 0     & 5      & $\Re[z_{8}^{2} \overline{z}_{12} \overline{z}_{2}^{2}]$               \\
                           & 0     & 5      & $ \Re[z_{12}^{2} \overline{z}_{10}^{2} \overline{z}_{4}]$             \\
                           & 0     & 5      & $\Re[z_{10} \overline{z}_{2}^{3} \overline{z}_{4}]$                   \\
                           & 0     & 5      & $\Re[z_{10}^{2} \overline{z}_{12} \overline{z}_{4}^{2}]$              \\
                           & 0     & 5      & $ \Re[z_{12} \overline{z}_{2}^{2} \overline{z}_{4}^{2}]$              \\
                           & 0     & 5      & $\Re[z_{12} \overline{z}_{2}^{3} \overline{z}_{6}]$                   \\
                           & 0     & 5      & $\Re[z_{10} z_{8} \overline{z}_{6}^{3}]$                              \\
                           & 0     & 5      & $\Re[z_{6}^{2} \overline{z}_{2}^{2} \overline{z}_{8}]$                \\
                           & 0     & 5      & $\Re[z_{10}^{2} \overline{z}_{6}^{2} \overline{z}_{8}]$               \\
                           & 0     & 5      & $ \Re[z_{10}^{2} \overline{z}_{4} \overline{z}_{8}^{2}]$              \\
                           & 0     & 5      & $\Re[z_{8}^{2} \overline{z}_{4} \overline{z}_{6}^{2}]$                \\
                           & 0     & 5      & $ \Re[z_{4}^{2} z_{6} \overline{z}_{12} \overline{z}_{2}]$            \\
                           & 0     & 5      & $\Re[z_{6} z_{8} \overline{z}_{10} \overline{z}_{2}^{2}]$             \\
                           & 0     & 5      & $ \Re[z_{10} z_{6} \overline{z}_{12} \overline{z}_{2}^{2}]$           \\
                           & 0     & 5      & $ \Re[z_{8}^{2} \overline{z}_{10} \overline{z}_{2} \overline{z}_{4}]$ \\
    \bottomrule
  \end{tabular}
\end{subtable}%
\begin{subtable}[T]{.5\linewidth}
  \centering
  \scriptsize
  \begin{tabular}{>{\stepcounter{mtligne} \themtligne}cccl}
    \toprule
    \multicolumn{1}{c}{\#} & order & degree & Formula                                                                    \\
    \midrule
                           & 0     & 5      & $ \Re[z_{12} z_{6} \overline{z}_{10} \overline{z}_{4}^{2}]$                \\
                           & 0     & 5      & $\Re[z_{10}^{2} \overline{z}_{12} \overline{z}_{2} \overline{z}_{6}]$      \\
                           & 0     & 5      & $\Re[z_{10} z_{4} \overline{z}_{2} \overline{z}_{6}^{2}]$                  \\
                           & 0     & 5      & $ \Re[z_{10} z_{6} \overline{z}_{4}^{2} \overline{z}_{8}]$                 \\
                           & 0     & 5      & $ \Re[z_{12}^{2} \overline{z}_{10} \overline{z}_{6} \overline{z}_{8}]$     \\
                           & 0     & 5      & $ \Re[z_{12} z_{6} \overline{z}_{2} \overline{z}_{8}^{2}]$                 \\
                           & 0     & 5      & $\Re[z_{10} z_{12} \overline{z}_{6} \overline{z}_{8}^{2}]$                 \\
                           & 0     & 5      & $ \Re[z_{10} z_{8} \overline{z}_{12} \overline{z}_{2} \overline{z}_{4}]
    $                                                                                                                    \\
                           & 0     & 5      & $\Re[z_{12} z_{8} \overline{z}_{10} \overline{z}_{4} \overline{z}_{6}]$    \\
                           & 0     & 5      & $\Re[z_{12} z_{4} \overline{z}_{2} \overline{z}_{6} \overline{z}_{8}]$     \\
                           & 0     & 6      & $\Re[z_{10} \overline{z}_{2}^{5}]$                                         \\
                           & 0     & 6      & $\Re[z_{8}^{2} \overline{z}_{10} \overline{z}_{2}^{3}]$                    \\
                           & 0     & 6      & $\Re[z_{8}^{3} \overline{z}_{10}^{2} \overline{z}_{4}]$                    \\
                           & 0     & 6      & $ \Re[z_{12} \overline{z}_{2}^{4} \overline{z}_{4}]$                       \\
                           & 0     & 6      & $\Re[z_{10}^{3} \overline{z}_{12}^{2} \overline{z}_{6}]$                   \\
                           & 0     & 6      & $ \Re[z_{6}^{3} \overline{z}_{10} \overline{z}_{4}^{2}]$                   \\
                           & 0     & 6      & $ \Re[z_{8}^{3} \overline{z}_{12} \overline{z}_{6}^{2}]$                   \\
                           & 0     & 6      & $ \Re[z_{10}^{2} \overline{z}_{2} \overline{z}_{6}^{3}]$                   \\
                           & 0     & 6      & $\Re[z_{10}^{2} \overline{z}_{4}^{3} \overline{z}_{8}]$                    \\
                           & 0     & 6      & $ \Re[z_{6}^{3} \overline{z}_{2} \overline{z}_{8}^{2}]$                    \\
                           & 0     & 6      & $\Re[z_{10} z_{6} \overline{z}_{4}^{4}]$                                   \\
                           & 0     & 6      & $\Re[z_{6} z_{8}^{2} \overline{z}_{10}^{2} \overline{z}_{2}]$              \\
                           & 0     & 6      & $\Re[z_{8}^{3} \overline{z}_{10} \overline{z}_{12} \overline{z}_{2}]$      \\
                           & 0     & 6      & $\Re[z_{10}^{2} z_{6} \overline{z}_{12}^{2} \overline{z}_{2}]$             \\
                           & 0     & 6      & $\Re[z_{10} z_{8} \overline{z}_{12} \overline{z}_{2}^{3}]$                 \\
                           & 0     & 6      & $ \Re[z_{10} z_{8}^{2} \overline{z}_{12}^{2} \overline{z}_{2}]$            \\
                           & 0     & 6      & $ \Re[z_{12} z_{6}^{2} \overline{z}_{10}^{2} \overline{z}_{4}]$            \\
                           & 0     & 6      & $\Re[z_{10}^{2} z_{8} \overline{z}_{12}^{2} \overline{z}_{4}]$             \\
                           & 0     & 6      & $\Re[z_{10}^{2} \overline{z}_{12} \overline{z}_{2}^{2} \overline{z}_{4}]$  \\
                           & 0     & 6      & $\Re[z_{12}^{2} \overline{z}_{10} \overline{z}_{4}^{2} \overline{z}_{6}]$  \\
                           & 0     & 6      & $\Re[z_{12}^{2} \overline{z}_{2} \overline{z}_{6} \overline{z}_{8}^{2}]$   \\
                           & 0     & 7      & $ \Re[z_{12} \overline{z}_{2}^{6}]$                                        \\
                           & 0     & 7      & $\Re[z_{8}^{3} \overline{z}_{6}^{4}]$                                      \\
                           & 0     & 7      & $ \Re[z_{10}^{2} \overline{z}_{4}^{5}]$                                    \\
                           & 0     & 7      & $ \Re[z_{10}^{2} \overline{z}_{12} \overline{z}_{2}^{4}]$                  \\
                           & 0     & 7      & $\Re[z_{6}^{4} \overline{z}_{10}^{2} \overline{z}_{4}]$                    \\
                           & 0     & 7      & $\Re[z_{12}^{3} \overline{z}_{10}^{3} \overline{z}_{6}]$                   \\
                           & 0     & 7      & $ \Re[z_{10}^{3} \overline{z}_{12} \overline{z}_{6}^{3}]$                  \\
                           & 0     & 7      & $\Re[z_{12}^{3} \overline{z}_{10}^{2} \overline{z}_{8}^{2}]$               \\
                           & 0     & 7      & $\Re[z_{10}^{3} \overline{z}_{6} \overline{z}_{8}^{3}]$                    \\
                           & 0     & 7      & $\Re[z_{10}^{2} z_{12} \overline{z}_{8}^{4}]$                              \\
                           & 0     & 7      & $\Re[z_{8}^{3} \overline{z}_{10}^{2} \overline{z}_{2}^{2}]$                \\
                           & 0     & 7      & $\Re[z_{10}^{2} z_{8} \overline{z}_{12}^{2} \overline{z}_{2}^{2}]$         \\
                           & 0     & 7      & $ \Re[z_{10}^{3} \overline{z}_{12}^{2} \overline{z}_{2} \overline{z}_{4}]$ \\
                           & 0     & 8      & $ \Re[z_{10}^{3} \overline{z}_{6}^{5}]$                                    \\
                           & 0     & 8      & $ \Re[z_{8}^{4} \overline{z}_{10}^{3} \overline{z}_{2}]$                   \\
                           & 0     & 8      & $\Re[z_{10}^{4} \overline{z}_{12}^{3} \overline{z}_{4}]$                   \\
                           & 0     & 8      & $\Re[z_{10}^{3} \overline{z}_{12}^{2} \overline{z}_{2}^{3}]$               \\
                           & 0     & 8      & $ \Re[z_{10}^{3} z_{8} \overline{z}_{12}^{3} \overline{z}_{2}]$            \\
                           & 0     & 9      & $ \Re[z_{10}^{4} \overline{z}_{8}^{5}]$                                    \\
                           & 0     & 9      & $\Re[z_{10}^{4} \overline{z}_{12}^{3} \overline{z}_{2}^{2}]$               \\
                           & 0     & 9      & $ \Re[z_{12}^{4} \overline{z}_{10}^{4} \overline{z}_{8}]$                  \\
                           & 0     & 10     & $\Re[z_{10}^{5} \overline{z}_{12}^{4} \overline{z}_{2}]$                   \\
                           & 0     & 11     & $\Re[z_{12}^{5} \overline{z}_{10}^{6}]$                                    \\
    \bottomrule
  \end{tabular}
\end{subtable}

\end{table}

\begin{table}[h]
  \caption{Hemitropic invariants of $\Sym^{12}(\RR^2)$}
  \label{tab:S12-SO2-hemitropic}
  \setcounter{mtligne}{109}

\begin{subtable}[T]{.5\linewidth}
  \centering
  \scriptsize
  \begin{tabular}{>{\stepcounter{mtligne} \themtligne}cccl}
    \toprule
    \multicolumn{1}{c}{\#} & order & degree & Formula                                                               \\
    \midrule
                           & 0     & 3      & $\Im[z_{4} \overline{z}_{2}^{2}]$                                     \\
                           & 0     & 3      & $ \Im[z_{8} \overline{z}_{4}^{2}]$                                    \\
                           & 0     & 3      & $ \Im[z_{12} \overline{z}_{6}^{2}]$                                   \\
                           & 0     & 3      & $\Im[z_{12} \overline{z}_{10} \overline{z}_{2}]$                      \\
                           & 0     & 3      & $ \Im[z_{6} \overline{z}_{2} \overline{z}_{4}]$                       \\
                           & 0     & 3      & $ \Im[z_{8} \overline{z}_{2} \overline{z}_{6}]$                       \\
                           & 0     & 3      & $\Im[z_{10} \overline{z}_{4} \overline{z}_{6}]$                       \\
                           & 0     & 3      & $\Im[z_{10} \overline{z}_{2} \overline{z}_{8}]$                       \\
                           & 0     & 3      & $ \Im[z_{12} \overline{z}_{4} \overline{z}_{8}]$                      \\
                           & 0     & 4      & $\Im[z_{6} \overline{z}_{2}^{3}]$                                     \\
                           & 0     & 4      & $ \Im[z_{12} \overline{z}_{4}^{3}]$                                   \\
                           & 0     & 4      & $\Im[z_{6}^{2} \overline{z}_{10} \overline{z}_{2}]$                   \\
                           & 0     & 4      & $\Im[z_{8}^{2} \overline{z}_{12} \overline{z}_{4}]$                   \\
                           & 0     & 4      & $ \Im[z_{8} \overline{z}_{2}^{2} \overline{z}_{4}]$                   \\
                           & 0     & 4      & $\Im[z_{10} \overline{z}_{2} \overline{z}_{4}^{2}]$                   \\
                           & 0     & 4      & $ \Im[z_{4}^{2} \overline{z}_{2} \overline{z}_{6}]$                   \\
                           & 0     & 4      & $\Im[z_{10} \overline{z}_{2}^{2} \overline{z}_{6}]$                   \\
                           & 0     & 4      & $\Im[z_{10}^{2} \overline{z}_{12} \overline{z}_{8}]$                  \\
                           & 0     & 4      & $\Im[z_{12} \overline{z}_{2}^{2} \overline{z}_{8}]$                   \\
                           & 0     & 4      & $\Im[z_{6}^{2} \overline{z}_{4} \overline{z}_{8}]$                    \\
                           & 0     & 4      & $ \Im[z_{8}^{2} \overline{z}_{10} \overline{z}_{6}]$                  \\
                           & 0     & 4      & $\Im[z_{4} z_{8} \overline{z}_{10} \overline{z}_{2}]$                 \\
                           & 0     & 4      & $\Im[z_{10} z_{4} \overline{z}_{12} \overline{z}_{2}]$                \\
                           & 0     & 4      & $\Im[z_{6} z_{8} \overline{z}_{12} \overline{z}_{2}]$                 \\
                           & 0     & 4      & $ \Im[z_{6} z_{8} \overline{z}_{10} \overline{z}_{4}]$                \\
                           & 0     & 4      & $\Im[z_{10} z_{6} \overline{z}_{12} \overline{z}_{4}]$                \\
                           & 0     & 4      & $\Im[z_{10} z_{8} \overline{z}_{12} \overline{z}_{6}]$                \\
                           & 0     & 4      & $\Im[z_{12} \overline{z}_{2} \overline{z}_{4} \overline{z}_{6}]$      \\
                           & 0     & 4      & $\Im[z_{4} z_{6} \overline{z}_{2} \overline{z}_{8}]$                  \\
                           & 0     & 5      & $ \Im[z_{8} \overline{z}_{2}^{4}]$                                    \\
                           & 0     & 5      & $ \Im[z_{6}^{2} \overline{z}_{4}^{3}]$                                \\
                           & 0     & 5      & $\Im[z_{12}^{2} \overline{z}_{8}^{3}]$                                \\
                           & 0     & 5      & $\Im[z_{4}^{3} \overline{z}_{10} \overline{z}_{2}]$                   \\
                           & 0     & 5      & $\Im[z_{8}^{2} \overline{z}_{12} \overline{z}_{2}^{2}]$               \\
                           & 0     & 5      & $ \Im[z_{12}^{2} \overline{z}_{10}^{2} \overline{z}_{4}]$             \\
                           & 0     & 5      & $\Im[z_{10} \overline{z}_{2}^{3} \overline{z}_{4}]$                   \\
                           & 0     & 5      & $\Im[z_{10}^{2} \overline{z}_{12} \overline{z}_{4}^{2}]$              \\
                           & 0     & 5      & $ \Im[z_{12} \overline{z}_{2}^{2} \overline{z}_{4}^{2}]$              \\
                           & 0     & 5      & $\Im[z_{12} \overline{z}_{2}^{3} \overline{z}_{6}]$                   \\
                           & 0     & 5      & $\Im[z_{10} z_{8} \overline{z}_{6}^{3}]$                              \\
                           & 0     & 5      & $\Im[z_{6}^{2} \overline{z}_{2}^{2} \overline{z}_{8}]$                \\
                           & 0     & 5      & $\Im[z_{10}^{2} \overline{z}_{6}^{2} \overline{z}_{8}]$               \\
                           & 0     & 5      & $ \Im[z_{10}^{2} \overline{z}_{4} \overline{z}_{8}^{2}]$              \\
                           & 0     & 5      & $\Im[z_{8}^{2} \overline{z}_{4} \overline{z}_{6}^{2}]$                \\
                           & 0     & 5      & $ \Im[z_{4}^{2} z_{6} \overline{z}_{12} \overline{z}_{2}]$            \\
                           & 0     & 5      & $\Im[z_{6} z_{8} \overline{z}_{10} \overline{z}_{2}^{2}]$             \\
                           & 0     & 5      & $ \Im[z_{10} z_{6} \overline{z}_{12} \overline{z}_{2}^{2}]$           \\
                           & 0     & 5      & $ \Im[z_{8}^{2} \overline{z}_{10} \overline{z}_{2} \overline{z}_{4}]$ \\
                           & 0     & 5      & $ \Im[z_{12} z_{6} \overline{z}_{10} \overline{z}_{4}^{2}]$           \\
                           & 0     & 5      & $\Im[z_{10}^{2} \overline{z}_{12} \overline{z}_{2} \overline{z}_{6}]$ \\
                           & 0     & 5      & $\Im[z_{10} z_{4} \overline{z}_{2} \overline{z}_{6}^{2}]$             \\
    \bottomrule
  \end{tabular}
\end{subtable}%
\begin{subtable}[T]{.5\linewidth}
  \centering
  \scriptsize
  \begin{tabular}{>{\stepcounter{mtligne} \themtligne}cccl}
    \toprule
    \multicolumn{1}{c}{\#} & order & degree & Formula                                                                    \\
    \midrule
                           & 0     & 5      & $ \Im[z_{10} z_{6} \overline{z}_{4}^{2} \overline{z}_{8}]$                 \\
                           & 0     & 5      & $ \Im[z_{12}^{2} \overline{z}_{10} \overline{z}_{6} \overline{z}_{8}]$     \\
                           & 0     & 5      & $ \Im[z_{12} z_{6} \overline{z}_{2} \overline{z}_{8}^{2}]$                 \\
                           & 0     & 5      & $\Im[z_{10} z_{12} \overline{z}_{6} \overline{z}_{8}^{2}]$                 \\
                           & 0     & 5      & $ \Im[z_{10} z_{8} \overline{z}_{12} \overline{z}_{2} \overline{z}_{4}]
    $                                                                                                                    \\
                           & 0     & 5      & $\Im[z_{12} z_{8} \overline{z}_{10} \overline{z}_{4} \overline{z}_{6}]$    \\
                           & 0     & 5      & $\Im[z_{12} z_{4} \overline{z}_{2} \overline{z}_{6} \overline{z}_{8}]$     \\
                           & 0     & 6      & $\Im[z_{10} \overline{z}_{2}^{5}]$                                         \\
                           & 0     & 6      & $\Im[z_{8}^{2} \overline{z}_{10} \overline{z}_{2}^{3}]$                    \\
                           & 0     & 6      & $\Im[z_{8}^{3} \overline{z}_{10}^{2} \overline{z}_{4}]$                    \\
                           & 0     & 6      & $ \Im[z_{12} \overline{z}_{2}^{4} \overline{z}_{4}]$                       \\
                           & 0     & 6      & $\Im[z_{10}^{3} \overline{z}_{12}^{2} \overline{z}_{6}]$                   \\
                           & 0     & 6      & $ \Im[z_{6}^{3} \overline{z}_{10} \overline{z}_{4}^{2}]$                   \\
                           & 0     & 6      & $ \Im[z_{8}^{3} \overline{z}_{12} \overline{z}_{6}^{2}]$                   \\
                           & 0     & 6      & $ \Im[z_{10}^{2} \overline{z}_{2} \overline{z}_{6}^{3}]$                   \\
                           & 0     & 6      & $\Im[z_{10}^{2} \overline{z}_{4}^{3} \overline{z}_{8}]$                    \\
                           & 0     & 6      & $ \Im[z_{6}^{3} \overline{z}_{2} \overline{z}_{8}^{2}]$                    \\
                           & 0     & 6      & $\Im[z_{10} z_{6} \overline{z}_{4}^{4}]$                                   \\
                           & 0     & 6      & $\Im[z_{6} z_{8}^{2} \overline{z}_{10}^{2} \overline{z}_{2}]$              \\
                           & 0     & 6      & $\Im[z_{8}^{3} \overline{z}_{10} \overline{z}_{12} \overline{z}_{2}]$      \\
                           & 0     & 6      & $\Im[z_{10}^{2} z_{6} \overline{z}_{12}^{2} \overline{z}_{2}]$             \\
                           & 0     & 6      & $\Im[z_{10} z_{8} \overline{z}_{12} \overline{z}_{2}^{3}]$                 \\
                           & 0     & 6      & $ \Im[z_{10} z_{8}^{2} \overline{z}_{12}^{2} \overline{z}_{2}]$            \\
                           & 0     & 6      & $ \Im[z_{12} z_{6}^{2} \overline{z}_{10}^{2} \overline{z}_{4}]$            \\
                           & 0     & 6      & $\Im[z_{10}^{2} z_{8} \overline{z}_{12}^{2} \overline{z}_{4}]$             \\
                           & 0     & 6      & $\Im[z_{10}^{2} \overline{z}_{12} \overline{z}_{2}^{2} \overline{z}_{4}]$  \\
                           & 0     & 6      & $\Im[z_{12}^{2} \overline{z}_{10} \overline{z}_{4}^{2} \overline{z}_{6}]$  \\
                           & 0     & 6      & $\Im[z_{12}^{2} \overline{z}_{2} \overline{z}_{6} \overline{z}_{8}^{2}]$   \\
                           & 0     & 7      & $ \Im[z_{12} \overline{z}_{2}^{6}]$                                        \\
                           & 0     & 7      & $\Im[z_{8}^{3} \overline{z}_{6}^{4}]$                                      \\
                           & 0     & 7      & $ \Im[z_{10}^{2} \overline{z}_{4}^{5}]$                                    \\
                           & 0     & 7      & $ \Im[z_{10}^{2} \overline{z}_{12} \overline{z}_{2}^{4}]$                  \\
                           & 0     & 7      & $\Im[z_{6}^{4} \overline{z}_{10}^{2} \overline{z}_{4}]$                    \\
                           & 0     & 7      & $\Im[z_{12}^{3} \overline{z}_{10}^{3} \overline{z}_{6}]$                   \\
                           & 0     & 7      & $ \Im[z_{10}^{3} \overline{z}_{12} \overline{z}_{6}^{3}]$                  \\
                           & 0     & 7      & $\Im[z_{12}^{3} \overline{z}_{10}^{2} \overline{z}_{8}^{2}]$               \\
                           & 0     & 7      & $\Im[z_{10}^{3} \overline{z}_{6} \overline{z}_{8}^{3}]$                    \\
                           & 0     & 7      & $\Im[z_{10}^{2} z_{12} \overline{z}_{8}^{4}]$                              \\
                           & 0     & 7      & $\Im[z_{8}^{3} \overline{z}_{10}^{2} \overline{z}_{2}^{2}]$                \\
                           & 0     & 7      & $\Im[z_{10}^{2} z_{8} \overline{z}_{12}^{2} \overline{z}_{2}^{2}]$         \\
                           & 0     & 7      & $ \Im[z_{10}^{3} \overline{z}_{12}^{2} \overline{z}_{2} \overline{z}_{4}]$ \\
                           & 0     & 8      & $ \Im[z_{10}^{3} \overline{z}_{6}^{5}]$                                    \\
                           & 0     & 8      & $ \Im[z_{8}^{4} \overline{z}_{10}^{3} \overline{z}_{2}]$                   \\
                           & 0     & 8      & $\Im[z_{10}^{4} \overline{z}_{12}^{3} \overline{z}_{4}]$                   \\
                           & 0     & 8      & $\Im[z_{10}^{3} \overline{z}_{12}^{2} \overline{z}_{2}^{3}]$               \\
                           & 0     & 8      & $ \Im[z_{10}^{3} z_{8} \overline{z}_{12}^{3} \overline{z}_{2}]$            \\
                           & 0     & 9      & $ \Im[z_{10}^{4} \overline{z}_{8}^{5}]$                                    \\
                           & 0     & 9      & $\Im[z_{10}^{4} \overline{z}_{12}^{3} \overline{z}_{2}^{2}]$               \\
                           & 0     & 9      & $ \Im[z_{12}^{4} \overline{z}_{10}^{4} \overline{z}_{8}]$                  \\
                           & 0     & 10     & $\Im[z_{10}^{5} \overline{z}_{12}^{4} \overline{z}_{2}]$                   \\
                           & 0     & 11     & $\Im[z_{12}^{5} \overline{z}_{10}^{6}]$                                    \\
    \bottomrule
  \end{tabular}
\end{subtable}

\end{table}

\begin{table}[h]
  \caption{Isotropic products of hemitropic invariants of $\Sym^{12}(\RR^2)$}
  \label{tab:S12-SO2-hemitropic-products}
  \setcounter{mtligne}{109}

\begin{subtable}[T]{.5\linewidth}
  \centering
  \scriptsize
  \begin{tabular}{>{\stepcounter{mtligne} \themtligne}cccl}
    \toprule
    \multicolumn{1}{c}{\#} & order & degree & Formula                                                                          \\
    \midrule
                           & 0     & 6      & $ \Im[z_{4}\overline{z}_{2}^{2} ] \Im[z_{12} \overline{z}_6^{2}]$                \\
                           & 0     & 6      & $\Im[z_{12} \overline{z}_6^{2}] \Im[z_{8} \overline{z}_{4}^{2} ]$                 \\
                           & 0     & 6      & $\Im[ z_{12}\overline{z}_{10} \overline{z}_{2}] \Im[z_{8}\overline{z}_{4}^{2} ]$ \\
                           & 0     & 6      & $ \Im[z_{12} \overline{z}_6^{2}] \Im[z_{10} \overline{z}_{2} \overline{z}_8]$    \\
    \bottomrule
  \end{tabular}
\end{subtable}%

\end{table}

\clearpage

\section{Conclusion}

In this paper, we have formulated, with full details, a method to compute a minimal integrity basis for the invariant algebra of any 2D constitutive tensor (and more generally of any linear representation of the orthogonal groups $\SO(2)$ or $\OO(2)$). These results are formalized as theorem~\ref{thm:SO2-minimal-integrity-basis} and theorem~\ref{thm:O2-integrity-basis}. In the second case, a minimal integrity basis is obtained after applying a ``cleaning algorithm'' which is explained in~\autoref{sec:algorithm}. Besides, several reduction lemmas have been proven in~\autoref{sec:integrity-bases}, which allow to reduce \textit{a priori} the complexity of the computations required to obtain minimality.

Meanwhile, a new paradigm has also been introduced: the concept of \emph{polynomial covariant} which extends the idea of a polynomial invariant. This notion is not really new in mathematics, since it goes back to the early ages of Classical Invariant Theory in the nineteenth century. However, it was not clear how to formulate this concept in the framework of tensor spaces rather than binary forms, for which it was introduced first. This task is now achieved and the covariant algebra of a representation $\VV$ of $G$ (where $ G=\OO(2)$ or $ G=\SO(2)$), is defined as
\begin{equation*}
  \cov(\VV, G) := \inv(\VV\oplus \RR^{2}, G).
\end{equation*}
In other words, it is defined as the invariant algebra of the considered vector space $\VV$ (and thus, in particular, for any constitutive tensor), to which is added a vector space $\RR^2$ (\textit{i.e.} a vector $\xx=(x,y)$). This concept, as exotic as it may sound first for a non specialist, has proven to be much more useful than the invariant algebra itself to solve, for instance, such a problem as the characterization of symmetry classes of the elasticity tensor~\cite{OKDD2022} in a simple manner (compared to invariant characterization of only its fourth-order harmonic part in~\cite{AKP2013}). It is implicit in the work of \cite{BKO1994} and more generally in the theory of tensorial representations \cite{Boe1987}.

Therefore, we have computed minimal integrity bases for covariant algebras rather than invariant algebras, and this for an exhaustive list of constitutive tensors and laws (a minimal basis for the invariant algebra is however immediately deduced from a minimal basis of the covariant algebra but the converse is, of course, not true). The proposed algorithm has proven to be very effective and this is illustrated by the fact that we have been able to compute a minimal integrity basis for the covariant algebra of the following constitutive tensors and laws: all third order tensors (including the piezoelectricty tensor), all fourth-order tensors  (including the elasticity, the Eshelby and the photoelasticity tensors), the linear viscoelasticity law, the Hill elasto-plasticity constitutive equations, the linear (coupled) piezoelectricity law, and (by theorem \ref{thm:subspace-reduction}) any even order symmetric tensors up to order 12.

Our method relies first on an explicit \emph{harmonic decomposition} of the given tensor space. The means to achieve this first task are explained with full details in~\autoref{sec:explicit-harmonic-decomposition} and~\autoref{sec:Sym-harmonic-decomposition}. This allows us then to parameterize a tensor by an $n$-uple of complex numbers $(z_{1}, \dotsc ,z_{r})$ together with some real parameters $\lambda_{k}$ (some isotropic invariants) and $\xi_{i}$ (some hemitropic invariants) but which do not enter explicitly in the computing process. A finite integrity basis is then obtained by solving a Diophantine equation which must be satisfied by the exponents of the monomials
\begin{equation*}
  \bm = z_{1}^{\alpha_{1}}\dotsm z_{r}^{\alpha_{r}}\zb_{1}^{\beta_{1}}\dotsm \zb_{r}^{\beta_{r}}.
\end{equation*}
Integrity bases are thus formulated first using these complex variables (together with the $\lambda_{k}$ and $\xi_{i}$). Moreover, a rigourous and quite exhaustive process has been achieved in~\autoref{sec:monomials-versus-tensors-invariants} to translate all these expressions into tensorial ones, since those are familiar to the mechanical community. Minimal integrity bases for most common constitutive tensors and laws have been expressed this way.

\appendix

\section{An explicit harmonic decomposition}
\label{sec:explicit-harmonic-decomposition}

In this appendix, we propose a general method to obtain explicitly an harmonic decomposition of any linear representation $\VV$ of the orthogonal groups $\SO(2)$ or $\OO(2)$, the main application remaining the special case of a tensorial representation. The method exposed here is rather simple, since it requires only the diagonalization of a matrix. It is however limited to dimension 2 and cannot be generalized to dimension 3 or higher. This is not the only procedure to obtain an explicit harmonic decomposition, other approaches are discussed, for instance, in \cite{ZZDR2001}. The procedure described here follows moreover the same lines as Verchery's original construction for the elasticity tensor~\cite{Ver1982}, and applied later to the piezoelectricity tensor by Vannucci~\cite{Van2007}. The methodology consists in the following steps.

(1) Consider first the linearized action of $\SO(2)$ on $\VV$, or more precisely, the induced representation $\rho^{\prime}$ of the Lie algebra $\mathfrak{so}(2)$ of $\SO(2)$ on $\VV$. The Lie algebra, $\mathfrak{so}(2)$, is one-dimensional and spanned by
\begin{equation*}
  u : = \left.\frac{d}{\mathrm{d}\theta}\right|_{\theta = 0} r_{\theta} =
  \begin{pmatrix}
    0 & -1 \\
    1 & 0
  \end{pmatrix}.
\end{equation*}
This action is given by
\begin{equation*}
  \rho^{\prime}(u) \vv = \left.\frac{d}{\mathrm{d}\theta}\right|_{\theta = 0} \rho(r_{\theta})\vv .
\end{equation*}

(2) Using an $\SO(2)$-invariant inner-product $\langle\cdot,\cdot\rangle$ on $\VV$, we have
\begin{equation*}
  \langle \rho(g)\vv_{1},\rho(g)\vv_{2}\rangle = \langle \vv_{1},\vv_{2}\rangle,
\end{equation*}
for all $\vv_{1},\vv_{2}\in \VV$ and for all $g\in \SO(2)$ and thus
\begin{equation*}
  \langle \rho^{\prime}(u)\vv_{1},\vv_{2}\rangle = - \langle \vv_{1},\rho^{\prime}(u)\vv_{2}\rangle.
\end{equation*}
In other words, $\rho^{\prime}(u)$ is a skew-symmetric linear endomorphism of $\VV$ relatively to this invariant inner product.

\begin{rem}
  An $\OO(2)$-invariant inner-product on $\TT^{n}(\RR^{2})$ is given, using \eqref{eq:rcontract}, by
  \begin{equation*}
    \langle \bT_{1},\bT_{2}\rangle : =  \bT_{1}\rdots{n} \bT_{2}
  \end{equation*}
  and can be extended into an hermitian product on the complexification $\left(\TT^{n}\right)^{\CC} := \TT^{n} \oplus i\TT^{n}$  of $\TT^{n}(\RR^{2})$, by
  \begin{equation*}
    \langle \bT_{1},\bT_{2}\rangle : =  \bT_{1}\rdots{n} \overline \bT_{2}.
  \end{equation*}
  The matrix representation $[\rho^{\prime}(u)]$ of $\rho^{\prime}(u)$ in the basis $(\ee_{i_{1}}\otimes \dotsc \otimes \ee_{i_{n}})$ of $\TT^{n}(\RR^{2})$ is thus obtained using the expression
  \begin{equation}\label{eq:Lie_Alg_Rep_On_Tens}
    \left.\frac{d}{\mathrm{d}\theta}\right|_{\theta = 0} \rho(r_{\theta}) \left(\ee_{i_{1}}\otimes \dotsc \otimes \ee_{i_{n}}\right)= \sum_{k = 1}^{n} \ee_{i_{1}}\otimes \dotsc \otimes u\ee_{i_{k}}\otimes \dotsc \otimes \ee_{i_{n}},
  \end{equation}
  so that
  \begin{equation}\label{eq:rhoprimu}
    (\rho'(u)\bT)_{i_{1}\dots i_{n} }=\sum_{k=1}^{n} u_{i_{k} j_{k}} T_{i_{1} \dotsc i_{k-1}  j_{k} i_{k+1} \dotsc i_{n}}.
  \end{equation}
  In practice, we usually consider representations which are stable subspaces $\VV$ of $\TT^{n}(\RR^{2})$ rather than $\TT^{n}(\RR^{2})$ itself. Typically, $\VV$ is a subspace of $\TT^{n}(\RR^{2})$ defined by some index symmetries. An orthonormal basis of $\VV$ consists then of linear combinations of $\ee_{i_{1}}\otimes \dotsc \otimes \ee_{i_{n}}$ (see examples~\ref{ex:S2-harmonic-decomposition}, \ref{ex:T3-harmonic-decomposition} and \ref{ex:P3-harmonic-decomposition}, for instance).
\end{rem}

(3) Since $\rho'(u)$ is skew-symmetric, its eigenvalues are pure imaginary complex numbers. Moreover, since we know \textit{a priori} that $\rho$ decomposes into harmonic factors, we can conclude that these eigenvalues write $i\, n$, where $n \in \ZZ$ are relative integers. Each vanishing eigenvalue corresponds to a factor $\HT{0}$, while each pair of eigenvalues $(i\,n,-i\,n)$ ($n \ge 1$) corresponds to a factor $\HT{n}$ (with possible multiplicity). Therefore, the de-complexification of an orthonormal basis (with respect to the hermitian product on the complexification space $\VV^{\CC} := \VV \oplus i\VV$) of eigenvectors for $\rho'(u)$ provides an explicit harmonic decomposition relatively to $\SO(2)$. Let denote by $\bU_{k}$ the (real) unit eigenvectors associated with the corresponding vanishing eigenvalues, by $\bW_{l}$ the (complex) unit eigenvectors associated with the eigenvalues $i\, n_{l}$ (when $n_{l} \ge 1$) and by $\overline{\bW_{l}}$, their complex conjugates, associated with the eigenvalue $-i\, n_{l}$.

\begin{rem}
  Note that each real unit eigenvector $\bU_{k}$ is defined up to a sign whereas each \emph{complex} unit eigenvector $\bW_{l}$ is defined up to a phase or in other words up to a complex number $e^{i\psi_{l}}$. This ambiguity is well-known in Quantum Mechanics and discussed, for instance, in~\cite[section II.B.1]{AS1999}.
\end{rem}

(4) When a representation of $\OO(2)$ is involved, one starts by calculating an harmonic decomposition relatively to the subgroup $\SO(2)$ and then checks, for each vanishing eigenvalue, whether the corresponding eigenvector $\bU_{k}$ satisfies $\rho(\sigma)\bU_{k} = \bU_{k}$ or $\rho(\sigma)\bU_{k} = -\bU_{k}$. In the first case, such an eigenvector will still be denoted by $\bU_{k}$ but in the second case, it will be changed to $\bV_{i}$.
Note now that
\begin{equation*}
  \rho(r_\theta)\bU_{k} = \bU_{k}, \qquad \rho(r_\theta)\bV_{i} = \bV_{i}, \qquad \rho(r_\theta)\bW_{l} = e^{in_{l} \theta} \bW_{l},
\end{equation*}
independently of the choice of the eigenvectors $\bU_{k}$, $\bW_{l}$. However, the problem is slightly more subtle for $\sigma$. Indeed, the $\CC$-linear extension of the representation $\rho$ satisfies
\begin{equation*}
  \rho(\sigma)\bU_{k} = \bU_{k}, \qquad \rho(\sigma)\bV_{i} = -\bV_{i}, \qquad \rho(\sigma)\bW_{l} = \overline{\bW}_{l} ,
\end{equation*}
only if the arbitrary phase $\psi_{l}$ in the choice of $\bW_{l}$ has been chosen so that
\begin{equation*}
  \rho(\sigma)\bW_{l} = \overline{\bW}_{l} ,
\end{equation*}
which happens only for two values of $\psi_{l}$ (which differ by $\pi$). This must be calculated in order to produce an explicit harmonic decomposition relative to the group $\OO(2)$ and fixes, by the way, the ambiguity. In the following examples, this choice has been made.

Finally, one ends up with the following orthonormal decompositions of $\bT\in \VV \subset \TT^{n}(\RR^2)$.
\begin{description}
  \item[$\SO(2)$-decomposition]
    \begin{equation*}
      \bT = \sum_{k=1}^{\nu_{0}} \lambda_{k} \bU_{k}
      + \sum_{l=1}^{r} \left( \overline{z_{l}} \bW_{l} + z_{l} \overline{\bW}_{l}\right),
    \end{equation*}
  \item[$\OO(2)$-decomposition]
    \begin{equation*}
      \bT=\sum_{k=1}^{m_{0}} \lambda_{k} \bU_{k}+\sum_{i=1}^{m_{-1}} \xi_{i} \bV_{i}
      + \sum_{l=1}^{r} \left( \overline{z_{l}} \bW_{l} + z_{l} \overline{\bW}_{l}\right),
    \end{equation*}
\end{description}
where
\begin{equation*}
  \lambda_{k} = \langle \bT, \bU_{k} \rangle \in \RR
  ,\qquad
  \xi_{i} = \langle \bT, \bV_{i} \rangle \in \RR
  ,\qquad
  z_{l} = \langle \bT, \overline{\bW}_{l} \rangle \in \CC.
\end{equation*}

\begin{rem}\label{rem:scaling-factors}
  Writing $z_{l} = a_{l} + i b_{l}$, we get
  \begin{equation*}
    \overline{z_{l}} \bW_{l} + z_{l} \overline{\bW}_{l}\ =  2\left(a_{l} \Re(\bW_{l}) + b_{l}\Im(\bW_{l})\right).
  \end{equation*}
  The factor $2$ here is meaningless, since the harmonic components are defined up to a scaling factor. The basis, $(\Re(\bW_{l}),\Im(\bW_{l}))$ of the factor $\HT{n_{l}}$ in the decomposition of $\bT$ is orthogonal but not orthonormal, since
  \begin{equation*}
    \norm{\Re(\bW_{l})} = \norm{\Im(\bW_{l})} = \frac{1}{\sqrt{2}}.
  \end{equation*}
  But the same is true for the basis $(\bK_{1}^{(n)}, \bK_{2}^{(n)})$ of $\HT{n}$, as defined in~\eqref{eq:Hn-tensorial-basis}, it is orthogonal but not orthonormal, since
  \begin{equation*}
    \norm{\bK_{1}^{(n)}} = \norm{\bK_{2}^{(n)}} = 2^{\frac{n-1}{2}}.
  \end{equation*}
\end{rem}

\begin{exa}\label{ex:S2-harmonic-decomposition}
  Let $\ba\in \VV =\Sym^2(\RR^2)\subset  \TT^2(\RR^2)$ be a symmetric second order tensor and $\ee_{1}\otimes \ee_{1}$, $\ee_{2}\otimes \ee_{2}$,
  $\frac{1}{\sqrt{2}}\left(\ee_{1}\otimes \ee_{2}+\ee_{2}\otimes \ee_{1}\right)$, be an orthonormal basis of $\VV$. Then $\rho(g)\ba=g \ba g^t$  writes
  $ [\rho(g)\ba] = [\rho(g)][\ba]$ with $[\ba]=(a_{11}, a_{22}, \sqrt{2}\, a_{12})^t$ and
  \begin{equation*}
    [\rho(r_{\theta})]=\left(
    \begin{array}{ccc}
        \cos^2 \theta                  & \sin^2 \theta                   & -\frac{\sin 2 \theta}{\sqrt{2}} \\
        \sin^2 \theta                  & \cos^2 \theta                   & \frac{\sin 2 \theta}{\sqrt{2}}  \\
        \frac{\sin 2 \theta}{\sqrt{2}} & -\frac{\sin 2 \theta}{\sqrt{2}} & \cos 2 \theta                   \\
      \end{array}
    \right),
    \quad
    [\rho'(u)]=\left(\begin{array}{ccc}
        0        & 0         & -\sqrt{2} \\
        0        & 0         & \sqrt{2}  \\
        \sqrt{2} & -\sqrt{2} & 0         \\
      \end{array}
    \right).
  \end{equation*}
  The eigenvalues of $[\rho'(u)]$ read $0, 2i, -2i$ and the associated eigenvectors are respectively
  $[\bU]=(U_{11},  U_{22}, \sqrt{2}\, U_{12})^t=(\frac{1}{\sqrt{2}},\frac{1}{\sqrt{2}}, 0)^t$, $[\bW_{2}]=((\bW_{2})_{11}, (\bW_{2})_{22},  \sqrt{2}\, (\bW_{2})_{12})^t=(\frac{1}{2},-\frac{1}{2}, -\frac{i}{\sqrt{2}})^t$ and $[\overline \bW_{2}]$, so that in $\OO(2)$ case
  \begin{equation*}
    \lambda=\langle \ba, \bU\rangle=\frac{1}{\sqrt{2}} \tr \ba,
    \qquad
    z_{2}= \langle \ba, \bW_{2} \rangle =a_{11}'+i a_{12},
  \end{equation*}
  where $a_{11}'=\frac{1}{2}(a_{11}-a_{22})$ is the first component of the deviatoric part $\ba'=\ba-\frac{1}{2} \left(\tr \ba\right) \id$.
\end{exa}

\begin{exa}\label{ex:T3-harmonic-decomposition}
  Let $\VV = \TT^{3}(\RR^{2})$. Then, an orthonormal basis is given by
  \begin{align*}
    \be_{111} & = \ee_{1}\otimes \ee_{1}\otimes \ee_{1}, &  & \be_{221} = \ee_{2}\otimes \ee_{2}\otimes \ee_{1}, &  & \be_{122} = \ee_{1}\otimes \ee_{2}\otimes \ee_{2}, \\
    \be_{212} & = \ee_{2}\otimes \ee_{1}\otimes \ee_{2}, &  & \be_{222} = \ee_{2}\otimes \ee_{2}\otimes \ee_{2}, &  & \be_{112} = \ee_{1}\otimes \ee_{1}\otimes \ee_{2}, \\
    \be_{211} & = \ee_{2}\otimes \ee_{1}\otimes \ee_{1}, &  & \be_{121} = \ee_{1}\otimes \ee_{2}\otimes \ee_{1},
  \end{align*}
  where $(\ee_{1},\ee_{2})$ is an orthonormal basis of $\RR^{2}$. Using \eqref{eq:rhoprimu} one obtains
  \begin{equation}\label{eq:rhopuT3}
    (\rho'(u)\bT)_{ijk}=u_{ip}T_{pjk}+u_{jp}T_{ipk}+u_{kp}T_{ijp}.
  \end{equation}
  Setting
  \begin{equation*}
    [\bT]=(T_{111}, T_{221}, T_{122}, T_{212}, T_{222}, T_{112}, T_{211}, T_{121})^t,
  \end{equation*}
  and the same for $ [\rho'(u)\bT]$ (as well as for the eigenvectors $[\bW_{n_{k}}]$ of $[\rho'(u)]$),
  one gets $[\rho'(u)\bT]=[\rho'(u)][\bT]$, where
  \begin{equation*}
    [\rho'(u)]= \begin{pmatrix}
      0 & 0  & 0  & 0  & 0  & -1 & -1 & -1 \\
      0 & 0  & 0  & 0  & -1 & 0  & 1  & 1  \\
      0 & 0  & 0  & 0  & -1 & 1  & 0  & 1  \\
      0 & 0  & 0  & 0  & -1 & 1  & 1  & 0  \\
      0 & 1  & 1  & 1  & 0  & 0  & 0  & 0  \\
      1 & 0  & -1 & -1 & 0  & 0  & 0  & 0  \\
      1 & -1 & 0  & -1 & 0  & 0  & 0  & 0  \\
      1 & -1 & -1 & 0  & 0  & 0  & 0  & 0  \\
    \end{pmatrix}.
  \end{equation*}
  $[\rho'(u)]$ is indeed skew-symmetric, with eigenvalues $i, i, i, 3i$ and their four conjugates, with corresponding eigenvectors
  $[\bW_{1a}]=\frac{1}{2}(1,0,0,1,-i,0,0,-i)^t$, $[\bW_{1b}]=\frac{1}{2 \sqrt{3}}(1,0,2,-1,-i,0,-2 i,i)^t$, $[\bW_{1c}]=\frac{1}{2 \sqrt{6}}(1,3,-1,-1,-i,-3 i,i,i)^t$, $[\bW_{3}]=\frac{1}{2 \sqrt{2}}(1,-1,-1,-1,i,-i,-i,-i)^t$ and their four conjugates. This corresponds to an harmonic decomposition $3 \HT{1}\oplus \HT{3}$ of $\TT^{3}(\RR^{2})$, where
  \begin{align*}
    z_{1a} & = \langle \bT, \bW_{1a} \rangle = \frac{1}{2}(T_{111}+T_{212})+\frac{i}{2}(T_{222}+T_{121}),                                                            \\
    z_{1b} & = \langle \bT,  \bW_{1b} \rangle = \frac{1}{2 \sqrt{3}}(T_{111}+2T_{122}- T_{212}) - \frac{i}{2 \sqrt{3}}(T_{222}+ 2T_{211} -T_{121}),
    \\
    z_{1c} & = \langle \bT, \bW_{1c} \rangle = \frac{1}{2 \sqrt{6}}(T_{111} + 3T_{221}-T_{122}-T_{212})+ \frac{i}{2 \sqrt{6}}(T_{222}+ 3 T_{112} -T_{211} -T_{121}),
    \\
    z_{3}  & = \langle \bT, \bW_{3}\rangle = \frac{1}{2\sqrt{2}}(T_{111}- T_{221}- T_{122}- T_{212}) + \frac{i}{2\sqrt{2}}(-T_{222}+ T_{112}+ T_{211}+ T_{121}).
  \end{align*}
\end{exa}

\begin{exa}\label{ex:P3-harmonic-decomposition}
  Let $\VV = \Piez\subset \TT^{3}(\RR^{2})$ be the vector space of 2D piezoelectricity tensors $\bP$, where $P_{ijk}=P_{ikj}$. An orthonormal basis is given by
  \begin{align*}
    \be_{111},                                 &  & \be_{122}, &  & \frac{1}{\sqrt{2}} (\be_{212} +\be_{221}), &  &
    \frac{1}{\sqrt{2}} (\be_{112} +\be_{121}), &  & \be_{211}
    ,                                          &  & \be_{222}.
  \end{align*}
  The induced representation of the Lie algebra $\mathfrak{so}(2)$ on $\Piez$ still writes as~\eqref{eq:rhopuT3}. We now set
  \begin{equation*}
    [\bP]=(P_{111}, P_{122},  \sqrt{2} P_{212},P_{222},P_{211},\sqrt{2} P_{112} )^t,
  \end{equation*} and the same for $ [\rho'(u)\bP]$ (and the $[\bW_{n_{k}}]$), so that
  \begin{equation*}
    [\rho'(u)]=
    \begin{pmatrix}
      0        & 0         & 0         & 0         & -1       & -\sqrt{2} \\
      0        & 0         & 0         & -1        & 0        & \sqrt{2}  \\
      0        & 0         & 0         & -\sqrt{2} & \sqrt{2} & 1         \\
      0        & 1         & \sqrt{2}  & 0         & 0        & 0         \\
      1        & 0         & -\sqrt{2} & 0         & 0        & 0         \\
      \sqrt{2} & -\sqrt{2} & -1        & 0         & 0        & 0         \\
    \end{pmatrix}
  \end{equation*}
  which is indeed skew-symmetric, with eigenvalues $i, i, 3i$ and their three conjugates with corresponding eigenvectors $[\bW_{1a}]=\frac{1}{\sqrt{6}}(\sqrt{2},0,1,-i\sqrt{2},0,-i)^t$, $[\bW_{1b}]=\frac{1}{2\sqrt{6}}(1,3 ,- \sqrt{2},-i,-3i,i\sqrt{2})$, $[\bW_{3}]=\frac{1}{2\sqrt{2}}(1,-1,- \sqrt{2},i,-i,-i\sqrt{2})^t$ and their three conjugates. This corresponds to an harmonic decomposition $2 \HT{1}\oplus \HT{3}$ of $\Piez$, where
  \begin{align*}
    z_{1a} & = \langle \bT, \bW_{1a} \rangle = \frac{1}{\sqrt{3}} (P_{111}  + P_{212})+\frac{i}{\sqrt{3}}(P_{222}+P_{112} ) ,
    \\
    z_{1b} & = \langle \bT, \bW_{1b} \rangle = \frac{1}{2 \sqrt{6}}(P_{111} +3 P_{122}- 2 P_{212}) + \frac{i}{2 \sqrt{6}}(P_{222}+ 3 P_{211} -2 P_{112} ),
    \\
    z_{3}  & = \langle\bT, \bW_{2} \rangle=  \frac{1}{2\sqrt{2}}(P_{111} -P_{122}-2 P_{212})+\frac{i}{2\sqrt{2}} (-P_{222}+ P_{211}+2 P_{112}).
  \end{align*}
\end{exa}

\section{Harmonic decomposition of a totally symmetric tensor}
\label{sec:Sym-harmonic-decomposition}

The harmonic decomposition of an homogeneous polynomial $\rp$ of degree $n$ is unique and writes
\begin{equation}\label{eq:HDecpol}
  \rp = \rh_{0} + \qq \rh_{1} + \dotsb+ \qq^{r}\rh_{r},
\end{equation}
where $r := \lfloor n/2 \rfloor$ (with $\lfloor x \rfloor$ the floor function), $\qq(\xx) := x^{2}+y^{2}$ and $\rh_{k}\in \HP{n-2k}$ are harmonic polynomials of respective degree $n-2k$. These harmonic functions can be calculated iteratively using the relation
\begin{equation*}
  \triangle^{r}(\qq^{r}\rh) = 4^{r}\frac{r!(n+r)!}{n!}\rh,
\end{equation*}
where $\rh$ is an homogeneous harmonic polynomial of degree $n$ and $r \ge 0$. More precisely, we have
\begin{equation*}
  \rh_{r} = \left\{
  \begin{array}{ll}
    \displaystyle{\frac{4^{-r}}{(r!)^{2}}}\triangle^{r} \rp, & \hbox{if $n$ is even;} \\
    \displaystyle{\frac{4^{-r}}{r!(r+1)!}}\triangle^{r} \rp, & \hbox{if $n$ is odd.}
  \end{array}
  \right.
\end{equation*}
Then, we compute inductively $\rh_{k}$, for $k = r-1, r-2, \dotsc$, using the relation
\begin{equation}\label{eq:hk-poly}
  \rh_{k} = 4^{-k}\frac{(n-2k)!}{k!(n-k)!} \triangle^{k}\left(\rp - \sum_{j=k+1}^{r} \qq^{j} \rh_{j}\right)
\end{equation}
which leads to $\rh_{0}$ after $r$ iterations. The harmonic function $\rh_{0}$ (of degree $n$) is called the \emph{leading harmonic part} of $\rp$ {and will be denoted by $\rp^{0}$ when necessary.}

Using now the correspondence between totally symmetric tensors and homogeneous polynomials, we set $\rh_{k}=\phi(\bH_{k})$ and we get
\begin{equation*}
  \phi(\id) = \qq, \quad \phi(\bH_{k})= \rh_{k}, \quad \phi(\underbrace{\id\odot \dotsm \odot \id}_\textrm{$k$ copies}\odot  \bH_{k}) = \qq^k \rh_{k} \quad \textrm{(no sum)}.
\end{equation*}
We deduce then the harmonic decomposition of a totally symmetric tensor $\bS\in \Sym^{n}(\RR^{2})$
\begin{equation}\label{eq:S-harmonic-decomposition}
  \bS = \bH_{0} + \id \odot \bH_{1} + \dotsb + \underbrace{\id\odot \dotsm \odot \id}_\textrm{$r$ copies} \odot \bH_{r}.
\end{equation}
Moreover, using the identity $\phi(\tr^{r} \bS) = \frac{(n-2r)!}{n!}\Delta^{r}\phi(\bS)$, valid for every symmetric tensor $\bS$ or order $n$, we can explicit the factor $\bH_{k}$, where we have set $\id^{\odot j}=\id\odot \dotsm \odot \id$ for $j$ copies of $\id$. Indeed,
\begin{equation*}
  \bH_{r} = \left\{
  \begin{array}{ll}
    \displaystyle 4^{-r}\frac{n!}{(r!)^{2}}\tr^r\bS, & \hbox{if $n$ is even;} \\
    \displaystyle 4^{-r}\frac{n!}{r!(r+1)!}\tr^r\bS, & \hbox{if $n$ is odd;}
  \end{array}
  \right.
\end{equation*}
and
\begin{equation*}
  \bH_{k}=4^{-k}\binom{n}{k} \tr^{k}\left( \bS-\sum_{j=k+1}^{r} \id^{\odot j}\odot \bH_{j} \right).
\end{equation*}
The \emph{leading harmonic part} $\bS':=\bH_{0}$ of $\bS$, such that $\phi(\bS^{\prime}) = \phi(\bS)^{0}$, is an harmonic tensor in $\HT{n}$.

\section{The Hilbert series of an $\OO(2)$-representation}
\label{sec:Hilbert-series}

Consider a linear representation $(\VV,\rho)$ of a compact group $G$ which splits into a direct sum of stable subspaces
\begin{equation*}
  \VV = \VV_{1} \oplus \dotsb \oplus \VV_{p}.
\end{equation*}
Then, each $\vv \in \VV$ can be decomposed uniquely as $\vv = (\vv_{1},\dotsc,\vv_{p})$ and each invariant polynomial writes
\begin{equation*}
  p(\vv) = p(\vv_{1},\dotsc,\vv_{p}),\quad \vv_{i}\in \VV_{i}.
\end{equation*}
Thus, the invariant algebra $\inv(\VV, G)$ inherits a multi-graduation
\begin{equation*}
  \inv{(\VV, G)} = \bigoplus_{k_{1}, \dotsc , k_{p}} \inv_{k_{1}, \dotsc , k_{p}}(\VV, G),
\end{equation*}
where $\inv_{k_{1}, \dotsc , k_{p}}(\VV, G)$ is the \emph{finite dimensional subspace} of multi-homogeneous invariant polynomials $p$ which have degree $k_{1}$ in $\vv_{1}$, $k_{2}$ in $\vv_{2}$, \ldots , and $k_{p}$ in $\vv_{p}$. The dimensions of these vector spaces are encoded into the \emph{multivariate Hilbert series}
\begin{equation}\label{eq:Hilbert-series}
  H_{\rho}(t_{1},\dotsc,t_{p}) = \sum_{k_{1}, \dotsc , k_{p}} a_{k_{1}, \dotsc , k_{p}} \, t_{1}^{k_{1}} \dotsm t_{p}^{k_{p}},
\end{equation}
where
\begin{equation*}
  a_{k_{1}, \dotsc , k_{p}} : =  \dim \inv_{k_{1}, \dotsc , k_{p}}(\VV, G).
\end{equation*}
The remarkable fact is that the Hilbert series is a rational function that can be calculated \textit{a priori}, using the \emph{Molien--Weyl formula}~\cite{Ste1994,Stu2008}.

\begin{thm}[Molien--Weyl formula]\label{thm:Molien--Weyl}
  The multi-graded \emph{Hilbert series}~\eqref{eq:Hilbert-series} of $\inv(\VV,G)$ writes as
  \begin{equation*}
    H_{\rho}(t_{1},\dotsc,t_{p}) = \int_{G} \prod_{k = 1}^{p}\frac{1}{\det (I - t_{k}\rho_{k}(g))} \, \mathrm{d} \mu(g)
  \end{equation*}
  where $\mathrm{d} \mu$ is the \emph{Haar measure} on $G$ (see definition~\ref{def:Haar-measure}) and $\rho_{k}$ is the restriction of the representation to $\VV_{k}$.
\end{thm}

In the special case where $G=\OO(2)$ and
\begin{equation*}
  \VV \simeq \HT{n_{1}} \oplus \dotsb \oplus\HT{n_{r}},
\end{equation*}
which is the important case for us, we have the following more explicit result.

\begin{thm}\label{thm:O2-Hilbert-series}
  Let $(\VV,\rho)$ be a linear representation of $\OO(2)$ which decomposes as
  \begin{equation*}
    \VV \simeq \HT{n_{1}} \oplus \dotsb \oplus\HT{n_{r}},
  \end{equation*}
  where $1 \le n_{1} \le \dotsb \le n_{r}$. Then, the \emph{Hilbert series} of $\inv(\VV,\OO(2))$ writes as
  \begin{equation*}
    H_{\rho}( t_{1},\dotsm,t_{r}) = \frac{1}{2}\left\{\frac{1}{2\pi i} \int_{\abs{z} = 1} \prod_{k = 1}^{r} \frac{1}{(1-t_{k}z^{n_{k}})(1-t_{k}z^{-n_{k}})} \, z^{-1}\mathrm{d}z+ \frac{1}{(1-t_{1}^{2})\dotsm(1-t_{r}^{2})} \right\},
  \end{equation*}
  where
  \begin{equation*}
    \frac{1}{2\pi i} \int_{\abs{z} = 1} \prod_{k = 1}^{r} \frac{1}{(1-t_{k}z^{n_{k}})(1-t_{k}z^{-n_{k}})} \, z^{-1}\mathrm{d}z= \sum_{k_{1},\dotsc,k_{r}\geq 0} b_{k_{1},\dotsc,k_{r}}\,t_{1}^{k_{1}}\dotsm t_{r}^{k_{r}},
  \end{equation*}
  and
  \begin{equation*}
    b_{k_{1},\dotsc,k_{r}} = \frac{1}{2i\pi} \sum_{0 \le \alpha_{i} \le k_{i}} \int_{\abs{z} = 1} z^{(2\alpha_{1}-k_{1})n_{1} + \dotsb + (2\alpha_{r}-k_{r})n_{r}}\,z^{-1}\mathrm{d}z
  \end{equation*}
  is the number of solutions $(\alpha_{1}, \dotsc ,\alpha_{r})$ of the linear Diophantine equation
  \begin{equation}\label{eq:diopheq}
    2\alpha_{1}n_{1} + \dotsb + 2\alpha_{r}n_{r} = k_{1}n_{1} + \dotsb + k_{r}n_{r},\quad \alpha_{i}\geq 0 .
  \end{equation}
\end{thm}

\begin{rem}\label{rem:betak}
  Note that
  \begin{equation*}
    \frac{1}{(1-t_{1}^{2})\dotsm(1-t_{r}^{2})}
    = \sum_{\ell_{1},\dotsc, \ell_{r}\geq 0} t_{1}^{2 \ell_{1}}\dotsm t_{r}^{2\ell_r}
    = \sum_{k_{1},\dotsc, k_{r}\geq 0}\beta_{k_{1}, \dotsc, k_r} t_{1}^{k_{1}}\dotsm t_{r}^{k_r},
  \end{equation*}
  where $\beta_{k_{1}, \dotsc, k_r}=0$ when at least one of the $k_{i}$ is odd and $\beta_{k_{1}, \dotsc, k_r}=1$ otherwise.
\end{rem}

\begin{rem}\label{rem:full-O2-Hilbert-series}
  When the harmonic decomposition of the representation $(\VV,\rho)$ of $\OO(2)$ involves some components $\HT{-1}$ and/or $\HT{0}$, then one can easily obtain its Hilbert series by modifying the series provided in theorem~\ref{thm:O2-Hilbert-series}, using the general formula given in theorem~\ref{thm:Molien--Weyl} and observing that
  \begin{align*}
    \det (I - t_{0}\rho_{0}(r_{\theta}))   & = 1-t_{0},  & \det (I - t_{0}\rho_{0}(\sigma r_{\theta}))   & = 1-t_{0},
    \\
    \det (I - t_{-1}\rho_{-1}(r_{\theta})) & = 1-t_{-1}, & \det (I - t_{-1}\rho_{-1}(\sigma r_{\theta})) & = 1 + t_{-1},
  \end{align*}
  where $\rho_{0}$ and $\rho_{-1}$ are defined in~\autoref{sec:representation-theory}.
\end{rem}

\begin{rem}\label{rem:SO2-Hilbert-series}
  For a linear representation $\VV$ of $\SO(2)$, which decomposes as
  \begin{equation*}
    \VV \simeq \nu_0 \HT{0}\oplus \HT{n_{1}} \oplus \dotsb \oplus\HT{n_{r}},
  \end{equation*}
  where $1 \le n_{1} \le \dotsb \le n_{r}$, the Hilbert series of $\inv(\VV,\SO(2))$ writes
  \begin{equation*}
    H(t_{0a}, t_{0b}, \dotsc ,t_{1},\dotsc,t_{r})=\frac{1}{2\pi i}\frac{1}{(1-t_{0a})(1-t_{0b})\dotsm}  \int_{\abs{z} = 1} \prod_{k = 1}^{r} \frac{1}{(1-t_{k}z^{n_{k}})(1-t_{k}z^{-n_{k}})} \, z^{-1}\mathrm{d}z.
  \end{equation*}
\end{rem}

\begin{proof}[Proof of theorem~\ref{thm:O2-Hilbert-series}]
  Observe first that
  \begin{equation*}
    \det (I - t_{n}\rho_{n}(\sigma r_{\theta})) = 1-t_{n}^{2}, \quad \forall n \ge 1,
  \end{equation*}
  where the notation $\rho_{n}$ has been introduced in~\autoref{sec:representation-theory}. Now, from Molien--Weyl formulae (see theorem~\ref{thm:Molien--Weyl}) and the expression of the Haar measure on $\OO(2)$ (see definition~\ref{def:Haar-measure}),
  we deduce that
  \begin{equation*}
    H_{\rho}(t_{1},\dotsc,t_{r}) = \frac{1}{2} \left\{\mathcal{I}(t_{1},\dotsc,t_{r}) + \frac{1}{(1-t_{1}^{2})\dotsm(1-t_{r}^{2})} \right\},
  \end{equation*}
  where
  \begin{equation*}
    \mathcal{I}(t_{1},\dotsc,t_{r}) = \frac{1}{2\pi}\int_{0}^{2\pi}\prod_{k = 1}^{r}\frac{1}{\det (I - t_{k}\rho_{n_{k}}(r_{\theta}))}\,\mathrm{d}\theta.
  \end{equation*}
  But, for all $n\geq 1$, we have
  \begin{align*}
    \det(I-t\rho_{n}(r_{\theta})) & = 1 + t^{2}-2t\cos(n\theta)
    \\
                                  & = (1-t\mathrm{e}^{in\theta})(1-t\mathrm{e}^{-in\theta})
    \\
                                  & = (1-tz^{n})(1-tz^{-n}),\quad z = \mathrm{e}^{i\theta},
  \end{align*}
  and thus
  \begin{equation*}
    \mathcal{I}(t_{1},\dotsc,t_{r}) = \frac{1}{2\pi i} \int_{\abs{z} = 1} \prod_{k = 1}^{r} \frac{1}{(1-t_{k}z^{n_{k}})(1-t_{k}z^{-n_{k}})} \, z^{-1}\mathrm{d}z,
  \end{equation*}
  where
  \begin{align*}
    \prod_{k = 1}^{r} \frac{1}{(1-t_{k}z^{n_{k}})(1-t_{k}z^{-n_{k}})} & = \prod_{k = 1}^{r} \left( \sum_{i} (t_{k}z^{n_{k}})^{i}\right)\left( \sum_{j} (t_{k}z^{-n_{k}})^j\right)
    \\
                                                                      & = \sum_{\alpha_{1},\beta_{1},\dotsc,\alpha_{r},\beta_{r}} z^{(\alpha_{1}n_{1} + \dotsb + \alpha_{r}n_{r}) - (\beta_{1}n_{1} + \dotsb + \beta_{r}n_{r})} t_{1}^{\alpha_{1} + \beta_{1}}\dotsm t_{r}^{\alpha_{r} + \beta_{r}}.
  \end{align*}
  Setting $k_{i}:=\alpha_{i}+\beta_{i}$, this latest sum can now be recast as
  \begin{equation*}
    \sum_{k_{1}, \dotsc , k_{r}} \left(\sum_{0 \le \alpha_{l} \le k_{l}} z^{(2\alpha_{1}-k_{1})n_{1} + \dotsb + (2\alpha_{r}-k_{r})n_{r}} \right) t_{1}^{k_{1}} \dotsm t_{r}^{k_{r}},
  \end{equation*}
  and we get thus
  \begin{equation*}
    \mathcal{I}(t_{1},\dotsc,t_{r}) = \sum_{k_{1},\dotsc,k_{r}\geq 0} b_{k_{1},\dotsc,k_{r}}t_{1}^{k_{1}}\dotsc t_{r}^{k_{r}},
  \end{equation*}
  where
  \begin{equation*}
    b_{k_{1},\dotsc,k_{r}} = \frac{1}{2i\pi} \sum_{0 \le \alpha_{l} \le k_{l}} \int_{\abs{z} = 1} z^{(2\alpha_{1}-k_{1})n_{1} + \dotsb + (2\alpha_{r}-k_{r})n_{r}} \, z^{-1}\mathrm{d}z.
  \end{equation*}
  Moreover, since
  \begin{equation*}
    \frac{1}{2i\pi} \int_{\vert z\vert  = 1} z^{m}\, z^{-1}\mathrm{d}z =    \begin{cases}
      1 & \text{if} \quad m = 0, \\
      0 & \text{otherwise},
    \end{cases}
  \end{equation*}
  $b_{k_{1},\dotsc,k_{r}}$ is the number of solutions of the linear Diophantine equation
  \begin{equation*}
    2\alpha_{1}n_{1} + \dotsc  + 2\alpha_{r}n_{r} = k_{1}n_{1} + \dotsc  + k_{r}n_{r},\quad \alpha_{i}\geq 0 .
  \end{equation*}
\end{proof}

\begin{exa}
  Let $\Ela$ be the vector space of bidimensional elasticity tensors. Its harmonic decomposition (see example~\ref{ex:dec-harm-4-order}) writes
  \begin{equation*}
    \Ela \simeq 2\HT{0} \oplus \HT{2} \oplus \HT{4}.
  \end{equation*}
  Let us associate to the components $(\lambda, \mu, \bh, \bH)$, the formal variables $(t_{0a},t_{0b},t_{2},t_{4})$. Then, using theorem~\ref{thm:O2-Hilbert-series}, remarks~\ref{rem:full-O2-Hilbert-series} and~\ref{rem:SO2-Hilbert-series}, we get
  \begin{equation*}
    H_{(\Ela,\SO(2))}(t_{0a},t_{0b},t_{2},t_{4}) = \frac{1+t_{2}^{2}t_{4}}{(1-t_{0a})(1-t_{0b})(1-t_{2}^{2})(1-t_{4}^{2})(1-t_{2}^{2}t_{4})},
  \end{equation*}
  and
  \begin{equation*}
    H_{(\Ela,\OO(2))}(t_{0a},t_{0b},t_{2},t_{4}) = \frac{1}{(1-t_{0a})(1-t_{0b})(1-t_{2}^{2})(1-t_{4}^{2})(1-t_{2}^{2}t_{4})}.
  \end{equation*}
  The Taylor expansion of $H_{(\Ela,\SO(2))}$ writes
  \begin{multline*}
    H_{(\Ela,\SO(2))}(t_{0a},t_{0b},t_{2},t_{4}) = 1+t_{0a}+t_{0b}+t_{0a}t_{0b}+t_{0a}^{2}+t_{0b}^{2}+t_{2}^{2}+t_{4}^{2}+\dotsb \\
    +2\,t_{2}^{2}t_{4}+\dotsb +3\,t_{0a}t_{2}^{4}t_{4}^{2}+\dotsb +11\,t_{2}^{24}t_{4}^{10}+\dotsb
  \end{multline*}
  We deduce from this expansion that the space of homogeneous invariants of multi-degree $(1,0,0,0)$, corresponding to $t_{0a}$, is of dimension $1$ (and spanned by $\lambda$) and that the space of homogeneous invariants of multi-degree $(0,0,2,0)$ corresponding to $t_{2}^{2}$ is of dimension $1$ (and spanned by $\tr \ba^{2}$).
\end{exa}

\section{Proofs}
\label{sec:proofs}

It is useful in invariant theory to introduce the following bi-differential operators
\begin{equation*}
  \triangle_{\alpha \beta} := \frac{\partial^{2}}{\partial x_{\alpha} \partial x_{\beta}} + \frac{\partial^{2}}{\partial y_{\alpha} \partial y_{\beta}}, \qquad \Omega_{\alpha \beta} := \frac{\partial^{2}}{\partial x_{\alpha} \partial y_{\beta}} - \frac{\partial^{2}}{\partial y_{\alpha} \partial x_{\beta}},
\end{equation*}
which are known respectively as the \emph{polarized Laplacian} and the \emph{Cayley operator}. The polarized Laplacian $\triangle_{\alpha \beta}$ is invariant under $\OO(2)$ and the Cayley operator $\Omega_{\alpha \beta}$ is invariant under $\SL(2,\RR)$. Anyway, they are both $\SO(2)$-invariant. Since we shall use the complex variables $(z,\overline{z})$ rather than the real variables $(x,y)$, we introduce also the complex differential operators
\begin{equation*}
  \frac{\partial}{\partial z} = \frac{1}{2}\left(\frac{\partial}{\partial x} - i \frac{\partial}{\partial y}\right), \qquad \frac{\partial}{\partial \overline{z}} = \frac{1}{2}\left(\frac{\partial}{\partial x} + i \frac{\partial}{\partial y}\right),
\end{equation*}
and the complex bi-differential operator
\begin{equation*}
  D_{\alpha \overline{\beta}} := \frac{\partial^{2}}{\partial z_{\alpha} \partial \overline{z}_{\beta}},
\end{equation*}
and we get
\begin{equation*}
  D_{\alpha \overline{\beta}} =  \frac{1}{4}\left(\triangle_{\alpha \beta} + i\Omega_{\alpha \beta} \right),
\end{equation*}
or
\begin{equation*}
  \triangle_{\alpha \beta} =  2\left(D_{\alpha \overline{\beta}} + D_{\beta \overline{\alpha}}\right), \qquad \Omega_{\alpha \beta} = -2i\left(D_{\alpha \overline{\beta}} - D_{\beta \overline{\alpha}}\right).
\end{equation*}

Given two homogeneous polynomials $\rp_{1}$, $\rp_{2}$ of respective degree $n_{1}$ and $n_{2}$, we have then
\begin{align*}
  \left\{\rp_{1},\rp_{2}\right\}_{r} & = \frac{(n_{1}-r)!(n_{2}-r)!}{n_{1}! n_{2}!} \left( \Delta_{\alpha \beta}^{r} \rp_{1}(\xx_{\alpha}) \rp_{2}(\xx_{\beta}) \right)_{\xx_{\alpha}=\xx_{\beta}=\xx}
  \\
                                     & = 2^{r} \frac{(n_{1}-r)!(n_{2}-r)!}{n_{1}! n_{2}!} \left( \left(D_{\alpha \overline{\beta}} + D_{\beta \overline{\alpha}}\right)^{r} \rp_{1}(\xx_{\alpha}) \rp_{2}(\xx_{\beta}) \right)_{\xx_{\alpha}=\xx_{\beta}=\xx}
  \\
                                     & = 2^{r} \frac{(n_{1}-r)!(n_{2}-r)!}{n_{1}! n_{2}!} \sum_{k=0}^{r} \binom{r}{k} \frac{\partial^{r} \rp_{1}}{\partial z^{k}\partial \overline{z}^{(r-k)}}\frac{\partial^{r}\rp_{2}}{\partial z^{(r-k)}\partial \overline{z}^{k}}.
\end{align*}

Now, if we apply this formula to two harmonic polynomials $\rh_{1} = \Re (\overline{z}_{1}z^{n_{1}})$ and $\rh_{2} = \Re (\overline{z}_{2}z^{n_{2}})$ with $1 \le r \le n_{1} \le n_{2}$, the only non-vanishing terms in the sum correspond to $k=0$ and $k=r$. We get thus
\begin{equation}\label{eq:transvectant-h1-h2}
  \left\{\rh_{1},\rh_{2}\right\}_{r} = 2^{r-1} (z\overline{z})^{n_{1}-r}\Re\left( z_{1}\overline{z}_{2}z^{n_{2}-n_{1}}\right).
\end{equation}

On the other hand, we have
\begin{equation*}
  \begin{aligned}
    \left[\rp_{1},\rp_{2}\right] & = -\frac{1}{n_{1}n_{2}} \left( \Omega_{\alpha \beta} \rp_{1}(\xx_{\alpha}) \rp_{2}(\xx_{\beta}) \right)_{\xx_{\alpha}=\xx_{\beta}=\xx}
    \\
                                 & = \frac{2i}{n_{1}n_{2}} \left( (D_{\alpha \overline{\beta}} - D_{\beta \overline{\alpha}}) \rp_{1}(\xx_{\alpha}) \rp_{2}(\xx_{\beta}) \right)_{\xx_{\alpha}=\xx_{\beta}=\xx}
    \\
                                 & = \frac{2i}{n_{1}n_{2}} \left( \frac{\partial \rp_{1}}{\partial z}\frac{\partial \rp_{2}}{\partial \overline{z}} - \frac{\partial \rp_{1}}{\partial \overline{z}}\frac{\partial \rp_{2}}{\partial z} \right).
  \end{aligned}
\end{equation*}
Applying this formula, first, to $\rh_{1} = \Re (\overline{z}_{1}z^{n_{1}}) $ and $\rh_{2} = \Re (\overline{z}_{2}z^{n_{2}})$, where $1 \le n_{1} \le n_{2}$, we get
\begin{equation}\label{eq:LP-h1-h2}
  \left[\rh_{1},\rh_{2}\right] =  (z\overline{z})^{n_{1}-1} \Im \left(z_{1}\overline{z}_{2} z^{n_{2}-n_{1}}\right),
\end{equation}
and, then, to the Euclidean metric $\qq := x^{2}+y^{2} = z\overline{z}$ and $\rh_{1} = \Re (\overline{z}_{1}z^{n_{1}})$, where $1 \le n_{1}$, we get
\begin{equation}\label{eq:LP-q-h1}
  \left[\qq,\rh_{1}\right] = \Im(\overline{z}_{1}z^{n_{1}}) = \widetilde{\rh}_{1}.
\end{equation}

Next, observe that given $p$ harmonic polynomials $\rh_{1}, \dotsc , \rh_{p}$, where $\rh_{k} = \Re(\overline{z}_{k}z^{n_{k}})$, the leading harmonic part of the product $\rh_{1} \dotsm \rh_{p}$
writes
\begin{equation}\label{eq:product-harmonic-part}
  (\rh_{1} \dotsm \rh_{p})^{0} = (\Re(\overline{z}_{1}z^{n_{1}}) \dotsm \Re(\overline{z}_{p}z^{n_{p}}))^{0} = \frac{1}{2^{p-1}} \Re(\overline{z}_{k_{1}}\dotsm \overline{z}_{k_{p}} z^{n_{k_{1}}+ \dotsb + n_{k_{p}}})
\end{equation}
according to~\autoref{sec:Sym-harmonic-decomposition}. For instance, if $1\le n_{1}\le n_{2} \le n_{3}$ and $n_{3} \le n_{1} + n_{2}$, we have
\begin{multline*}
  \Re(\overline{z}_{1} z^{n_{1}}) \Re(\overline{z}_{2} z^{n_{2}}) \Re(\overline{z}_{3} z^{n_{3}}) = \frac{1}{2^{2}}\left\{ \Re(\overline{z}_{1}\overline{z}_{2}\overline{z}_{3}z^{n_{1}+n_{2}+n_{3}}) + \qq^{n_{1}}\Re(z_{1}\overline{z}_{2}\overline{z}_{3}z^{-n_{1}+n_{2}+n_{3}}) \right.
  \\
  \left. + \qq^{n_{2}}\Re(\overline{z}_{1}z_{2}\overline{z}_{3}z^{n_{1}-n_{2}+n_{3}}) + \qq^{n_{3}}\Re(\overline{z}_{1}\overline{z}_{2}z_{3}z^{n_{1}+n_{2}-n_{3}})\right\},
\end{multline*}
and its leading harmonic term, of degree $n_{1}+n_{2}+n_{3}$ in $z, \overline{z}$, is
\begin{equation*}
  (\Re(\overline{z}_{1} z^{n_{1}}) \Re(\overline{z}_{2} z^{n_{2}}) \Re(\overline{z}_{3} z^{n_{3}}))^0 = \frac{1}{2^{2}}\Re(\overline{z}_{1}\overline{z}_{2}\overline{z}_{3}z^{n_{1}+n_{2}+n_{3}}).
\end{equation*}

Finally, given an homogeneous harmonic polynomial $\rh$ of degree $n$ and $r \ge 0$, we have
\begin{equation*}
  \triangle^{r}(\qq^{r}\rh) = 4^{r}\frac{r!(n+r)!}{n!}\rh, 
\end{equation*}
and thus, for $n_{1}\leq n_{2}$, we get by~\eqref{eq:LP-h1-h2}
\begin{equation}\label{eq:Laplacian-LP}
  \triangle^{n_{1}-1}\left[\rh_{1},\rh_{2}\right] = \frac{4^{n_{1}-1}(n_{1}-1)!(n_{2}-1)!}{(n_{2}-n_{1})!} \Im \left(z_{1}\overline{z}_{2} z^{n_{2}-n_{1}}\right),
\end{equation}
since $\Im \left(z_{1}\overline{z}_{2} z^{n_{2}-n_{1}}\right)$ is of degree $n_{2}-n_{1}$.

\begin{proof}[Proof of theorem~\ref{thm:1xH}]
  Let $\bH_{k} \in \HT{n_{k}}$, where $n_{k} \ge 1$ and set $\rh_{k} = \phi(\bH_{k})$. Then,
  \begin{equation*}
    \widetilde \rh_{k} = \Im (\overline{z}_{k}z^{n_{k}}) =  \left[\qq,\rh_{k}\right] = \phi(\id \times \bH_{k}),
  \end{equation*}
  by~\eqref{eq:LP-q-h1}, and as $\id \times \bH_{k}$ translates as $[\qq, \rh_{k}]$, according to~\autoref{sec:tensors}.
\end{proof}

\begin{proof}[Proof of theorem~\ref{thm:harmonic-product}]
  Let $\bH_{j}\in \HT{n_{k_{j}}}$ be harmonic tensors where $n_{k_{j}} \ge 1$, for $1 \le j \le p$ and set $\phi(\bH_{j}) = \rh_{j} = \Re(\overline{z}_{j} z^{n_{k_{j}}})$. Then, $\phi\left(\bH_{1} \odot \dotsb \odot \bH_{p}\right) = \rh_{1} \dotsm \rh_{p}$ and by~\eqref{eq:product-harmonic-part}, we get thus
  \begin{align*}
    \phi\left( (\bH_{1} \odot \dotsb \odot \bH_{p})^{\prime} \right) & = (\rh_{1} \dotsm \rh_{p})^{0}
    \\
                                                                     & = \frac{1}{2^{p-1}} \Re(\overline{z}_{1}\dotsm \overline{z}_{p} z^{n_{k_{1}}+ \dotsb + n_{k_{p}}}).
  \end{align*}
\end{proof}

\begin{proof}[Proof of theorem~\ref{thm:monomials-translation}]
  Let $\bH_{j}\in \HT{n_{k_{j}}}$ be harmonic tensors where $n_{k_{j}} \ge 1$ for $1 \le j \le p+s$ and set $\phi(\bH_{j}) = \rh_{j} = \Re(\overline{z}_{j} z^{n_{k_{j}}})$, $Z_{1} = z_{1}\dotsm z_{p}$, $Z_{2} = z_{p+1}\dotsm z_{p+s}$, $N_{1} = n_{k_{1}} + \dotsb + n_{k_{p}}$ and $N_{2} = n_{k_{p+1}} + \dotsb + n_{k_{p+s}}$. There is no loss of generality in assuming that $N_{1} \le N_{2}$. Then, by~\eqref{eq:transvectant-h1-h2} with $r=N_{1}$, we get
  \begin{equation*}
    \Re\left(z_{1}\dotsm z_{p} \zb_{p+1} \dotsm \zb_{p+s} z^{N_{2}-N_{1}}\right) = \Re(Z_{1} \overline{Z}_{2} z^{N_{2}-N_{1}}) = \frac{1}{2^{N_{1}-1}} \left\{\mathrm{H}_{1},\mathrm{H}_{2}\right\}_{N_{1}},
  \end{equation*}
  where $\mathrm{H}_{i} := \Re (\overline{Z}_{i}z^{N_{i}})$ for $i=1,2$. But, by~\eqref{eq:product-harmonic-part}, we have
  \begin{equation*}
    \mathrm{H}_{1} = 2^{p-1} (\rh_{1} \dotsm \rh_{p})^{0} = 2^{p-1}\phi\left( (\bH_{1} \odot \dotsb \odot \bH_{p})^{\prime} \right),
  \end{equation*}
  and
  \begin{equation*}
    \mathrm{H}_{2} = 2^{s-1} (\rh_{p+1} \dotsm \rh_{p+s})^{0} = 2^{s-1}\phi\left( (\bH_{p+1} \odot \dotsb \odot \bH_{p+s})^{\prime} \right).
  \end{equation*}
  But, according to~\autoref{sec:tensors}
  \begin{equation*}
    \left\{\mathrm{H}_{1},\mathrm{H}_{2}\right\}_{N_{1}} = 2^{p+s-2} \phi\left( (\bH_{1} \odot \dotsb \odot \bH_{p})^{\prime} \rdots{N_{1}} (\bH_{p+1} \odot \dotsb \odot \bH_{p+s})^{\prime} \right).
  \end{equation*}
  We get therefore
  \begin{multline*}
    \Re\left(z_{1}\dotsm z_{p} \zb_{p+1} \dotsm \zb_{p+s} z^{N_{2}-N_{1}}\right)
    \\
    = 2^{(p+s-1-N_{1})} \phi\left( (\bH_{1} \odot \dotsb \odot \bH_{p})^{\prime} \rdots{N_{1}} (\bH_{p+1} \odot \dotsb \odot \bH_{p+s})^{\prime} \right),
  \end{multline*}
  which is the first identity of theorem~\ref{thm:monomials-translation}. Next, by~\eqref{eq:Laplacian-LP}, we get
  \begin{align*}
    \Im\left(z_{1}\dotsm z_{p} \zb_{p+1} \dotsm \zb_{p+s} z^{N_{2}-N_{1}}\right) & = \Im(Z_{1}\overline{Z}_{2}z^{N_{2}-N_{1}})
    \\
                                                                                 & = \frac{(N_{2}-N_{1})!}{4^{N_{1}-1}(N_{1}-1)!(N_{2}-1)!} \triangle^{N_{1}-1}\left[\mathrm{H}_{1},\mathrm{H}_{2}\right].
  \end{align*}
  Besides, we have
  \begin{equation*}
    \Delta^{r}\phi(\bS) = \frac{n!}{(n-2r)!} \phi(\tr^{r} \bS),
  \end{equation*}
  for every symmetric tensor $\bS$ or order $n$. Applying this formula to $\phi(\bS) = \left[\mathrm{H}_{1},\mathrm{H}_{2}\right]$, of degree $n = N_{1}+N_{2}-2$, with $r = N_{1}-1$ and where
  \begin{equation*}
    \bS = 2^{p+s-2}\left( (\bH_{1} \odot \dotsb \odot \bH_{p})^{\prime} \times (\bH_{p+1} \odot \dotsb \odot \bH_{p+s})^{\prime} \right),
  \end{equation*}
  we get
  \begin{equation*}
    \triangle^{N_{1}-1}\left[\mathrm{H}_{1},\mathrm{H}_{2}\right] = 2^{p+s-2} \frac{(N_{1}+N_{2}-2)!}{(N_{2}-N_{1})!} \phi\left(\tr^{(N_{1}-1)} \left( (\bH_{1} \odot \dotsb \odot \bH_{p})^{\prime} \times (\bH_{p+1} \odot \dotsb \odot \bH_{p+s})^{\prime}\right) \right),
  \end{equation*}
  and thus
  \begin{multline*}
    \Im\left(z_{1}\dotsm z_{p} \zb_{p+1} \dotsm \zb_{p+s} z^{N_{2}-N_{1}}\right)
    \\
    = \frac{2^{(p + s - 2N_{1})}(N_{1}+N_{2}-2)!}{(N_{1}-1)!(N_{2}-1)!} \phi \left(\tr^{(N_{1}-1)} \left((\bH_{1} \odot \dotsb \odot \bH_{p})^{\prime} \times (\bH_{p+1} \odot \dotsb \odot \bH_{p+s})^{\prime}\right)\right),
  \end{multline*}
  which is the second identity of theorem~\ref{thm:monomials-translation}. It remains to show that
  \begin{align*}
    \Im(z_{1}\dotsm z_{p} \zb_{p+1} \dotsm \zb_{p+s} z^{N_{2}-N_{1}}) & = -2^{(p + s-1-N_{1})} \phi\left(([\id\times\bH_{1}] \odot \dotsb \odot \bH_{p})^{\prime} \rdots{N_{1}} (\bH_{p+1} \odot \dotsb \odot \bH_{p+s})^{\prime}\right)
    \\
                                                                      & = 2^{(p + s-1-N_{1})} \phi\left((\bH_{1} \odot \dotsb \odot \bH_{p})^{\prime} \rdots{N_{1}} ([\id\times\bH_{p+1}] \odot \dotsb \odot \bH_{p+s})^{\prime}\right).
  \end{align*}
  To do so, we use the fact that $\Im z = \Re (-iz)$ and thus that $\widetilde{\rh}_{k} = \Re(\overline{iz_{k}}z^{n_{k}})$. We have therefore
  \begin{align*}
    \Im(z_{k_{1}}\dotsm z_{k_{p}} \zb_{k_{p+1}} \dotsm \zb_{k_{p+s}} z^{N_{2}-N_{1}}) & = \Re\left((-iz_{k_{1}})\dotsm z_{k_{p}} \zb_{k_{p+1}} \dotsm \zb_{k_{p+s}} z^{N_{2}-N_{1}}\right)
    \\
                                                                                      & = -\Re\left(iZ_{1}\overline{Z}_{2} z^{N_{2}-N_{1}}\right)
    \\
                                                                                      & = -\frac{1}{2^{N_{1}-1}} \left\{\widetilde{\mathrm{H}_{1}},\mathrm{H}_{2}\right\}_{N_{1}}
    \\
                                                                                      & = - \frac{2^{p+s-2}}{2^{N_{1}-1}}\left\{(\widetilde{\rh}_{k_{1}}\dotsm \rh_{k_{p}})',(\rh_{k_{p}+1} \dotsm \rh_{k_{p}+s})'\right\}_{N_{1}}
    \\
                                                                                      & = - 2^{(p + s-1-N_{1})} \phi\left(([\id\times\bH_{1}] \odot \dotsb \odot \bH_{p})^{\prime} \rdots{N_{1}} (\bH_{p+1} \odot \dotsb \odot \bH_{p+s})^{\prime}\right),
  \end{align*}
  by~\eqref{eq:transvectant-h1-h2} and~\eqref{eq:product-harmonic-part}. The same way, we deduce that
  \begin{align*}
    \Im(z_{k_{1}}\dotsm z_{k_{p}} \zb_{k_{p+1}} \dotsm \zb_{k_{p+s}} z^{N_{2}-N_{1}}) & = \Re\left(Z_{1}\overline{iZ_{2}} z^{N_{2}-N_{1}}\right)
    \\
                                                                                      & = \frac{1}{2^{N_{1}-1}} \left\{\mathrm{H}_{1},\widetilde{\mathrm{H}_{2}}\right\}_{N_{1}}
    \\
                                                                                      & = 2^{(p + s-1-N_{1})} \phi\left((\bH_{1} \odot \dotsb \odot \bH_{p})^{\prime} \rdots{N_{1}} ([\id\times\bH_{p+1}] \odot \dotsb \odot \bH_{p+s})^{\prime}\right).
  \end{align*}
\end{proof}



\begin{thebibliography}{100}

\bibitem{AS2018}
H.~Abdoul-Anziz and P.~Seppecher.
\newblock Strain gradient and generalized continua obtained by homogenizing
  frame lattices.
\newblock {\em Mathematics and Mechanics of Complex Systems}, 6(3):213--250,
  July 2018.

\bibitem{AS1999}
A.~Ashtekar and T.~A. Schilling.
\newblock Geometrical formulation of quantum mechanics.
\newblock In {\em On {E}instein's path ({N}ew {Y}ork, 1996)}, pages 23--65.
  Springer, New York, 1999.

\bibitem{AHQ2019}
N.~Auffray, Q.~He, and H.~L. Quang.
\newblock Complete symmetry classification and compact matrix representations
  for 3d strain gradient elasticity.
\newblock {\em International Journal of Solids and Structures}, 159:197--210,
  Mar. 2019.

\bibitem{AKO2016}
N.~Auffray, B.~Kolev, and M.~Olive.
\newblock Handbook of bi-dimensional tensors: Part i: Harmonic decomposition
  and symmetry classes.
\newblock {\em Mathematics and Mechanics of Solids}, 22(9):1847--1865, may
  2016.

\bibitem{AKP2013}
N.~Auffray, B.~Kolev, and M.~Petitot.
\newblock On anisotropic polynomial relations for the elasticity tensor.
\newblock {\em Journal of Elasticity}, 115(1):77--103, jun 2013.

\bibitem{AQH2013}
N.~Auffray, H.~L. Quang, and Q.~He.
\newblock Matrix representations for 3d strain-gradient elasticity.
\newblock {\em Journal of the Mechanics and Physics of Solids},
  61(5):1202--1223, May 2013.

\bibitem{AR2016}
N.~Auffray and P.~Ropars.
\newblock Invariant-based reconstruction of bidimensional elasticity tensors.
\newblock {\em International Journal of Solids and Structures}, 87:183--193,
  June 2016.

\bibitem{Bac1970}
G.~Backus.
\newblock A geometrical picture of anisotropic elastic tensors.
\newblock {\em Reviews of Geophysics}, 8(3):633, 1970.

\bibitem{Bae1993}
R.~Baerheim.
\newblock Harmonic decomposition of the anisotropic elasticity tensor.
\newblock {\em Quart. J. Mech. Appl. Math.}, 46(3):391--418, 1993.

\bibitem{Ber2016}
A.~Bertram.
\newblock Compendium on gradient materials.
\newblock {\em OvGU, Magdeburg}, 2016.

\bibitem{Bet1987}
J.~Betten.
\newblock Irreducible invariants of fourth-order tensors.
\newblock {\em Mathematical Modelling}, 8:29--33, 1987.
\newblock Mathematical modelling in science and technology (Berkeley, Calif.,
  1985).

\bibitem{BH1995}
J.~Betten and W.~Helisch.
\newblock Integrity bases for a fourth-rank tensor.
\newblock In {\em Solid Mechanics and Its Applications}, volume~39 of {\em
  Solid Mech. Appl.}, pages 37--42. Springer Netherlands, 1995.

\bibitem{BOR1996}
A.~Blinowski, J.~Ostrowska-Maciejewska, and J.~Rychlewski.
\newblock {T}wo-dimensional {H}ooke's tensors---isotropic decomposition,
  effective symmetry criteria.
\newblock {\em Arch. Mech. (Arch. Mech. Stos.)}, 48(2):325--345, 1996.

\bibitem{Boe1987}
J.-P. Boehler.
\newblock {\em {A}pplication of tensor functions in solid mechanics}.
\newblock CISM Courses and Lectures. Springer-Verlag, Wien, 1987.

\bibitem{BKO1994}
J.~P. Boehler, A.~A. Kirillov, and E.~T. Onat.
\newblock On the polynomial invariants of the elasticity tensor.
\newblock {\em Journal of Elasticity}, 34(2):97--110, Feb. 1994.

\bibitem{BBS2008}
A.~Bóna, I.~Bucataru, and M.~Slawinski.
\newblock {S}pace of ${SO}(3)$-orbits of elasticity tensors.
\newblock {\em Arch. Mech.}, 60(2):123--138, 2008.

\bibitem{BI2010}
W.~Bruns and B.~Ichim.
\newblock {N}ormaliz: algorithms for affine monoids and rational cones.
\newblock {\em Journal of Algebra}, 324(5):1098--1113, Sept. 2010.

\bibitem{CDC2018}
V.~Caccuri, R.~Desmorat, and J.~Cormier.
\newblock Tensorial nature of $\gamma$$\prime$-rafting evolution in
  nickel-based single crystal superalloys.
\newblock {\em Acta Materialia}, 158:138--154, Oct. 2018.

\bibitem{Car1981}
E.~Cartan.
\newblock {\em The theory of spinors}.
\newblock Dover Publications, Inc., New York, 1981.
\newblock With a foreword by Raymond Streater, A reprint of the 1966 English
  translation, Dover Books on Advanced Mathematics.

\bibitem{Cha1984}
J.~Chaboche.
\newblock Anisotropic creep damage in the framework of continuum damage
  mechanics.
\newblock {\em Nuclear Engineering and Design}, 79(3):309--319, June 1984.

\bibitem{Cha1978}
J.-L. Chaboche.
\newblock {\em Description thermodynamique et phénoménologique de la
  viscoplasticité cyclique avec endommagement.}
\newblock Universit\'e Paris VI and Technical report Onera, 1978.

\bibitem{CF1957}
E.~G. Coker and U.~N.~G. Filon.
\newblock {\em A Treatise on Photoelasticity, revised by H. T. Hessop}.
\newblock Cambridge University Press, Cambridge MA, 1957.

\bibitem{CW2010}
F.~Cormery and H.~Welemane.
\newblock A stress-based macroscopic approach for microcracks unilateral
  effect.
\newblock {\em Comp. Mat. Sci.}, 47:727--738, 2010.

\bibitem{Cow1989}
S.~C. Cowin.
\newblock {P}roperties of the anisotropic elasticity tensor.
\newblock {\em Q. J. Mech. Appl. Math.}, 42:249--266, 1989.

\bibitem{CM1990}
S.~C. Cowin and M.~M. Mehrabadi.
\newblock {E}igentensors of linear anisotropic elastic materials.
\newblock {\em Q. J. Mech. Appl. Math.}, 43:15--41, 1990.

\bibitem{dSV2013}
G.~de~Saxc\'e and C.~Vall\'ee.
\newblock {I}nvariant measures of the lack of symmetry with respect to the
  symmetry groups of 2{D} elasticity tensors.
\newblock {\em J. Elasticity}, 111:21--39, 2013.

\bibitem{DK2015}
H.~Derksen and G.~Kemper.
\newblock {\em Computational Invariant Theory}, volume 130 of {\em
  Encyclopaedia of Mathematical Sciences}.
\newblock Springer Berlin Heidelberg, enlarged edition, 2015.
\newblock With two appendices by Vladimir L. Popov, and an addendum by Norbert
  A'Campo and Popov, Invariant Theory and Algebraic Transformation Groups,
  VIII.

\bibitem{DA2019}
B.~Desmorat and N.~Auffray.
\newblock Space of 2{D} elastic materials: a geometric journey.
\newblock {\em Continuum Mechanics and Thermodynamics}, 31(4):1205--1229, may
  2019.

\bibitem{DD2015}
B.~Desmorat and R.~Desmorat.
\newblock Tensorial polar decomposition of {2D} fourth-order tensors.
\newblock {\em Comptes Rendus M{é}canique}, 343(9):471--475, sep 2015.

\bibitem{DAD2019}
R.~Desmorat, N.~Auffray, B.~Desmorat, B.~Kolev, and M.~Olive.
\newblock Generic separating sets for three-dimensional elasticity tensors.
\newblock {\em Proc. R. Soc. A}, 475, 2019.

\bibitem{DL1985/86}
J.~Dixmier and D.~Lazard.
\newblock {L}e nombre minimum d'invariants fondamentaux pour les formes
  binaires de degr\'e {$7$}.
\newblock {\em Portugal. Math.}, 43(3):377--392, 1985/86.

\bibitem{DK2016}
L.~Dormieux and D.~Kondo.
\newblock {\em Micromechanics of Fracture and Damage}.
\newblock John Wiley {\&} Sons, Inc., may 2016.

\bibitem{DTdLac1993}
E.~Du~Trémolet~de Lacheisserie.
\newblock {\em Magnetostriction : theory and applications of
  magnetoelasticity}.
\newblock CRC Press, Boca Raton, 1993.

\bibitem{FV1997}
S.~Forte and M.~Vianello.
\newblock {S}ymmetry classes and harmonic decomposition for photoelasticity
  tensors.
\newblock {\em International Journal of Engineering Science},
  35(14):1317--1326, 1997.

\bibitem{FV2014}
S.~Forte and M.~Vianello.
\newblock {A} unified approach to invariants of plane elasticity tensors.
\newblock {\em Meccanica}, 49(9):2001--2012, Mar. 2014.

\bibitem{GW2002a}
G.~Geymonat and T.~Weller.
\newblock Symmetry classes of piezoelectric solids.
\newblock {\em Comptes rendus de l'Acad\'emie des Sciences. S\'erie I},
  335:847--8524, 2002.

\bibitem{GSS1988}
M.~Golubitsky, I.~Stewart, and D.~G. Schaeffer.
\newblock {\em {S}ingularities and groups in bifurcation theory. {V}ol. {II}},
  volume~69 of {\em Applied Mathematical Sciences}.
\newblock Springer-Verlag, New York, 1988.

\bibitem{Gor1868}
P.~Gordan.
\newblock {B}eweis, dass jede {C}ovariante und {I}nvariante einer {B}ineren
  {F}orm eine ganze {F}unction mit numerischen {C}oefficienten einer endlichen
  {A}nzahl solcher {F}ormen ist.
\newblock {\em Journal für die reine und angewandte Mathematik},
  69:323--354, 1868.

\bibitem{Gor1873}
P.~Gordan.
\newblock {U}eber die {A}uflosung linearen {G}leidungen mit reallen
  {C}oefficienten.
\newblock {\em Math. Ann.}, (6):23--28, 1873.

\bibitem{Gor1875}
P.~Gordan.
\newblock {\em {U}ber das {F}ormensystem {B}inaerer {F}ormen}.
\newblock 1875.

\bibitem{Gor1987}
P.~Gordan.
\newblock {\em {V}orlesungen \"uber {I}nvariantentheorie}.
\newblock Chelsea Publishing Co., New York, second edition, 1987.
\newblock Erster Band: Determinanten. [Vol. I: Determinants], Zweiter Band:
  Binäre Formen. [Vol. II: Binary forms], Edited by Georg Kerschensteiner.

\bibitem{Gur1973}
M.~E. Gurtin.
\newblock The linear theory of elasticity.
\newblock In {\em Linear theories of elasticity and thermoelasticity}, pages
  1--295. Springer, 1973.

\bibitem{HZ1996}
Q.-C. He and Q.-S. Zheng.
\newblock {O}n the symmetries of 2{D} elastic and hyperelastic tensors.
\newblock {\em J. Elasticity}, 43(3):203--225, 1996.

\bibitem{Hil1993}
D.~Hilbert.
\newblock {\em {T}heory of algebraic invariants}.
\newblock Cambridge University Press, Cambridge, 1993.

\bibitem{Hil1948}
R.~Hill.
\newblock A theory of the yielding and plastic flow of anisotropic metals.
\newblock {\em Proceedings of the Royal Society of London A: Mathematical,
  Physical and Engineering Sciences}, 193(1033):281--297, 1948.

\bibitem{Hub2019}
O.~Hubert.
\newblock Multiscale magneto-elastic modeling of magnetic materials including
  isotropic second order stress effect.
\newblock {\em Journal of Magnetism and Magnetic Materials}, 491:165564, dec
  2019.

\bibitem{Mathematica}
W.~R. Inc.
\newblock Mathematica, {V}ersion 12.1.
\newblock Champaign, IL, 2020.

\bibitem{JT1984}
S.~Jemiolo and J.~Telega.
\newblock Fabric tensors in bone mechanics.
\newblock {\em Engineering Transactions}, 46(1), 1984.

\bibitem{Ju1989}
J.~Ju.
\newblock On energy-based coupled elastoplastic damage theories: Constitutive
  modeling and computational aspects.
\newblock {\em International Journal of Solids and Structures}, 25(7):803--833,
  1989.

\bibitem{Kac1993}
M.~Kachanov.
\newblock {\em Elastic solids with many cracks and related problems}, volume~1,
  pages 259--445.
\newblock J. Hutchinson and T. Wu Ed., Academic Press Pub., 1993.

\bibitem{KBS2012}
M.~Kadic, T.~Bückmann, N.~Stenger, M.~Thiel, and M.~Wegener.
\newblock On the practicability of pentamode mechanical metamaterials.
\newblock {\em Applied Physics Letters}, 100(19):191901, may 2012.

\bibitem{Kan1984}
K.-i. Kanatani.
\newblock Distribution of directional data and fabric tensors.
\newblock {\em International Journal of Engineering Science}, 22(2):149--164,
  1984.

\bibitem{KS1974}
E.~Kiral and G.~F. Smith.
\newblock {O}n the constitutive relations for anisotropic materials ---
  {T}riclinic, monoclinic, rhombic, tetragonal and hexagonal crystal systems.
\newblock {\em Int. J. Eng. Sci.}, 12:471--490, 1974.

\bibitem{KO09}
K.~Kowalczyk-Gajewska and J.~Ostrowska-Maciejewska.
\newblock Review on spectral decomposition of hooke’s tensor for all symmetry
  groups of linear elastic material.
\newblock {\em Engineering Transactions}, 57(3-4):145--183, 2009.

\bibitem{KP2000}
H.~Kraft and C.~Procesi.
\newblock {C}lassical {I}nvariant {T}heory, a {P}rimer.
\newblock Lectures notes avaiable at
  \url{http://www.math.unibas.ch/~kraft/Papers/KP-Primer.pdf}, 2000.

\bibitem{Kry2006}
S.~L. Kryvyi.
\newblock Algorithms for solving systems of linear diophantine equations in
  integer domains.
\newblock {\em Cybernetics and Systems Analysis}, 42(2):163--175, Mar. 2006.

\bibitem{KR1984}
J.~P.~S. Kung and G.-C. Rota.
\newblock {T}he invariant theory of binary forms.
\newblock {\em Bull. Amer. Math. Soc. (N.S.)}, 10(1):27--85, 1984.

\bibitem{LC1985}
J.~Lemaitre and J.-L. Chaboche.
\newblock {\em M\'ecanique des mat\'eriaux solides}.
\newblock Dunod, english translation 1990 'Mechanics of Solid Materials'
  Cambridge University Press, 1985.

\bibitem{LD2005}
J.~Lemaitre and R.~Desmorat.
\newblock {\em Engineering Damage Mechanics}.
\newblock Springer-Verlag, 2005.

\bibitem{LO2017}
R.~Lercier and M.~Olive.
\newblock Covariant algebra of the binary nonic and the binary decimic.
\newblock {\em AMS comtemporary mathematics}, 686, 2017.
\newblock \url{https://doi.org/10.1090/conm/686}.

\bibitem{LK1993}
V.~Lubarda and D.~Krajcinovic.
\newblock Damage tensors and the crack density distribution.
\newblock {\em Int. J. Solids Structures}, 30:2859--2877, 1993.

\bibitem{MC1995}
G.~W. Milton and A.~V. Cherkaev.
\newblock Which elasticity tensors are realizable?
\newblock {\em Journal of Engineering Materials and Technology},
  117(4):483--493, oct 1995.

\bibitem{MZC2019}
Z.~Ming, L.~Zhang, and Y.~Chen.
\newblock An irreducible polynomial functional basis of two-dimensional
  {E}shelby tensors.
\newblock {\em Applied Mathematics and Mechanics}, 40(8):1169--1180, May 2019.

\bibitem{Oda1982}
M.~Oda.
\newblock Fabric tensor for discontinuous geological materials.
\newblock {\em Soils and Foundations}, 22(4):96--108, dec 1982.

\bibitem{Oli2014a}
M.~Olive.
\newblock {\em G\'eom\'etrie des espaces de tenseurs, une approche effective
  appliqu\'ee \`a la m\'ecanique des milieux continus}.
\newblock PhD thesis, Aix-Marseille Universit{é}, 2014.

\bibitem{Oli2016}
M.~Olive.
\newblock About gordan's algorithm for binary forms.
\newblock {\em Foundations of Computational Mathematics}, 17(6):1407--1466, jun
  2016.

\bibitem{Oli2017a}
M.~Olive.
\newblock About {G}ordan’s algorithm for binary forms.
\newblock {\em Foundations of Computational Mathematics}, 17(6):1407--1466,
  2017.

\bibitem{OA2014b}
M.~Olive and N.~Auffray.
\newblock {I}sotropic invariants of a completely symmetric third-order tensor.
\newblock {\em Journal of Mathematical Physics}, 55(9):092901, 2014.

\bibitem{OD2021}
M.~Olive and R.~Desmorat.
\newblock Effective rationality of second-order symmetric tensor spaces.
\newblock {\em Annali di Matematica Pura ed Applicata (1923 -)}, may 2021.

\bibitem{OKA2017}
M.~Olive, B.~Kolev, and N.~Auffray.
\newblock A minimal integrity basis for the elasticity tensor.
\newblock {\em Archive for Rational Mechanics and Analysis}, 226(1):1--31, Oct.
  2017.

\bibitem{OKDD2018}
M.~Olive, B.~Kolev, B.~Desmorat, and R.~Desmorat.
\newblock Harmonic {F}actorization and {R}econstruction of the {E}lasticity
  {T}ensor.
\newblock {\em J. Elasticity}, 132(1):67--101, 2018.

\bibitem{OKDD2022}
M.~Olive, B.~Kolev, R.~Desmorat, and B.~Desmorat.
\newblock Characterization of the symmetry class of an elasticity tensor using
  polynomial covariants.
\newblock {\em Mathematics and Mechanics of Solids}, 27(1):144--190, 2022.

\bibitem{ODK2021}
C.~Oliver-Leblond, R.~Desmorat, and B.~Kolev.
\newblock Continuous anisotropic damage as a twin modelling of discrete
  bi-dimensional fracture.
\newblock {\em European Journal of Mechanics - A/Solids}, 89:104285, aug 2021.

\bibitem{Olv1999}
P.~J. Olver.
\newblock {\em {C}lassical invariant theory}, volume~44 of {\em London
  Mathematical Society Student Texts}.
\newblock Cambridge University Press, Cambridge, 1999.

\bibitem{Ona1984}
E.~T. Onat.
\newblock Effective properties of elastic materials that contain penny shaped
  voids.
\newblock {\em nt. J. Eng. Sci.}, 22:1013--1021, 1984.

\bibitem{Pie1994/95}
J.~F. Pierce.
\newblock Representations for transversely hemitropic and transversely
  isotropic stress-strain relations.
\newblock {\em J. Elasticity}, 37(3):243--280, 1994/95.

\bibitem{Pol2018}
C.~Polizzotto.
\newblock Anisotropy in strain gradient elasticity: Simplified models with
  different forms of internal length and moduli tensors.
\newblock {\em European Journal of Mechanics - A/Solids}, 71:51--63, Sept.
  2018.

\bibitem{PSM2018}
M.~Poncelet, A.~Somera, C.~Morel, C.~Jailin, and N.~Auffray.
\newblock An experimental evidence of the failure of cauchy elasticity for the
  overall modeling of a non-centro-symmetric lattice under static loading.
\newblock {\em International Journal of Solids and Structures}, 147:223--237,
  2018.

\bibitem{RKM2009}
J.~Rahmoun, D.~Kondo, and O.~Millet.
\newblock A 3d fourth order fabric tensor approach of anisotropy in granular
  media.
\newblock {\em Computational Materials Science}, 46(4):869--880, oct 2009.

\bibitem{RIBD2018}
N.~Ranaivomiarana, F.-X. Irisarri, D.~Bettebghor, and B.~Desmorat.
\newblock Concurrent optimization of material spatial distribution and material
  anisotropy repartition for two-dimensional structures.
\newblock {\em Continuum Mechanics and Thermodynamics}, 31(1):133--146, apr
  2018.

\bibitem{Riv1955}
R.~Rivlin.
\newblock {F}urther {R}emarks on the {S}tress-{D}eformation {R}elation for
  {I}sotropic {M}aterials.
\newblock {\em J. Rational Mech. Anal.}, 4:681--701, 1955.

\bibitem{Riv1997}
R.~S. Rivlin.
\newblock {\em Collected Papers of R.S. Rivlin}.
\newblock Springer, 1997.

\bibitem{RA2019}
G.~Rosi and N.~Auffray.
\newblock Continuum modelling of frequency dependent acoustic beam focussing
  and steering in hexagonal lattices.
\newblock {\em European Journal of Mechanics-A/Solids}, 77:103803, 2019.

\bibitem{RD1999}
D.~Royer and E.~Dieulesaint.
\newblock {\em Elastic waves in solids I: Free and guided propagation}.
\newblock Springer Science \& Business Media, 1999.

\bibitem{Smi1971}
G.~Smith.
\newblock {O}n isotropic functions of symmetric tensors, skew-symmetric tensors
  and vectors.
\newblock {\em Int. J. Eng. Sci.}, 9:899--916, 1971.

\bibitem{Smi1994}
G.~Smith.
\newblock {\em {C}onstitutive {E}quations for {A}nisotropic and {I}sotropic
  {M}aterials}.
\newblock North-Holland, Amsterdam, 1994.

\bibitem{SR1971}
G.~Smith and R.~Rivlin.
\newblock {T}he strain-energy function for anisotropic elastic materials.
\newblock {\em Trans. Amer. Math. Soc.}, 88:175--193, 1971.

\bibitem{Smi1965}
G.~F. Smith.
\newblock On isotropic integrity bases.
\newblock {\em Arch. Rational Mech. Anal.}, 18:282--292, 1965.

\bibitem{SB1997}
G.~F. Smith and G.~Bao.
\newblock {I}sotropic invariants of traceless symmetric tensors of orders three
  and four.
\newblock {\em Int. J. Eng. Sci.}, 35(15):1457--1462, 1997.

\bibitem{Sou1996}
J.-M. Souriau.
\newblock {\em Grammaire de la Nature}.
\newblock 1996.

\bibitem{Spe1970}
A.~Spencer.
\newblock A note on the decomposition of tensors into traceless symmetric
  tensors.
\newblock {\em Int. J. Engng Sci.}, 8:475--481, 1970.

\bibitem{SR1958/1959}
A.~J.~M. Spencer and R.~S. Rivlin.
\newblock Finite integrity bases for five or fewer symmetric {$3\times 3$}\
  matrices.
\newblock {\em Arch. Rational Mech. Anal.}, 2:435--446, 1958/1959.

\bibitem{SR1962}
A.~J.~M. Spencer and R.~S. Rivlin.
\newblock Isotropic integrity bases for vectors and second-order tensors. {I}.
\newblock {\em Arch. Rational Mech. Anal.}, 9:45--63, 1962.

\bibitem{Spr1980}
T.~A. Springer.
\newblock {O}n the invariant theory of {${\rm SU}\sb{2}$}.
\newblock {\em Nederl. Akad. Wetensch. Indag. Math.}, 42(3):339--345, 1980.

\bibitem{Ste1994}
S.~Sternberg.
\newblock {\em {G}roup theory and physics}.
\newblock Cambridge University Press, Cambridge, 1994.

\bibitem{Stu2008}
B.~Sturmfels.
\newblock {\em Algorithms in Invariant Theory}.
\newblock Texts \& Monographs in Symbolic Computation. 2\textsuperscript{nd}
  edition, Springer Wien New-York, 2008.

\bibitem{TNS2001}
D.~Tikhomirov, R.~Niekamp, and E.~Stein.
\newblock On three-dimensional microcrack density distribution.
\newblock {\em Z. Angew. Math. Mech.}, 81:3--16, 2001.

\bibitem{TN2004}
C.~Truesdell and W.~Noll.
\newblock {\em {T}he non-linear field theories of mechanics}.
\newblock Springer-Verlag, Berlin, third edition, 2004.
\newblock Edited and with a preface by Stuart S. Antman.

\bibitem{Van2005}
P.~Vannucci.
\newblock Plane anisotropy by the polar method.
\newblock {\em Meccanica}, 40:437--454, 2005.

\bibitem{Van2007}
P.~Vannucci.
\newblock {T}he polar analysis of a third order piezoelectricity-like plane
  tensor.
\newblock {\em Int. J. Solids Struct.}, 44:7803--7815, 2007.

\bibitem{VV2001}
P.~Vannucci and G.~Verchery.
\newblock Stiffness design of laminates using the polar method.
\newblock {\em International Journal of Solids and Structures}, 38:9281--9894,
  2001.

\bibitem{Ver1982}
G.~Verchery.
\newblock Les invariants des tenseurs d'ordre 4 du type de
  l'{é}lasticit{é}.
\newblock In {\em Mechanical Behavior of Anisotropic Solids / Comportment
  M{é}chanique des Solides Anisotropes}, pages 93--104. Springer
  Netherlands, 1982.

\bibitem{Via1997}
M.~Vianello.
\newblock An integrity basis for plane elasticity tensors.
\newblock {\em Arch. Mech. (Arch. Mech. Stos.)}, 49(1):197--208, 1997.

\bibitem{VVA2013}
A.~Vincenti, P.~Vannucci, and M.~R. Ahmadian.
\newblock {O}ptimization of laminated composites by using genetic algorithm and
  the polar description of plane anisotropy.
\newblock {\em Mech. Adv. Mater. Struc.}, 20:242--255, 2013.

\bibitem{Wey1997}
H.~Weyl.
\newblock {\em {T}he classical groups}.
\newblock Princeton Landmarks in Mathematics. Princeton University Press,
  Princeton, NJ, 1997.
\newblock Their invariants and representations, Fifteenth printing, Princeton
  Paperbacks.

\bibitem{WP1964}
A.~Wineman and A.~Pipkin.
\newblock {M}aterial symmetry restrictions on constitutive equations.
\newblock {\em Arch. Rational Mech. Anal.}, 17:184--214, 1964.

\bibitem{Zhe1994}
Q.-S. Zheng.
\newblock {T}heory of representations for tensor functions - {A} unified
  invariant approach to constitutive equations.
\newblock {\em Appl. Mech. Rev.}, 47:545--587, 1994.

\bibitem{ZB1995b}
Q.-S. Zheng and J.~Betten.
\newblock {O}n the tensor function representation of 2nd-order and 4th-order
  tensors. {P}art {I}.
\newblock {\em ZAMM Z. Angew. Math. Mech.}, 75:269--281, 1995.

\bibitem{ZZDR2001}
W.-N. Zou, Q.-S. Zheng, D.-X. Du, and J.~Rychlewski.
\newblock Orthogonal irreducible decompositions of tensors of high orders.
\newblock {\em Mathematics and Mechanics of Solids}, 6(3):249--267, 2001.

\end{thebibliography}
\end{document}